\documentclass[a4paper,draft,12pt]{article}

\usepackage{amssymb,amstext,amsthm,latexsym}
\usepackage{amsmath}
\usepackage[all]{xy}

\newcommand{\Label}[1]{\label{#1}}
\theoremstyle{definition}
\newtheorem{DEF}{Definition}[section]
\theoremstyle{plain}
\newtheorem{COR}[DEF]{Corollary}
\newtheorem{LEM}[DEF]{Lemma}
\newtheorem{PRP}[DEF]{Proposition}
\newtheorem{THM}[DEF]{Theorem}
\theoremstyle{remark}
\newtheorem{EXA}[DEF]{Example}
\newtheorem{REM}[DEF]{Remark}
\newtheorem{NOR}[DEF]{}
\numberwithin{equation}{section}
\newenvironment{Aufl}{\begin{enumerate}
  
  }{\end{enumerate}}
\newcommand{\Ad}[1]{\emph{Ad #1:}}
\newenvironment{Diagramm}{\begin{displaymath}}{\end{displaymath}}
\newcommand{\Name}[1]{#1}
\newcommand{\FRem}[1]{\empty}
\newcommand{\TEMP}[1]{\empty}
\newcommand{\Ann}[1]{\ensuremath{\mathrm{Ann}\left( #1\right)}}
\newcommand{\Aut}[1]{\ensuremath{\mathrm{Aut}\left( #1\right)}}
\newcommand{\C}{\ensuremath{\mathbb{C}}}
\newcommand{\Cast}{\ensuremath{\mathrm{C}^*}}
\newcommand{\Cont}[2][{}]{\ensuremath{C_{#1}\left( #2\right)}}
\newcommand{\fdg}{\,\colon\;\, }
\newcommand{\Fock}[1]{\ensuremath{\mathcal{F}\left( #1\right)}}
\newcommand{\Hm}[1]{\mathcal{H}({#1})}
\newcommand{\HM}[2][\mathcal{H}]{\ensuremath{{#1}_{#2}}}
\newcommand{\hu}{\ensuremath{\mathsf{h}}}
\newcommand{\id}[1]{\ensuremath{\mathrm{id}_{#1}}}
\newcommand{\indlim}[1]{\text{indlim}\left( #1 \right)}
\newcommand{\K}{\ensuremath{\mathbb{K}}}
\newcommand{\ke}{\ensuremath{\mathsf{k}}}
\newcommand{\KOp}[1]{\ensuremath{\K\left( #1\right)}}
\newcommand{\LOp}[1]{\mathcal{L}\left( #1\right)}
\newcommand{\latF}[1]{\ensuremath{\mathbb{F}\left( #1 \right)}}
\newcommand{\latI}[1]{\ensuremath{\mathbb{I}\left( #1 \right)}}
\newcommand{\latO}[1]{\ensuremath{\mathbb{O}\left( #1 \right)}}
\newcommand{\Mult}[1]{\ensuremath{\mathcal{M}\left( #1\right)}}
\newcommand{\N}{\ensuremath{\mathbb{N}}}
\newcommand{\Norm}[1]{\ensuremath{\mathcal{N}\left( #1\right)}}
\newcommand{\OO}[1]{\ensuremath{\mathcal{O}_{#1}}}
\newcommand{\OOO}[1]{\ensuremath{\mathcal{O}\left( #1\right)}}
\newcommand{\Prim}[1]{\mathrm{Prim}\left(#1\right)}
\newcommand{\Prime}[1]{\mathrm{prime}\left(#1\right)}
\newcommand{\R}{\ensuremath{\mathbb{R}}}
\newcommand{\SP}[2]{\ensuremath{\left<{#1},{#2}\right>}}
\newcommand{\Span}[1]{\mathrm{span}\left( #1\right)}
\newcommand{\T}[1]{\ensuremath{\mathrm{T}_{#1}}}
\newcommand{\TT}{\ensuremath{\mathcal{T}}}
\newcommand{\TTT}[1]{\ensuremath{\TT \left(#1\right)}}
\newcommand{\wcl}[1]{\overline{#1}^{\,\mathrm{w}}}
\newcommand{\Z}{\ensuremath{\mathbb{Z}}}

\begin{document}
\title{The inverse problem for primitive ideal spaces}
\author{Hergen Harnisch and Eberhard Kirchberg}
\date{July 10, 2005}
%
%HEREXXX
% AMS subject class !!!!!! to be included
%
%\subjclass{Primary: 46L35; 
%Secondary: 46L80, 06D, 54D}
%
%%cst-algebras
%%general topology
%%distributive lattices, Hausdorff lattices
%
%
\maketitle
\begin{abstract}
A pure topological characterization 
of primitive ideal spaces of separable nuclear 
\Cast-algebras is given.
We show that a \T{0}-space 
$X$ is a primitive ideal space of a
separable nuclear \Cast-algebra $A$ if and only if $X$ is
point-complete (cf.~Definition~\ref{D:MscTopSpc}),
second countable and there is a
continuous pseudo-open and pseudo-epimorphic map 
(Definition\ \ref{D:pseudo})
from a locally compact Polish space $P$ into $X$.
We use this pseudo-open map to construct a \Name{Hilbert} 
bi-module $\mathcal{H}$
over $\Cont[0]{P,\K}$ such that 
$X$ is isomorphic to the primitive ideal
space of the \Name{Cuntz--Pimsner} algebra 
$\mathcal{O}_{\mathcal{H}}$
generated by $\mathcal{H}$. Moreover, our
$\mathcal{O}_{\mathcal{H}}$ 
is $\mathrm{KK}(X;.,.)$--equivalent to $\Cont[0]{P}$
(with action of $X$ on $\Cont[0]{P}$ given be the natural
map from $\latO{X}$ into
$\latO{P}\cong \latI{\Cont[0]{P}}$, and in the sense of
\cite[sec.~4]{Kir00}).
Our construction becomes almost functorial in 
$X$ if we tensor $\mathcal{O}_{\mathcal{H}}$ with
the \Name{Cuntz} algebra 
$\mathcal{O}_2$.
\end{abstract}

%%%
\tableofcontents
%%%
\section{Introduction and main results}
\FRem{size}
Recall that for every separable \Cast-algebra  
$A$ its space of
primitive ideals $\Prim{A}$ with the Jacobson $\hu\ke$-topology
is a second countable,
locally quasi-compact $\T{0}$-space which is a continuous and
open image of the Polish space $P(A)$ of pure states on $A$. 
There is a well-known lattice iso\-mor\-phism from the lattice 
$\latI{A}$
of closed ideals of $A$ onto the lattice $\latO{\Prim{A}}$
of open subsets of $\Prim{A}$ given by $I\in\latI{A}\mapsto
U_I:= \{ J\in\Prim{A}\fdg I\not\subset J \}$ 
(the support of $I$
in $\Prim{A}$).
The correspondence maps closed linear spans (respectively 
intersections) of
families of closed ideals of $A$ to unions 
(respectively interiors of 
intersections) of the corresponding open subsets of $\Prim{A}$.
We will use this iso\-mor\-phism and its properties
frequently without mentioning it. 

We call a point-complete (Def.\ \ref{D:MscTopSpc}), 
second countable, 
locally quasi-compact
$\T{0}$-space a \emph{Dini space}, because its structure is
determined completely by its Dini functions 
(cf.\ Subsection \ref{ssec:2.1}, 
Section \ref{sec:7} and \cite{DiniEK1}
for our notations on $\T{0}$-spaces).
For example, 
$\Prim{A}$ is point-complete for separable $A$, because
it is the continuous and open image of $P(A)$.
Point-complete second countable $\T{0}$-spaces with
linearly ordered lattice of 
open subsets are examples of Dini spaces.
One has moreover that
every Dini space is the image of a Polish space by an open and
continuous map (cf.\ \cite{DiniEK3}). 
It is an unsolved problem whether 
\emph{all} Dini spaces are primitive ideal
spaces of separable \Cast-algebras or not.  
Here we limit ourselves to the
characterization of the primitive ideal spaces of separable
\emph{nuclear} \Cast-algebras.  
We get a complete description in pure
topological terms with help of a result of \cite{KR03}, 
see below.

We suppose that the reader is
not very familiar
with our topological terminology for 
$\T{0}$-spaces and its maps.
We explain some of our later used topological ideas 
with help of 
open continuous epi\-mor\-phisms 
(instead of the later used pseudo-open 
pseudo-epi\-mor\-phisms).

Consider for example a $\T{0}$-space $X$ that is 
the image of a
Polish space $P$ under a continuous and open map $\pi$  
(as it happens in the case of a Dini space $X$):

One can consider the map $\Psi(U):=\pi^{-1}U$ from 
the lattice $\latO{X}$ of open subsets $U$ of $X$ 
into the lattice of open subsets of $P$.
Then $\Psi$ is a lattice mono\-mor\-phism from
$\latO{X}$ into $\latO{P}$ with $\Psi(X)=P$ and 
$\Psi(\emptyset)=\emptyset$
that preserves l.u.b.\ and g.l.b., 
i.e.~satisfies the following properties (I)--(IV).

\begin{DEF}\Label{D:l-g-preserv}
A map $\Psi$ from the lattice $\latO{X}$ of open subsets of a 
$\T{0}$-space into the lattice 
$\latO{P}$ of a $\T{0}$-space $P$
is a 
\emph{l.u.b.- and g.l.b.-pre\-serving lattice mono\-mor\-phism} 
if it has the following properties (I)--(IV):
\begin{Aufl}
\item[(I)] % $\Psi$ is non-degenerate, i.e.~
  $\Psi^{-1}(P)=\{X\}$ and 
  $\Psi(\emptyset)=\emptyset$.
\item[(II)] 
        
$\Psi \left(\left(\bigcap_\alpha U_\alpha\right)^\circ \right) 
 =\left( \bigcap_\alpha \Psi\left(U_\alpha\right)\right)^\circ$ 
        for every family of open
  subsets $U_\alpha\subset X$. 
        (Here $Z^\circ$ denotes the interior 
  of a subset $Z$ of $X$ respectively $P$.)
\item[(III)] 
        $\Psi\left(\left(\bigcup_\alpha U_\alpha\right)\right) 
  =\bigcup_\alpha \Psi\left(U_\alpha\right)$ 
        for every family of open
  subsets $U_\alpha\subset X$.
\item[(IV)] $\Psi\left(U_1\right)=\Psi\left(U_2\right)\,$ 
        implies $\,U_1=U_2\,\,$ 
  for all $\,\, U_1, U_2\in\latO{X}$.
\end{Aufl}
\end{DEF}
(I)--(IV) means equivalently that 
$\Psi\colon \latO{X}\to \latO{P}$ 
defines a lattice iso\-mor\-phism of
the lattice 
$\latO{X}$ of open subsets of $X$ onto a sub-lattice 
$\mathcal{Z}$ of
$\latO{P}$ that contains
$P$ and $\emptyset$ and is closed under 
l.u.b.\ (unions) and g.l.b.\
(interiors of intersections).

\begin{DEF}\Label{D:regular}
A \Cast-sub\-al\-ge\-bra $C$ of a 
\Cast-algebra $E$ is called 
\emph{regular}
in $E$, if 
%for all closed ideals $I,J$ of $E$ holds
\begin{Aufl}
\item[(i) ] $C$ separates the closed ideals of $E$ 
  (i.e.\ $I\cap C=J\cap C$ implies $I=J$ for every 
$I,J\in\latI{E}$),
\item[(ii)] $(I+J)\cap C=(I\cap C)+(J\cap C)$ 
        holds for every 
  $I,J\in\latI{E}$.
\end{Aufl}
\end{DEF}

If $C\subset E$ is regular in $E$, then 
the map 
$\Psi\colon \latO{\Prim{E}} \to \latO{\Prim{C}}$ 
that corresponds 
to the map $J\mapsto C\cap J$
from $\latI{E}\cong \latO{\Prim{E}}$ into
$\latI{C}\cong\latO{\Prim{C}}$
satisfies (I)--(IV) of Definition \ref{D:l-g-preserv}
(cf.~Lemma \ref{L:reg-abel-psi}).

\smallskip

In the case where $X$ is Hausdorff, (I)--(IV) of
Definition \ref{D:l-g-preserv} conversely
imply that $\pi$ is an open epimorphism. But
since in our case $X$ is only a $\T{0}$-space, the properties 
(I)--(IV) do not imply
that $\pi$ is an \emph{open} map onto $X$ 
(see Example \ref{Ex:[0,1]lsc}).  
In fact, in the case of a point-complete
$\T{0}$-space $X$ one has that for maps 
$\Psi\colon \latO{X}\to\latO{P}$  with properties (I)--(IV)
there is a unique map 
$\pi\colon P\to X$ with 
$\Psi(U)=\pi^{-1}U$ for open $U\subset X$
and this (continuous) map $\pi$ is
pseudo-open and pseudo-epimorphic in the
sense of the following definition 
(cf.\ Proposition \ref{P:map-Psi-pi}).
%
%%%%%%%%%%%%%%%%%%
\begin{DEF}\Label{D:pseudo}
Suppose that $X$ and $P$ are $\T{0}$-spaces.  We define the
\emph{pseudo-graph} $R_\pi$ of a continuous map 
$\pi\colon P\to X$ by
\begin{equation}
R_\pi:= \left\{ (p,q)\in P\times P\fdg 
\pi(q)\in \overline{ \{ \pi(p) \} }\right\}.
\end{equation}
A subset $Z$ of $P$ is $R_\pi$-invariant\footnote{
$R_\pi$ is a partial order on $P$. Being $R_\pi$-invariant
is the same as ``downward closed'' or ``lower invariant''
in lattice-theoretic sense 
(where $p\le_R q\;\Longleftrightarrow \;(p,q)\in R$).}
if $q\in Z$ and
$(p,q)\in R_\pi$ imply $p\in Z$.
We call the map $\pi\colon P\to X$
\begin{Aufl}
\item \emph{pseudo-open} if the natural map 
        $(p,q)\in R_\pi \mapsto p\in P$ 
  is an open map from the space $R_\pi$ into $P$, and
  the image $\pi(V)$ of every $R_\pi$-invariant 
  open subset $V$ of $P$ is an open subset of $\pi(P)$,
\item \emph{pseudo-epimorphic} 
  if $\pi(P)\cap F$ is dense in $F$ for every closed subset 
        $F$ of $X\,$
  (i.e.\ if $U\setminus V$ contains a point of
  $\pi(P)$ for every pair $V\subset U$ of open subsets of 
        $X$ with
  $V\neq U$).
\end{Aufl}
\end{DEF}
$R_\pi$ is the ordinary
equivalence relation on $P$ defined by $\pi$ if $X$ is 
a $\T{1}$-space. 
Note that pseudo-open (respectively pseudo-epimorphic) maps 
$\pi\colon P\to X$ are open (respectively epimorphic) 
if $X$ is
a $\T{1}$-space. 
In general, $R_\pi$ is not a closed subset of $P\times P$
(even if $P$ is Polish).

\medskip

We are now in position to state our main result:
\begin{THM}\Label{T:main}
Suppose that $X$ is a \emph{point-complete} 
\T{0}-space which has a faithful map $\Psi$ from $\latO{X}$
into the open sets $\latO{P}$ of a 
locally compact Polish space $P$ 
with the properties (I)--(IV) of 
Definition \ref{D:l-g-preserv}.
Then there exists a separable inductive limit $E$
of type~I \Cast-algebras and an auto\-mor\-phism 
$\sigma$ of $E$ such that $X$ is homeomorphic
to 
the primitive ideal space of the 
\Cast-algebra crossed product
$E\rtimes_\sigma\Z$.
Moreover, $E$ and $\sigma$ can be chosen such that:
\begin{Aufl}
\item $E\rtimes_\sigma\Z$ is isomorphic to a 
\Name{Cuntz--Pimsner} algebra 
$\OOO{\Hm{A,h}}$ of a 
\Name{Hilbert} $A$-bi-module $\Hm{A,h}$ given
by 
a non-degenerate *-homo\-mor\-phism  
$h\colon A\to \Mult{A}$
with $h(A)\cap A=\{ 0\}$
and  $A\cong \Cont[0]{P}\otimes\K$.
\item The closed ideals of $E \rtimes_\sigma \Z$ are in 
1-1-correspondence with the 
$\sigma$-invariant closed ideals of $E$.
\item 
$A\cong \Cont[0]{P}\otimes\K$ is a 
regular \Cast-sub\-al\-ge\-bra of 
$E\rtimes_\sigma\Z$, is contained in $E$, and $E$ is the
smallest $\sigma$-invariant \Cast-subalgebra of $E$
containing $A$.
The map $\Psi$ is induced by
\[ J\in\latI{E\rtimes_\sigma\Z}\mapsto 
J\cap A\in\latI{A}\cong\latO{P} \]
via the
homeomorphism of $\Prim{E\rtimes_\sigma\Z}$ with $X$. 
\end{Aufl}
\end{THM}

%
%\emph{If $X$ is point-complete (Def.\ \ref{D:MscTopSpc}), 
%second countable and if there exists a
%continuous pseudo-open and pseudo-epimorphic 
%map $\pi$ from
%a locally compact Polish space $P$ into $X$, 
%then $X$ is isomorphic to the
%primitive ideal space of a \Cast-algebra crossed product 
%$A_X:=E\rtimes_\sigma \Z$ for a \Cast-algebra
%$E$ which is an inductive limit of 
%type~I sub\-al\-ge\-bras of 
%$E$ and $\sigma\in\Aut{E}$.}
%%%HEREXXX check $E$ good?

%
\FRem{''$E$'' OK?}
%

%Our \Cast-algebra $A_X$ contains a copy $C$ of 
%$\Cont[0]{P}$ as a regular 
%\Cast-sub\-al\-ge\-bra of $A_X\,$ 
%(cf.\ Remark \ref{R:zuTh-main}).
A full hereditary \Cast-sub\-al\-ge\-bra $C$ of $E$ is 
always regular in $E$, because $J\mapsto C\cap J$
is a lattice isomorphism from $\latI{E}$ onto $\latI{C}$.
Thus, by Theorem \ref{T:main}(iii), $\Cont[0]{P}$ is
isomorphic to a regular Abelian \Cast-subalgebra of
$E\rtimes\Z\cong A_X$.
In \cite{KR03} it has been shown that for every 
separable \emph{nuclear}
\Cast-algebra $B$ there is an Abelian regular
\Cast-sub\-al\-ge\-bra $C$
of $B\otimes\OO{2}$ with (at most one-dimensional) maximal
ideal space $P:=\Prim{C}$, (cf.~\cite[thm.~6.11]{KR03}). 
%%
%{\it A point-complete
%$\T{0}$-space $X$ is isomorphic to the primitive ideal space
%of a separable nuclear \Cast-algebra $B$ if and only if 
%there exists
%a locally compact Polish space $P$ and 
%a continuous pseudo-open and pseudo-epimorphic
%map $\pi\colon P\to X$.}

Note that
$\Prim{B}\cong\Prim{B\otimes\OO{2}}\cong
\Prim{B\otimes\OO{2}\otimes\K}$.
Thus, in conjunction with our solution of the inverse problem,
Theorem \ref{T:main}(iii), and by
Proposition \ref{P:map-Psi-pi}
the following characterization of primitive ideal
spaces of separable nuclear \Cast-algebras:

\begin{COR} \Label{C:char-prim-nuc}
A $\T{0}$ space 
$X$ is isomorphic to the primitive ideal space of a
separable nuclear \Cast-algebra, if and only if, 
\begin{Aufl}
\item $X$ is point-complete (cf.~Definition 
\ref{D:MscTopSpc})\footnote{``point-complete'' 
is called
``spectral'' in \cite{HoffKeim} and
``sober'' by others}
%\item $X$ is second countable,
\item there exists a locally compact Polish space $P$ 
        and a pseudo-open and 
  pseudo-epimorphic continuous map 
        $\pi\colon P\to X$ in the
  sense of Definition \ref{D:pseudo}.
\end{Aufl}
\end{COR}

By a result of the second named author every 
iso\-mor\-phism $\kappa$ from the 
primitive ideal space $\Prim{B}$ 
onto the primitive ideal space $\Prim{A}$ 
of separable nuclear \Cast-algebras $A$ and $B$ 
can be realized by an iso\-mor\-phism $\varphi$ from 
$A\otimes\OO{2}\otimes \K$ 
onto $B\otimes\OO{2}\otimes\K$ with $\varphi(\kappa(J))=J$
for $J\in \Prim{B\otimes\OO{2}\otimes\K}\cong \Prim{B}$.
Moreover, $\varphi$ with this property is uniquely 
determined up to
unitary homotopy by unitaries in the multiplier algebra of 
$B\otimes\OO{2}\otimes\K$ (cf.\ 
Definition \ref{D:unitary-htpy} and \cite[cor.~L]{K.book}).
In particular, 
{\it there is up to isomorphisms exactly one separable
stable nuclear \Cast-algebra $A$ with $\Prim{A}\cong X$ and
$A\otimes \OO{2}\cong A$}.

\smallskip

More generally:  Suppose that $A$ and $B$ are stable
separable \Cast-algebras, where $A$ is exact and
$B$ is strongly purely infinite (in the sense of
\cite[def.\ 5.1]{KirRor2}, e.g.\ that 
$B\otimes\OO{\infty}\cong B$
or even $B\otimes\OO{2}\cong B$).
Then
every map $\Psi$ from $\latO{\Prim{B}}$ to $\latO{\Prim{A}}$
that satisfies the properties 
\begin{Aufl}
\item[(I)  ] %$\Psi$ is non-degenerate, i.e.
  $\Psi^{-1}(\Prim{A})=\{\Prim{B}\}$ and 
  $\Psi(\emptyset)=\emptyset$,
\item[(II) ] 
$
\Psi\left( (\bigcap_\alpha U_\alpha )^\circ \right) 
  = 
  \left(\bigcap_\alpha \Psi\left(U_\alpha\right)\right)^\circ
$ 
  for every family $\{U_\alpha\}$ of open
  subsets of $\Prim{B}$, 
  %(Here $Z^\circ$ denotes the interior 
  %of a subset $Z$ of $\Prim{B}$ respectively $\Prim{A}$.)
\item[(III$_0$)] 
        $\Psi \left(\left(\bigcup_\alpha U_\alpha \right)\right) 
  =\bigcup_\alpha \Psi\left(U_\alpha\right)$ for every 
  \emph{upward directed} net of open
  subsets $U_\alpha\subset \Prim{B}$,
\end{Aufl}
can be realized by a \emph{non-degenerate} 
nuclear *-mono\-mor\-phism $h$ from 
$A\otimes\OO{2}\otimes\K$ into $B\otimes\OO{2}\otimes\K$
with $\Psi(J)=h^{-1}(h(A)\cap J)$ for 
$J\in \latI{B\otimes\OO{2}\otimes \K}
\cong \latO{\Prim{B}}$, and  $h$
is uniquely determined by this property up to unitary homotopy
(cf.~\cite[thm.\ K]{K.book} and 
Definition \ref{D:unitary-htpy}).
(Note here that (III$_0$) does not imply 
$\Psi(U)\cup \Psi(V)=\Psi(U\cup V)$.)

This implies that our construction 
of a separable nuclear algebra $A$
with given primitive ideal space $X$ is 
(up to unitary homotopy) 
a contravariant
functor on
the maps $\Psi\colon \latO{\Prim{B}}\to\latO{\Prim{A}}$
with (I), (II) and
(III$_0$) if we tensor with $\OO{2}\otimes \K$.

\bigskip

Theorem \ref{T:main} also implies  that 
a separable \Cast-algebra $B$ has a primitive ideal space 
$X=\Prim{B}$ which is the continuous pseudo-open and 
pseudo-epimorphic image 
of a locally compact Polish space $P$ 
if and only if $B\otimes\OO{2}$
contains a regular Abelian 
\Cast-sub\-al\-ge\-bra in the sense of 
Definition \ref{D:regular} (see Corollary 
\ref{C:exist.regular}).

\smallskip

%%HEREXXX begin 17.7.05
Since $h(A)\cap A=\{0\}$ for our Hilbert bi-module
$\Hm{A,h}$, we get from \cite[cor.~3.14]{Pi97}
that $\mathcal{T}(\Hm{A,h})\cong \OOO{\Hm{A,h}}$ 
and then (by \cite[thm.~4.3]{Pi97}) that the natural
embedding $A\hookrightarrow\OOO{\Hm{A,h}}$ induces
a KK-equivalence between $\OOO{\Hm{A,h}}$ and $A$.
Thus $E\rtimes_\sigma\Z\cong\OOO{\Hm{A,h}}$ is
KK-equivalent to $\Cont[0]{P}$, which may give
sometimes 
different nuclear \Cast-algebras that have the same primitive
ideal space $X$ but have not the same K-groups.
The natural embedding from $A$ into $B:= \OOO{\Hm{A,h}}$ 
is  $\Psi$-equivariant with respect  
to the natural actions $\Psi_A$ and $\Psi_B$ of
$\latO{X}\cong \latI{B}$ on $A$ and $B$, because 
it is compatible with the natural action of $\latO{X}$ on
$\Hm{A,h}$, $\LOp{\Hm{A,h}}$, 
$\Fock{\Hm{A,h}}$ and $\TTT{\Hm{A,h}}\cong B$.

Therefore
it defines also an element of the group $\mathrm{KK}(X; A,B)$. 
A closer look to the proof of \cite[thm.~4.4]{Pi97} 
shows that the Kasparov $B$-$A$-module 
$(\Fock{\Hm{A,h}}\oplus \Fock{\Hm{A,h}}, \pi_0\oplus \pi_1,T)$
of \cite[def.~4.3]{Pi97}
(which defines the KK-inverse 
$\beta\in \mathrm{KK}(B,A)$ of $[A\hookrightarrow B]$) 
carries in a natural way an action of $\latO{X}$ 
such that the $B$-$A$-module becomes $\Psi$-equivariant
and defines therefore
an element of $\mathrm{KK}(X;B,A)$. The homotopy constructed in
in the proof of \cite[thm.~4.4]{Pi97} 
turns out to be also $\Psi$-equivariant
and can be used to prove that \textit{the natural inclusion 
$A\hookrightarrow \OOO{\Hm{A,h}}$ defines a 
$\textrm{KK}(X;.,.)$--equivalence between $A$ and
$\OOO{\Hm{A,h}}$.}

We can replace $P$ by $\R_+\times P$ in the proof of 
Theorem \ref{T:main}.
Then the $E\rtimes_\sigma\Z$ becomes KK-trivial. It becomes
even $\mathrm{KK}(X;.,.)$-trivial, and thus absorbs
$\OO{2}$ tensorially by \cite[chp.~1,cor.~N]{K.book} and the following 
observation.

The *-monomorphism $h\colon A\to \Mult{A}$ in our special
construction of $\Hm{A,h}$ (from $\Psi_A$
and  $A=\Cont[0]{P}\otimes\K$) is unitarily equivalent  
to its infinite repeat $\delta_\infty\circ h$. 
Then
$\Hm{A,h}$ is isomorphic to 
$\Hm{A,h}\otimes \Hm{\K,\delta_\infty}$.
One can show that the latter
implies that $\OOO{\Hm{A,h}}$ is isomorphic
to $\OOO{\Hm{A,h}}\otimes \OO{\infty}$. In particular,
our algebras $E\rtimes \Z$ in Theorem \ref{T:main} 
are strongly purely infinite.\footnote{Details on
strong pure infiniteness of $\OOO{\Hm{A,h}}$
and on the $\mathrm{KK}(X;.,.)$--equivalence of
$A$ and $\OOO{\Hm{A,h}}$ will be published elsewhere.}

In general $\OOO{\Hm{A,h}}$ need \emph{not} be purely infinite
if one works with the weaker assumptions on $\Hm{A,h}$
of Corollaries 
\ref{C:char-[D]sigma-by-h}--\ref{C:O(H(A,h))-simple},
as R{\o}rdam's example of a stably infinite nuclear separable
stable simple \Cast-algebra shows  (which is
KK-equivalent to $C(S^2\times S^2\times\cdots)$ by \cite{Pi97}
and is \emph{not} purely infinite, cf.\ \cite{Ror}).

We recall in the Appendix some needed basic facts on
$\T{0}$-spaces, pseudo-open maps, Hilbert \Cast-modules
and crossed products by $\Z$.  We give there elementary
proofs for the needed results.
%%%HEREXXX end 17.7.05

%%%
\section{Realization of ideal-lattice mor\-phisms}
\subsection{\Name{Hilbert} bi-modules and cones of c.p.~maps}
\begin{DEF}\Label{D:matr-conv-cone}
Suppose that $A$ and $B$ are \Cast-algebras, and
let $CP(A,B)$ denote the cone of completely
positive maps from $A$ into $B$.
A subset $\mathcal{C} $ of $CP(A,B)$ is an
\emph{operator convex cone} of c.p.~maps
if $\mathcal{C}$ has the following 
properties (i)--(iii).
\begin{Aufl}
\item $\mathcal{C}$ is a cone.
\item If $V \in \mathcal{C}$ and $b \in B$, then
  the map $a \mapsto b^*V(a)b$ belongs to $\mathcal{C}$.
\item 
        $V_{r,\,c}\colon 
        a\in A\mapsto c^*V\otimes \id{n}(r^*ar)c$
        is in $\mathcal{C} $ for every $V\in \mathcal{C} $, 
        every row-matrix $r\in M_{1,n}(A)$ 
        and every column-matrix $c\in M_{n,1}(B)$.
\end{Aufl}
$\mathcal{C}$ is \emph{full} if the linear span of 
$\{ V(a)\fdg V\in \mathcal{C}, a\in A \}$
is dense in $B$, and $\mathcal{C}$ is \emph{separating}
if $V(a)=0$ for all $V\in \mathcal{C}$ implies $a=0$.

It is easy to see that the point-norm closure of 
an operator convex cone is again
operator convex.

Let $S$ be a subset of $CP(A,B)$. We denote by
$K(S)$ the smallest subcone of $CP(A,B)$
which is invariant under the operations
in (ii) and (iii), and by $\mathcal{C}(S)$ the point-norm
closure of $K(S)$
(i.e.\  the closure of $K(S)$ in $\LOp{A,B}$ 
w.r.t.~the strong operator
topology).  Then $K(S)$ and $\mathcal{C}(S)$
are operator convex cones of completely
positive maps. We say that $S\subset \mathcal{C}$ 
\emph{generates} the 
(point-norm closed) operator convex cone $\mathcal{C}$
if $K(S)$ is dense in $\mathcal{C}$.
$\mathcal{C}$ is \emph{countably generated}
if a countable subset $S$ of $\mathcal{C}$ generates 
$\mathcal{C}$.
\end{DEF}
Note that
$$
V_{r,\,c}(a)= 
\sum_{i=1}^n \sum_{j=1}^n c_i^*V(r_i^*ar_j)c_j\,
$$
for
$(r_1,\ldots,r_n)=r\in M_{1,n}(A)$
and $(c_1,\ldots, c_n)=c^t\in M_{1,n}(B)$.
\begin{REM}\Label{R:cone-from-bi-module} 
If $E$ is a \Name{Hilbert} $B$-module
and $h\colon A\to \LOp{E}$ a *-homo\-mor\-phism,
then the \Name{Hilbert} $B$-module sum
$E_\infty :=E\oplus E\oplus \cdots\,$ with right 
$A$-module
structure given by 
$h_\infty\colon A \to \LOp{E_\infty}\,$, 
$h_\infty(a):=h(a)\oplus h(a)\oplus \cdots\,$,
has the property that  the set $C$ of c.p.~maps 
$V_e\colon A\to B$ with $V_e(a):=\SP{e}{h_\infty(a)(e)}$
is an operator convex cone.

\end{REM}

\begin{LEM}\Label{L:bi-module-from-cone}
Suppose that $\mathcal{C}\subset CP(A,B)$ is a 
point-norm closed operator convex cone.
Then there is a \Name{Hilbert} $B$-module
$E$ and a non-degenerate *-homomorphism 
$h\colon A\to \LOp{E}$ such that
$\mathcal{C}$ is the point-norm closure of
the maps $V_e\colon A\to B$ for
$e\in E$, where $V_e(a):= \SP{e}{h(a)e}$.

If $A$ is separable and $\mathcal{C}$ 
is countably generated,
one can manage that, in addition, 
$E$ is countably generated as \Name{Hilbert}
$B$-module.
\end{LEM}

\begin{proof} Let $\mathcal{S}$ a subset of $\mathcal{C}$ 
that generates $\mathcal{C}$, and let $E^T$ and
$h^T\colon A\to \LOp{E^T}$ as in
Lemma \ref{L:bi-mod-from-cp-map} for $T\in \mathcal{S}$.
Define $F$ as the \Name{Hilbert} $B$-module sum
$$F:= \bigoplus _{T\in \mathcal{S}} E^T$$
with right $A$-module structure given by the 
non-degenerate *-homomorphism
$$g\colon a\in A\mapsto 
\bigoplus _{T\in \mathcal{S}} h^T(a).$$
The point-norm closure of the set
of c.p.~maps $V_e(a):=\SP{e}{g(a)e}$ 
contains $\mathcal{S}$, and is contained in $\mathcal{C}$
by Lemma \ref{L:bi-mod-from-cp-map} and Definition
\ref{D:matr-conv-cone}.
Thus, by Remark \ref{R:cone-from-bi-module},
$E:=F_\infty =F\oplus F\oplus \cdots\,$
and the non-degenerate *-homomorphism
$h(a):=g_\infty(a)=g(a)\oplus g(a)\oplus\cdots\,$
are as desired.

If $A$ is separable and $\mathcal{S}$ is countable,
then $E$ is countably generated as a \Name{Hilbert}
$B$-module, because then $E^T$ is countably
generated for $T\in \mathcal{S}$ by Lemma 
\ref{L:bi-mod-from-cp-map}.
\end{proof}
%
%%%%%%%%%%%%
\begin{REM}\Label{R:delta-infty}
A  \Cast-algebra $B$ is stable  if and only if
there is a sequence of isometries 
$s_1,s_2,\ldots \in \Mult{B}$ such 
that $\sum s_ns_n^*$ converges strictly to $1$. 
(Then automatically  $s_j^*s_k=\delta_{j,k}1$.) 

The \emph{infinite repeat} 
is (up to unitary equivalence) 
the unital endo\-mor\-phism given by
$$
\delta_\infty\colon d\in\Mult{B}\mapsto
\sum s_ids_i^*\in\Mult{B}.
$$ 
It is not hard to check that
\begin{Aufl} 
\item if $t_1,t_2,\dots$ 
        is a second sequence of isometries with
  $\sum t_it_i^*$ converging strictly to $1$, then 
  $\delta_\infty$ and 
        $\delta_\infty'\colon d\mapsto\sum t_idt_i^*$
  are unitarily equivalent
  by the unitary $U:=\sum s_nt_n^*\in\Mult{B}$,
\item $\delta_\infty\circ\delta_\infty$ is 
        unitarily equivalent to $\delta_\infty$.
\item Also note that $\delta_\infty(\Mult{B})\cap B=\{0\}$,
  because $\delta_\infty(b)\in B$ implies
  $b=s_n^*\delta_\infty(b)s_n\to 0$ for
  $n\to\infty$.
%\item 
%$\delta_\infty(\Mult{B})'\cap \Mult{B}$ 
%contains a sequence of
%isometries $r_1,r_2,\ldots$ such that 
%$\sum_n r_nr_n^*$ converges
%strictly to $1$, and
%\item 
%$\delta_\infty(\Mult{B})\cap (\Mult{B,J}+B)= 
%\delta_\infty(\Mult{B})\cap \Mult{B,J}$.
\end{Aufl}
Cf.\ \cite[rem.~5.1.2,~lem.~5.1.3]{K.book} for details.
\end{REM}

\begin{REM}\Label{R:A-cong-HA} 
$\HM{A}$ and $A$ are isomorphic as
\Name{Hilbert} $A$-modules if and only if $A$ is stable.\\
Indeed: 
%\begin{proof}
$\LOp{\HM{A}}\cong \Mult{\K \left(\HM{A}\right)}$ 
always contains a sequence of isometries 
$s_1,s_2,\ldots$
such that $\sum _n s_ns_n^*$ strictly converges to
$1$ in $\Mult{\K\left(\HM{A}\right)}$, because 
$\HM{A}\cong \HM{A}\oplus \HM{A}\oplus \cdots$.
An \Name{Hilbert} $A$-module iso\-mor\-phism from  
$\HM{A}$ 
onto $A$ defines an iso\-mor\-phism from 
$\K\left(\HM{A}\right)$ 
onto $\K \left(A_A\right)\cong A$.

By Remark \ref{R:delta-infty}, $A$ is stable
if and only if there is a sequence of isometries 
$s_1,s_2,\ldots\in \Mult{A}$ with $\sum s_ns_n^*$ 
strictly convergent
to $1$.  

The maps
$a\in A\mapsto \left(s_1^*a,s_2^*a, \ldots\right)$
and $(a_1,a_2,\ldots)\in \HM{A}\mapsto \sum_n s_na_n$
preserve the $A$-valued sesqui\-linear forms, 
are right $A$-module maps, and are inverse
to each other.
%\end{proof}
\end{REM}

\begin{REM}\Label{R:H-mod-for-stable+sep}
For a separable and stable \Cast-algebra $B$ every 
countably generated
\Name{Hilbert} B-module $E$ is isomorphic 
to $pB$ for a projection $p\in\Mult{B}$ with
$B$-valued product $\SP{a}{b}:=a^*b$.

Thus: If $A$ is separable, then every
non-degenerate left $A$-module structure on
$E$ is given (up to isometric
bi-module isomorphisms)
by a non-degenerate *-homo\-mor\-phism
$h\colon A\to \Mult{pBp}\cong p\Mult{B}p$ and $E:=pB$.
(Use Remark \ref{R:A-cong-HA} and Examples 
\ref{Ex:H-module}(iv) and \ref{Ex:M(A)}.)
%
%(cf.~\cite[lem.~1.3.2]{JT91}).
\end{REM}

\begin{REM}\Label{R:full-corner}
Suppose that $B$ is stable and $\sigma$-unital and that
$D$ is a corner of $B$ (i.e.\ there is a projection 
$p=p^*\in \Mult{B}$ with $pBp=D$).
Then:\\ 
{\it $D$ is stable and full (i.e.\ the span of $BDB$ is
dense in $B$), if and only if, there is an isometry
$v\in \Mult{B}$ with $vv^*=p$.}

Indeed, by \cite[Lemma 2.5]{Brown} 
there are 
the Murray--von-Neumann equivalences 
$1\otimes e_{1,1}\sim 1\otimes 1\sim p\otimes 1\sim 
p\otimes e_{1,1}$ in $\Mult{B\otimes \K}$.
Note here that
$pbp$ is a strictly positive element of $pBp$
if $b\in B_+$ is a strictly positive element
of $B$.
%If $B$ and $D$ are stable and contain
%strictly positive elements, and if $D=pBp$ is full
%in $B$, then
%(by \cite{Brown}) one finds an element 
%$z\in B$ such that
%$z^*z$ is a strictly positive element of $B$ and $zz^*$ is
%a strictly positive element of $D$. Let 
%$z=v(z^*z)^{1/2}=(zz^*)^{1/2}v$ the
%polar decomposition of $z$ in the second conjugate W*-algebra
%$B^{**}$ of $B$.
%It is easy to check 
%that $v$ is an isometry
%in $\Mult{B}\subset B^{**}$ with $vv^*=p$.
\end{REM}

\begin{PRP}\Label{P:h-from-cone}
Suppose that $A$ and $B$ are stable and separable
\Cast-al\-ge\-bras, and that
$\mathcal{C}$ is a point-norm closed
full operator convex cone of completely
positive maps from $A$ into $B$ in the sense
of Definition \ref{D:matr-conv-cone}. 

Then there
is a non-degenerate *-homo\-mor\-phism
$h\colon A\to \Mult{B}$ with:
\begin{Aufl}
\item $h$ is unitarily equivalent to 
$\delta_\infty \circ h$,
\item the c.p.~maps 
$V_b\colon a\mapsto b^*h(a)b$ are in 
$\mathcal{C}$ for every $b\in B$, and
\item for every $V\in \mathcal{C}$ there is a sequence
$b_n\in B$ such that $\| b_n\| ^2\leq \| V\|$
and $\lim _n b_n^*h(a)b_n=V(a)$ for all $a\in A$.
\end{Aufl}
\end{PRP}

\begin{proof} $\mathcal{C}$ is separable in the 
point-norm topology, because $A$ and $B$ are separable.
Thus $\mathcal{C}$ is countably generated.
By Lemma \ref{L:bi-module-from-cone}, 
there is a countably generated
\Name{Hilbert} $B$-module $E$ and a
*-homo\-mor\-phism $h_1 \colon A \to \LOp{E}$
such that $\mathcal{C}$ is the point-norm
closure of the set of maps $V_e$ with $e\in E$,
and $h_1(A)E$ is dense in $E$.
 
By Remark \ref{R:H-mod-for-stable+sep}, there 
are a projection $p\in \Mult{B}$
and an isometric $B$-module isomorphism $\iota$
from $E$ onto $pB$.
Then $\iota(e)^*\iota(f)=\SP{e}{f}$ and
 $h_2(a):=\iota\circ h_1(a)\circ \iota ^{-1}$
defines a *-homo\-mor\-phism
from $A$ into $\LOp{pB}\cong \Mult{pBp}\cong p\Mult{B}p$
with $h_2(A)pBp=pBp$. 
Thus, $h_2$ uniquely extends to a unital and
strictly continuous *-homo\-mor\-phism $\Mult{h_2}$
from $\Mult{A}$ onto $\Mult{pBp}$.
Since $A$ is stable, there is a sequence 
$t_1,t_2,\ldots$
of isometries in $\Mult{A}$ such that $\sum _n t_nt_n^*$
converges strictly to $1$. The same happens with the
isometries $s_n:=\Mult{h_2}$ in $\Mult{pBp}$.
It follows that
$pBp$ is a stable corner of $B$.

The fullness of $\mathcal{C}$ and the
properties of $E$ and $h_1$ imply that 
every $c\in B$ is in the closed linear span of
the elements $b^*ph_2(a)pb$ with 
$a\in A$ and $b\in B$.
It yields that $pBp$ is also a full corner of $B$.

By Remark \ref{R:full-corner} there is an isometry
$v\in \Mult{B}$ with $vv^*=p$.

Let $h_3(a):=v^*h_2(a)v$, then
$h_3\colon A\to \Mult{B}$ is a non-degenerate
*-homo\-mor\-phism
and satisfies (ii) and (iii) 
(with $h_3$ in place of $h$).

By Remark \ref{R:A-cong-HA},
the bi-module given by $\HM[B]{B}$ and
$\delta_\infty\circ h_3 \colon A\to \Mult{B}$ 
is isomorphic  to the bi-module given by $\HM{B}$
and
$(h_3)_\infty\colon a\mapsto
h_3(a)\oplus h_3(a)\oplus \cdots\, \in \LOp{\HM{B}}$.
Thus 
$h:= \delta_\infty\circ h_3\colon A\to \Mult{B}$ 
still satisfies (ii) and (iii) by 
Remark \ref{R:cone-from-bi-module}.

Since $\delta_\infty\circ \delta_\infty$ 
is unitarily equivalent to
$\delta_\infty$ (cf.\ Remark \ref{R:delta-infty}), 
we get that $h:=\delta_\infty\circ h_3$ satisfies
(i)--(iii).
\end{proof}

\begin{DEF}\Label{D:Mult}
$\Mult{B}$ 
denotes the multiplier 
\Cast-algebra of a \Cast-algebra $B$.
For any closed ideal $J$ of  $B$,
$\Mult{B,J}$ means the set 
$$\Mult{B,J}:= 
\{ t\in\Mult{B}\fdg tB\subset J \} \subset\Mult{B}$$
of relative multipliers which multiply $B$ into $J$.
\end{DEF}
\FRem{cf.\ last chap./ next needed? }
Observe that $\Mult{B,J}$ 
is a strictly closed ideal of $\Mult{B}$
and that $J=\Mult{B,J}\cdot B=\Mult{B,J}\cap B$.
Moreover $\Mult{B,J}$ is the kernel of the natural 
strictly continuous
*-homo\-mor\-phism $\Mult{\pi_J}$
from $\Mult{B}$ into $\Mult{B/J}$.

\begin{REM}\Label{R:Psi-from-cp-maps}
The *-homomorphism $h\colon A\to \Mult{B}$
of Proposition \ref{P:h-from-cone}
defines a map $\Psi _h\colon \latI{B}\to \latI{A}$
by $\Psi _h(J):=h^{-1}(h(A)\cap \Mult{B,J})$.
Then $\Psi _h (B) =A$ and $\Psi _h(\{ 0\})=\ker(h)$,
and,
by (i) and (ii), 
$$
\Psi_h(J)_+ =
\{ 
a\in A_+\,\colon \,\, 
V(a)\in J \quad \textrm{for all}\;\, V\in \mathcal{C}
\}
.$$
It follows that 
$\Psi_h(\bigcap_\alpha J_\alpha)=
\bigcap_\alpha \Psi_h(J_\alpha)$
for every family $\{ J_\alpha \}$ of closed
ideals of $B$. I.e.~$\Psi _h$ satisfies property 
(II) of Definition \ref{D:l-g-preserv} 
if translated to $\latO{\Prim{B}}\cong \latI{B}$
and $\latO{\Prim{A}}\cong \latI{A}$.
Further, $h$ is faithful if and only if
$\mathcal{C}$ is separating.
\end{REM}

\begin{DEF}\Label{D:unitary-htpy}
Let $h_j\colon A\to\Mult{B}$, $j=1,2$, 
*-homo\-mor\-phisms from $A$
into the multiplier algebra $\Mult{B}$ of $B$.
\FRem{def. unitary equiv. needed? right place?}
%$h_1$ and $h_2$ are 
%\emph{unitarily equivalent} if there is a unitary
%$U\in\Mult{B}$ such that $h_1(a)=U^* h_2(a) U$ 
%for all $a\in A$.

We call $h_1$ and $h_2$ \emph{unitarily homotopic} 
if there is a 
norm-continuous map 
$t\mapsto U(t)$ from the non-negative real numbers
$\R_+$ into the unitaries in $\Mult{B}$, such that, for
$t\in\R_+$ and $a\in A$,
\begin{eqnarray*}
 &&U(t)^*h_1(a)U(t)-h_2(a)\in B\\
\text{and} &&
 h_2(a)=\lim_{t\to\infty} U(t)^*h_1(a)U(t).
\end{eqnarray*}
The limit is taken in the norm of $\Mult{B}$.
\end{DEF}

\begin{REM}\Label{R:unitary-homotopy}
$h\colon A\to \Mult{B}$ with (i)--(iii) of
Proposition \ref{P:h-from-cone} is unique
up to unitary homotopy 
(cf.\ Definition \ref{D:unitary-htpy} and 
\cite[cor.~5.1.6]{K.book}).
\end{REM}

\subsection{Construction of 
%$\Hm{A,h}$ 
$A$ and $h\colon A\to \Mult{A}$
from $\Psi$}
\begin{LEM}\Label{L:Brown}
Suppose that $D$ is a 
$\sigma$-unital hereditary \Cast-sub\-al\-ge\-bras
of a stable $\sigma$-unital \Cast-algebra $B$ such that 
$D$ is  full in $B$ (i.e.~
$\Span{BDB}$ is dense in $B$).
Then 
%$I\mapsto D\cap I$ defines a lattice iso\-mor\-phism from 
%$\latI{B}$ onto $\latI{D}$, and 
there
is an iso\-mor\-phism $\varphi$ from $D\otimes \K$ onto $B$
such that $\varphi \left((D\cap I)\otimes \K\right)=I$
for all $I\in \latI{B}$.

If, in addition, $D$ is stable, then there is an 
iso\-mor\-phism
$\psi$ from $D$ onto $B$ with $\psi(D\cap I)=I$ for all 
$I\in \latI{B}$.
\end{LEM}
(Special case of the
$\Psi$-equivariant version 
of the stable iso\-mor\-phism theorem of Brown \cite{Brown},
cf.~\cite[cor.~5.2.6]{K.book}). 
\begin{proof} One can modify the
proof in \cite{Brown} as follows:
We take a sequence of isometries $s_1,s_2,\ldots\in \Mult{B}$
with $\sum s_ns_n^*$ strictly convergent to $1$, and strictly
positive contractions  $b\in B_+$  and $d\in D_+$.

Let $e:=\sum _n 2^{-n}s_nds_n^*$
and let $e_{j,k}$ denote  matrix units of $\K$. 
It is easy to check that
$$\sum _{j,k} a_{j,k} \otimes  e_{j,k} \mapsto
\sum _{j,k} s_ja_{j,k}s_k^*$$
extends to a *-mono\-mor\-phism $\tau$
from $D\otimes \K$ onto the hereditary 
\Cast-sub\-al\-ge\-bra
$D_0:=\overline{eBe}$ of $B$ and satisfies  
$\tau\left((D\cap I)\otimes \K\right)=D_0\cap I$ 
for closed ideals
$I$ of $B$.  In particular, $D_0$ is full in $B$ if and
only if $D$ is full in $B$. Thus it suffices to find
an iso\-mor\-phism $\psi_0$ from the stable $D_0$ onto $B$ with
$\psi_0\left(D_0\cap I\right)=I$ for $I\in \latI{B}$,
because then $\varphi:=\psi_0\circ \tau$ is as desired.

Suppose now that $D$ is a stable full 
and $\sigma$-unital hereditary \Cast-sub\-al\-ge\-bra
of $B$ (i.e.\ $DBD=D$, $D\cong D\otimes \K$, 
$\Span{BDB}$ is dense 
in $B$ and $D_+$ contains a strictly positive element $d$).
Let $d_1:= s_1bs_1^*$, 
$d_2:=s_2ds_2^*$, $d_3:=d_1+d_2$, and
let $D_k$ be the hereditary 
\Cast-sub\-al\-ge\-bra
of $B$ generated by $d_k$ ($k=1,2,3$).
Then $D_k$ is a stable $\sigma$-unital full hereditary 
\Cast-sub\-al\-ge\-bra 
of $B$ for $k=1,2,3$, $D_1=s_1Bs_1*$ and $D_2=s_2Ds_2^*$.
Moreover, $D_1$ and $D_2$ are orthogonal
corners of $D_3$ such that $D_1+D_2$ contains the strictly
positive element $d_3$ of $D_3$. Thus $D_1$ and
$D_2$ are $\sigma$-unital, stable and full corners of $D_3$.
By Remark \ref{R:full-corner}, there are isometries
$t_1,t_2\in \Mult{D_3}$ with 
$D_j=t_jt_j^*D_3t_jt_j^*=t_jD_3t_j^*$ for $j=1,2$.
Let $v:=t_1t_2^*\in \Mult{D_3}$ and 
$\psi(a):= s_1^*v(s_2as_2^*)v^*s_1$ for $a\in D$.
$\psi$ is an isomorphism from $D$ onto $B$
and satisfies $\psi(D\cap I)=I$, because
$s_2(D\cap I)s_2^*=D_2\cap I$, $s_1^*(I\cap D_1)s_1=I$,
$D_k\cap I=D_k\cap(D_3\cap I)$.
\end{proof}

\begin{DEF}\Label{D:r-nuc}
Let $\Psi\colon \latI{B}\to \latI{A}$ an order preserving map.
A completely positive map $V\colon A\to B$ is 
\emph{$\Psi$-equivariant} if 
$V\left(\Psi(J)\right)\subset J$ for every $J\in \latI{B}$.  
$V$ is \emph{$\Psi$-residually nuclear} if 
$V$ is $\Psi$-equivariant and
the induced maps
$[V]_J\colon A/\Psi(J)\to B/J$ are nuclear for every 
$J\in \latI{B}$.
\end{DEF}
Clearly, every $\Psi$-equivariant
c.p.~map $V\colon A\to B$ is 
$\Psi$-residually nuclear
if $A$ or $B$ is nuclear. 
But for non-nuclear $A$ and $B$
the nuclearity of c.p.~maps with 
$V\left(\Psi(J)\right)\subset J$ does not
imply the nuclearity of $[V]_J$ in general (cf.~\cite[sec.~5.3]{K.book}).

\begin{PRP}\Label{P:H0}
Suppose that $A$ and $B$ are separable stable \Cast-algebras
such that 
$B\otimes\OO{2}$ contains an Abelian regular \Cast-subalgebra
$C$ (cf.\ Definition  \ref{D:regular}).  Then every map 
$$\Psi\colon \latI{B}\cong\latO{\Prim{B}}
\to\latI{A}\cong\latO{\Prim{A}}$$
 with
properties (I) and (II) of Definition \ref{D:l-g-preserv} 
can be realized 
by a *-mono\-mor\-phism $h\colon A\hookrightarrow \Mult{B}$ 
with the following properties:
\begin{Aufl}
\item $\, h\, $ is non-degenerate, i.e.\ $h(A)B$ is dense
  in $B$.
\item $\, h\, $ is unitarily equivalent to its infinite repeat 
  $\delta_\infty\circ h$.
\item $\Psi(J)=h^{-1} \left( h(A)\cap \Mult{B,J}\right)$
  for every $J\in\latI{B}$.
\item For every $b\in B$ the completely positive map
  $T_b\colon a\in A\mapsto b^*h(a)b\in B$ 
        is $\Psi$-residually nuclear.
\end{Aufl}
\end{PRP}

\begin{proof}
First we consider the case where $B= \Cont[0]{Y,\K}$
for a locally compact Polish space $Y$. 
The set $\mathcal{C}$
of $\Psi$-equivariant c.p.~maps $T$ from $A$ to $B$
is closed in point-norm, and is operator convex
in the sense of Definition \ref{D:matr-conv-cone}.
Every $T\in \mathcal{C}$ is $\Psi$-residually nuclear,
because $\Cont[0]{Y,\K}$ is nuclear.

If we use the correspondence between $\latI{B}$
and $\latO{Y}$, then Lemma \ref{L:exist-Psi-eqiv.-cp-map}
tells us that for every $J\in \latI{B}$
and every $a\not\in \Psi(J)$ there is $T\in \mathcal{C}$
with $T(a)\not\in J$. Thus, 
$$
\Psi(J) =\{ a\in A\fdg 
V(a)\in J\quad\textrm{for all}\; V\in \mathcal{C} \}
.$$
This implies $\Psi(J_0)=A$ for the closed linear span 
$J_0$ of 
$\{ T(a)\fdg  T\in \mathcal{C},\, a\in A\}$.
Since $\Psi^{-1}(A)=\{ B\}$ 
(by property (II) of $\Psi$), 
it follows $J_0=B$, i.e.~$\mathcal{C}$ is full.

By Proposition \ref{P:h-from-cone} 
there is a non-degenerate
*-homo\-mor\-phism $h\colon A\to \Mult{B}$ with
properties (i)--(iv).
Since our $\mathcal{C}$ is separating for $A_+\setminus\{0\}$,
it follows that $h$ is faithful by 
Remark \ref{R:Psi-from-cp-maps}. 

\smallskip

The case of general $B$ and Abelian regular $C\subset B$
reduces to the above special case as follows:

Let $B$ a separable stable \Cast-algebra
and $C\subset B$ an Abelian regular \Cast-subalgebra.
In particular, the hereditary \Cast-subalgebra
$D:=\overline{CBC}$ of $B$ is full. Thus, by Lemma
\ref{L:Brown} there is an isomorphism
$\varphi$ from $D\otimes \K$ onto $B$ with
$\varphi \left((D\cap J)\otimes \K\right)=J$
for $J\in \latI{B}$.

It holds
$C\cong \Cont[0]{Y}$ 
for the locally compact space
$Y:=\Prim{C}$ (by Gelfand transformation), 
and the restriction of $\varphi$
to $C\otimes \K$ defines a non-degenerate
*-mono\-mor\-phism $\eta$ from 
$\Cont[0]{Y,\K}\cong C\otimes \K\subset D\otimes \K$ 
into $B$. 
Let $B_1:=\eta(\Cont[0]{Y,\K})=\varphi(C\otimes \K)$. 
Then
$$
B_1\cap J
=\varphi((C\cap J)\otimes \K)
=\eta (\Cont[0]{U,\K})
$$
for $J\in \latI{B}$ and the support $U$ of
$C\cap J$ in $Y=\Prim{C}$.

Let $\Psi_1\colon \latO{\Prim{B}}\to \latO{Y}$ the 
lattice monomorphism (with properties (I)--(IV)
of Definition \ref{D:l-g-preserv}) that is 
induced by $I\mapsto \eta^{-1}(I\cap \varphi (B_1))$
(cf.\ Lemma \ref{L:reg-abel-psi}).
We use the right-inverse $\Phi_1$ of $\Psi_1$ as
considered in
Definition \ref{D:pseudo-inv} and Lemma \ref{L:Phi},
and  define a map
$\Phi_2\colon \latO{Y}\to \latI{B}$ by
$\Phi_1(U):=\ke (\Prim{B}\setminus \Phi_1(U))$.
Then $\Phi_1$ satisfies property (I) and (II)
of Definition \ref{D:l-g-preserv} by Lemma
\ref{L:Phi}. 
Further $\Phi_2(U)=J$ for the
support $U$ of $I:=\eta^{-1}(B_1\cap J)$,
because $C\cap J$ and  $I\subset \Cont[0]{Y,\K}$
have the same support in $Y=\Prim{C}$
and $\Psi_1(V)=U$ for the support 
$V\in \latO{\Prim{B}}$ of $J$.

Let $\Psi_3(U):= \Psi\left(\Phi_2(U)\right)$ for 
$U\in \latO{Y}$. $\Psi_3$ satisfies properties (I)
and (II),
because $\Psi$ and $\Psi_2$ satisfy (I) and (II), and
$\Psi_3(U)=\Psi(J)$ for $J\in \latI{B}$ and for
the support $U$ of $\eta^{-1}(B_1\cap J)$.
{}For $I\in \latI{B_1}$
let $\Psi_4(I):= \Psi_3(U)$ for the support 
$U\in \latO{Y}$
of $\eta^{-1}(I)\in \Cont[0]{Y,K}$.
Then $\Psi_4(B_1\cap J)=\Psi(J)$ for 
$J\in \latI{B}$. 

By the above considered special case, there
is a *-mono\-mor\-phism $h\colon A\to \Mult{B_1}$
with (i)--(iv) 
(for $(B_1, \Psi_4)$ in place of $(B, \Psi)$).

The nuclear \Cast-sub\-al\-ge\-bra $B_1\subset B$ 
satisfies
$\Span{B_1B}$ dense in $B$. Thus, $\Mult{B_1}$ is
(in a natural way)
a strictly closed \Cast-subalgebra of $\Mult{B}$.

Then $h\colon A\to \Mult{B}$ 
is a monomorphism and satisfies (i)--(iii):

$h(A)B=h(A)B_1B=B_1B=B$, $h$ is unitarily equivalent
to $\delta_\infty \circ h$  by 
Remark \ref{R:delta-infty}(ii), and
$h(A)\cap \Mult{B,J}=h(\Psi(J))$ because 
$$\Mult{B,J}\cap \Mult{B_1}=\Mult{B_1,B_1\cap J}$$
and $\Psi_4(B_1\cap J)=\Psi(J)$.

\smallskip

$V_b\colon a\mapsto b^*h(a)b$
is a $\Psi$-equivariant c.p.\ map from
$A$ into $B$ by (iii).
$[V_b]_J\colon A/\Psi(J)\to B/J$ factorizes over the
nuclear \Cast-algebra $B_1/(B_1\cap J)$,
because 
$$
[V_b]_J(a+\Psi(J))
= \pi_J(d)^* [V_c]_I(a+\Psi(J))\pi(d)
$$
for $I:=B_1\cap J$,
$c\in B_1$ and $d\in B$ with $cd=b$.
Thus
$V_b$ is also residually nuclear for every $b\in B$, 
i.e.\ $h$ satisfies also (iv) for $\Psi$ and $B$.
\end{proof}

\begin{REM}\Label{R:unique}
$\, h\, $ has the property that
a completely positive map $V$ from $A$ to $B$ 
is $\Psi$-residually nuclear (cf.~Definition \ref{D:r-nuc})
if and only if 
$V$ can be approximated in point-norm by completely positive
maps $W_b$  for a suitable $b\in B$, 
cf.\ \cite[chp.~3]{K.book}.
It follows that 
the map 
$\, h\, $
of Proposition \ref{P:H0} is determined by 
(i)--(iv) up to unitary homotopy, 
cf. Remark \ref{R:unitary-homotopy}. 
\end{REM}

\begin{REM}\Label{R:iso-of-prim}
By \cite[chp.~1,~cor.~L]{K.book}, 
every iso\-mor\-phism
$\alpha$ from the primitive ideal space $\Prim{A}$ 
of a separable nuclear 
\Cast-algebra $A$ onto the primitive ideal space 
$\Prim{B}$ of a
separable nuclear \Cast-algebra $B$ is induced by an 
*-iso\-mor\-phism
$\varphi$ from $A\otimes\OO{2}\otimes\K$ onto 
$B\otimes\OO{2}\otimes\K$.
\end{REM}

\begin{COR}\Label{C:H0-from-Omega}
Let $P$ be a locally compact Polish space and 
$A:=\Cont[0]{P,\K}$.  
Suppose that
$\Omega$ is a sublattice of 
$\,\latO{P}\cong \latI{A}$ which is closed
under l.u.b.\ and g.l.b.\ and contains 
$\emptyset\sim \{0\}$ and 
$P\sim A$.
Then there exists a *-mono\-mor\-phism 
$h\colon A\to\Mult{A}$
with the following properties:
\begin{Aufl}
\item $\, h\, $ is non-degenerate and faithful,
\item $\, h\, $ is unitarily equivalent to its infinite repeat 
  $\delta_\infty\circ h$,
\item if $J\in\latI{A}$ satisfies $h(J)A\subset J$ then
  $h(A)\cap \Mult{A,J}=h(J)$,
\item the support of  $J\in\latI{A}$ is in 
        $\Omega$ if and only if
 $h(J)A\subset J$
\end{Aufl}
\end{COR}

\FRem{$\mathcal{Z}$ better, or $\Omega$ good}

\begin{proof} 
By Remark \ref{R:sublattice} there is a map 
$\Psi\colon \latO{P} \to \latO{P}$ 
that satisfies properties (I) and (II) of
Definition \ref{D:l-g-preserv}, $\Psi\circ \Psi=\Psi$, 
$\Psi(V)\subset V$ for $V\in \latO{P}$ and
$\Psi(V)=V$ if and only if $V\in \Omega$. 

Recall that $P=\Prim{A}$,
$\ke (P\setminus U)=\Cont[0]{U,\K}\in \latI{A}$ for
$U\in \latO{P}$, and 
$J=\Cont[0]{U_J,\K}$ for the support $U_J\in \latO{P}$
and $J\in\latI{A}$.
By Proposition \ref{P:H0} there exists a non-degenerate 
*-mono\-mor\-phism
$h\colon A\hookrightarrow \Mult{A}$ such that 
$$\Cont[0]{\Psi(V),\K}
=h^{-1}\left(h(A)\cap \Mult{A,\Cont[0]{V,\K}}\right)$$
for open subsets $V$ of $P$ and 
$\delta_\infty\circ h$ unitarily
equivalent to $\, h\,$.
Thus, for $J=\Cont[0]{U_J,\K}$ holds $h(J)A\subset J$, 
i.e.~$J\subset h^{-1}\left(h(A)\cap \Mult{A,J}\right)$, 
if and
only if $U_J\subset \Psi\left(U_J\right)$. The latter
implies $U_J=\Psi\left(U_J\right)$, 
i.e.~$h(J)=h(A)\cap \Mult{A,J}$.
Hence, $\, h\, $ satisfies (i)-(iv).
\end{proof}

\begin{REM}\Label{R:Omega-from-H0}
Suppose that $A$ is \Cast-algebra and $h\colon A\to\Mult{A}$
a non-de\-gener\-ate *-mono\-mor\-phism 
(i.e.~$h^{-1}(0)=\{ 0\}$ 
and $h(A)A=A$).
Consider the set $\Omega$ of supports of closed ideals $J$ 
with the property $h(J)A\subset J$.
Then $\Omega$ is a sub-lattice of $\latO{\Prim{A}}$ that 
contains $\emptyset$ and $\Prim{A}$ 
and is closed under l.u.b.\ (=unions) and
g.l.b.\ (=interiors of intersections).

Indeed, $h\left( J_\alpha \right) A\subset J_\alpha$ implies
$h \left( \bigcap_\alpha J_\alpha \right) A
\subset \bigcap_\alpha J_\alpha$, 
i.e.~$\Omega$ is closed under kernels of intersections
(g.l.b.). Similarly,
$$h \left( \overline{\sum_\alpha 
J_\alpha} \right) A\subset\overline{\sum_\alpha J_\alpha}\,,$$
 i.e.\
that $\Omega$ is also closed under unions (l.u.b.).
\end{REM}

\begin{REM}\Label{R:H0-inv-ideal}
The assumption of Corollary \ref{C:H0-from-Omega} 
implies, that
for closed ideals $J$ of $A$ with $h(J)A\subset J$ holds
\begin{equation} \label{Eq:H0-inv-ideal}
h(J)=h(A)\cap\Mult{A,J}=
h(A)\cap \left( \Mult{A,J}+A \right) , 
\end{equation}
i.e.\ if $a,b\in A$ and 
% $J\in\latI{A}$ satisfy $h(J)A\subset J$ and
$\left( h(a)+b\right) A\subset J$ then $a,b\in J$.

Indeed, if $U\in\Mult{A}$ is a unitary with $U^*h(a)U
=\delta_\infty \left( h(a)\right) $ and  
$c:=U^*bU\in A$, then
$ \left(\delta_\infty\left( h(a)\right) + c\right)A\subset J$.
For elements $e\in A$ we get $h(a)e+s_j^*cs_je\in J$ and
$\lim_{j\to\infty} \left\| cs_j \right\|=0$ because 
$\sum s_js_j^*$ converges strictly to $1$ in $\Mult{A}$.
It follows that $h(a)\in \Mult{A,J}$, thus $a\in J$ by
(iii).

\smallskip

In a similar way it holds for $d\in \Mult{A}$, that 
$\delta_\infty (d) \in \left(\Mult{A,J}+A\right)$, if and
only if, $d\in \Mult{A,J}$, if and only if, 
$\delta_\infty (d)\in \Mult{A,J}$.
%Thus, if  
%$\delta_\infty h$ is unitarily equivalent to $h$ then
%$$ h^{-1} \left( h(A)\cap \Mult{A,J} \right) 
%= 
%h^{-1} \left( h(A)\cap \left( \Mult{A,J}+A \right)\right).$$
\end{REM}

\FRem{notation $E$, $F$ in following}

\begin{REM}\Label{R:sigma-from-H0}
Let $A$ be a \Cast-algebra and $\, h\, $ a non-degenerate
*-mono\-mor\-phism from $A$ into $\Mult{A}$ such that
$h(A)\cap A=\{0\}$.  
Then $\, h\, $ uniquely extends to a faithful unital
strictly continuous endo\-mor\-phism
of the multiplier algebra $\Mult{A}$, 
which we denote also by $\, h\, $.

If $A$ is a type I \Cast-algebra then clearly the closure
$E$ of
$$A+h(A)+h^2(A)+h^3(A)+\dots $$
is a type I \Cast-algebra, that has a decomposition series with
intermediate factors isomorphic to $A$.
$\, h\, $ defines a non-degenerate endo\-mor\-phism of $E$.
If we take the inductive limit $F=\indlim{h\colon E\to E}$ 
then
there is a natural iso\-mor\-phism $\sigma$ of $F$ such that
(under canonical identification) $\sigma^{-1}(a)=h(a)$ for
$a\in A\subset E\subset F$.  Let $D=h(A)\subset E\subset F$.
$D$ satisfies $D\sigma(D)\subset\sigma(D)$, that
$\sigma(D)$ is an essential ideal of $D+\sigma(D)$, and
$D\cap\sigma(D)=\{0\}$.
Then $E$ is naturally isomorphic to $D_{-\infty,1}$, cf.\ 
Remark \ref{R:type-I-endo}.

We give an explicit description 
of the corresponding embeddings and
the iso\-mor\-phism 
$\sigma$ by natural embedding of $F$ into sequence
spaces modulo zero sequences:

Let 
$B:=\ell_\infty\left(\Mult{A}\right)/c_0\left(\Mult{A}\right)$, 
$\sigma\in\Aut{B}$ induced by the forward shift on 
$\ell_\infty\left(\Mult{A}\right)$,
$$ 
\sigma \left( ( m_1,m_2,\dots ) 
         +c_0\left(\Mult{A}\right)\right) 
                =
   ( 0,m_1,m_2,\dots ) +c_0\left(\Mult{A}\right).
$$
The inductive limit 
$\indlim{h\colon \Mult{A}\to \Mult{A}}\supset F\supset E$ 
embeds canonically into $B$:
\begin{Diagramm}
\xymatrix{
\Mult{A} \ar[r]^{h} & \Mult{A}\ar[r]^{h} 
& \Mult{A}\ar[r]^{h} & \dots\ar@{.>}[r] & B \\
A \ar[ur]_{{h}} \ar[r] & A+h(A)\ar[ur]_{{h}}
\ar[r] & \dots&\ar@{.>}[r] & E \ar@{^{(}->}[u]
}
\end{Diagramm}
Then $E$ is embedded in $B$ and $\sigma^{-1}(c)$ equals 
$h(c)$
for $c\in E$, and $F$ is the smallest $\sigma$-invariant
\Cast-sub\-al\-ge\-bra of $B$ containing $D:=h(A)\subset E$.
We shall see below that the crossed product $F\rtimes_\sigma\Z$
is isomorphic to the
\Name{Cuntz--Pimsner} algebra of the
\Name{Hilbert} bi-module $\Hm{A,h}$ given by 
$h\colon A\to\Mult{A}$
(cf.\ (i) of Theorem \ref{T:main}).
%\begin{eqnarray*}
%D_1=D&:=&h_{1,\infty}(A) \\
%D_n &=& h_{n,\infty}(A)
%\end{eqnarray*}
%then
%\begin{eqnarray*}
%h_{n,\infty}(a)&=& ( 0,\dots,0,a,h(a),h^2(a),\dots ) 
%  +c_0(\Mult{A})\\
%\sigma^{n-1}(D)&=&D_n
%\end{eqnarray*}
%because 
%$\sigma^{n-1} ( h_{1,\infty}(a) ) =h_{n,\infty}(a)$.
\end{REM}

\begin{REM}
If $h\colon A\to\Mult{A}$ is as in Remark \ref{R:sigma-from-H0}
and every closed ideal $J$ of $A$ with 
$h(J)A\subset J$ satisfies Equation (\ref{Eq:H0-inv-ideal}) 
then the lattice of closed ideals of
$E\rtimes_\sigma\Z$ is naturally isomorphic to the sub-lattice 
$\Omega$
of closed ideals $J$ of $A$ with $h(J)A\subset J$
(cf.\ (ii) of Theorem \ref{T:main}).
\end{REM}
%%%%%%%%%%%%%%%%%%%%%%%%%%%%%%%%%%%%

\section{\Name{Cuntz--Pimsner} algebras}

\subsection{\Name{Fock} bi-module and \Name{Toeplitz} algebra}
In \cite{Pi97} \Name{Pimsner} has defined two algebras, 
namely the 
\Name{Toeplitz} algebra $\TTT{E}$ and the 
\Name{Cuntz--Pimsner} algebra
$\OOO{E}$ of a \Name{Hilbert} $A$-bi-module $\HM[E]{}$.
Those algebras are 
defined as a \Cast-subalgebra and a \Cast-sub-quotient of 
the adjoint-able bounded operators over the (generalized)
\emph{\Name{Fock} space} $\Fock{E}$ which is the 
\Name{Hilbert} $A$-module
%\begin{eqnarray*}
$$
\Fock{E}:=
%&:=&
\bigoplus_{n=1}^\infty E^{(\otimes_A)n}, 
$$
%\\
%\text{
where
%}\quad 
$$
E^{(\otimes_A)0}:=A_A\,,\;\,  E^{(\otimes_A)1}:=E\,,\;\; 
\text{and}\;\; 
%\\
E^{(\otimes_A)n}:=
%&:=&
\underbrace{E\otimes_A E \otimes_A \dots \otimes_A E
}_{n\;\, \textrm{factors}}.
$$
%\end{eqnarray*}
%
{}For the precise Definition of $E^{(\otimes_A)n}$
see Remark \ref{R:L-e}, where
also the maps $\eta_n\colon A\to \LOp{E^{(\otimes_A)n}}$
and left $A$-module structures 
$a\cdot e:=\eta_n(h(a))e$ on $E^{(\otimes_A)n}$
are defined.
(Here $E^{(\otimes_A)0}=A_A$ is as in Example
\ref{Ex:H-module}(i) and the left multiplication
is defined by $a\cdot e:=ae$.)

The 
{}For $e\in E$, let $T_e$ denote the operator
\begin{eqnarray}\Label{Eq:T-e}
T_e \left( a,f_1,f_2,\ldots \right) 
:=
\left(0,ea,e\otimes_A f_1,e\otimes_A f_2,\ldots \right)\,,
\end{eqnarray}
where $a\in A$ and $f_n\in E^{(\otimes_A)n}$
with $\sum _n \SP{f_n}{f_n}$ convergent in $A$.
Then, by Remark \ref{R:L-e}(iii),  
$\| T_e\| \leq \| e\|$ and $T_e$ is adjoint-able
with 
adjoint $\left(T_e\right)^*$ given by
$$
\left(T_e\right)^*
\left(a,e_1,e_2\otimes_A f_1, e_3\otimes_A f_2, \ldots\right)=
\left(
\SP{e}{e_1}, \SP{e}{e_2}\cdot f_1,\SP{e}{e_3}\cdot f_2,\ldots 
\right)\,,
$$
where $e_n\in E$, $f_n\in E^{(\otimes_A)n}$
with 
$\sum _n 
\SP{e_{n+1}\otimes _A f_n}{e_{n+1}\otimes _A f_n}\in A$. 
%convergent in $A$.

\begin{DEF}\Label{D:Toep/J(E)}
The (generalized) \emph{\Name{Toeplitz} algebra} is the 
\Cast-sub\-al\-ge\-bra 
$\TTT{E}$ of $\LOp{\Fock{E}}$ generated by operators 
$T_e$ for all $e\in E$. 

Let $p_k\left(a,f_1,\ldots , f_k,f_{k+1}\ldots\right):=
\left(a,f_1,\ldots,f_k,0,0,\ldots\right)$
for $k\in \N$. $p_k$ is a self-adjoint projection in 
$\LOp{\Fock{E}}$ and  
$p_k \leq  p_{k+1}$. 

$\mathcal{J}(E)$ denotes the hereditary 
\Cast-sub\-al\-ge\-bra
of $\LOp{\Fock{E}}$ generated by $\{ p_1,p_2,\ldots \}$,
i.e.~the closure of $\bigcup_k p_k\LOp{\Fock{E}}p_k$.
\end{DEF}

\begin{REM}\Label{R:Fock}
$\mathcal{J}(E)$ is essential in $\LOp{\Fock{E}}$,
because the (left) $\mathcal{J}(E)$-module $E$
is (obviously) non-degenerate.
Thus, the (two-sided) normalizer algebra 
$\Norm{\mathcal{J}(E)}$
\FRem{definition below in \ref{D:Norm/Ann}}
of $\mathcal{J}(E)$ in $\LOp{\Fock{E}}$
is in a natural manner a unital \Cast-sub\-al\-ge\-bra
of the multiplier algebra $\Mult{\mathcal{J}(E)}$ of
$\mathcal{J}(E)$. 

Since $p_{k+1}T_ep_k=T_ep_k$ and $p_kT_e=p_kT_ep_{k-1}$,
$T_e$ is in the normalizer algebra of $\mathcal{J}(E)$.
Thus the \Name{Toeplitz} algebra 
$\TTT{E}\subset\LOp{\Fock{E}}$
is naturally contained in  $\Mult{\mathcal{J}(E)}$.
\end{REM}
\begin{DEF}\Label{D:CP-alg}
The \emph{\Name{Cuntz--Pimsner} algebra} 
$\OOO{E}$ is defined as the image of
the \Name{Toeplitz} algebra $\TTT{E}$ in the corona algebra 
$\Mult{\mathcal{J}(E)}/\mathcal{J}(E)$
of  $\mathcal{J}(E)$.
\end{DEF}

\begin{REM}
The natural
*-mono\-mor\-phism from $\Norm{\mathcal{J}(E)}$
%$\Norm{\LOp{\Fock{E}}, \mathcal{J}(E)}$
to the multiplier algebra $\Mult{\mathcal{J}(E)}$ 
is an \emph{iso\-mor\-phism} from 
%$\Norm{\LOp{\Fock{E}}, \mathcal{J}(E)}$
$\Norm{\mathcal{J}(E)}$
\emph{onto} $\Mult{\mathcal{J}(E)}$,  
because, if a bounded net 
$\left\{ S_\alpha \right\}\subset \mathcal{J}(E)$ 
converges strictly to an element $S$ of 
$\Mult{\mathcal{J}(E)}$,
then the net $\left\{ S_\alpha \right\}$ converges in 
$\LOp{\Fock{E}}$
strongly to a multiplier of $\mathcal{J}(E)$ (cf.\ the remark
following Definition \ref{D:LOp/K}). 
\end{REM}

\subsection{The \Name{Cuntz--Pimsner} algebra 
$\OOO{\Hm{A,h}}$}
We give another description of the 
\Name{Cuntz--Pimsner} algebra $\OOO{\Hm{A,h}}$
of a bi-module 
$\Hm{A,h}$ corresponding to a \emph{non-degenerate}
*-homomorphism $h\colon A\to\Mult{A}$ (cf.\ Definition 
\ref{D:H-bi-module}, compare also Examples 
\ref{Ex:tensor}, \ref{Ex:M(A)} and \ref{Ex:H-module}).

$h$ defines a full \Name{Hilbert} $A$-bi-module $\Hm{A,h}$
with right \Name{Hilbert} $A$-module $E:=A_A$
considered in  Examples \ref{Ex:H-module}(i), 
and left $A$-module structure given by 
$h\colon A\to \Mult{A}=\LOp{E}$ (cf.\ \ref{Ex:M(A)}).
The unique extension of $h$ to a strictly continuous 
unital *-homomorphism from $\Mult{A}$ into $\Mult{A}$
will also be denoted by $h$ (to keep notation
simple). Then the restrictions of
$h^n$ to $A$ are non-degenerate *-homomorphisms
$h^n\colon A\to \Mult{A}$.

\begin{REM}\Label{R:h-non-deg}
By Remark \ref{R:tensor},
there are isomorphisms $I_n$ from the
the $n$-fold tensor products $E^{(\otimes A)n}$
of the \Name{Hilbert} $A$-bi-module $E$ 
given by the non-degenerate
*-homo\-mor\-phism $h\colon A\to \Mult{A}$
onto the \Name{Hilbert} $A$-bi-module given by
$h^n\colon A\to \Mult{A}$.
%, where
%$h$ has to be naturally extended to a
%unital strictly continuous *-homomorphism
%from $\Mult{A}$ into $\Mult{A}$.
Under this identifications $e\otimes_A f$
becomes $h^n(e)f$ for $e\in A\cong E$ and
$f\in A\cong E^{(\otimes A)n}$. 
(Recall here that $E^{(\otimes A)0}$ is 
$A_A$ with left multiplication
given by $h^0\colon A\to \Mult{A}$ where $h^0(a)b:=ab$.)

Hence, 
$\Fock{E}$ is isomorphic to $\HM{A}$ 
by an iso\-mor\-phism such that
$T_e$ becomes 
$T_e \left(a_0,a_1,a_2,\ldots\right):=
\left(0,ea_0, h(e)a_1, h^2(e)a_2, \ldots\right)$
for $(a_0,a_1,\ldots)$ in $\HM{A}$. 

\FRem{give precise cross-ref's below!!}

Since $\HM{A}$ is nothing else than 
$\left(A\otimes \K\right)\left(1\otimes e_{0,0}\right)$ 
for the minimal projection 
$e_{0,0}\in \K= \K \left(\ell_2\{ 0,1,2,\ldots \}\right)$,
there is a natural iso\-mor\-phism from
$\Mult{A\otimes \K}$ onto 
$\LOp{\HM{A}}\cong \LOp{\Fock{E}}$ such that
the hereditary \Cast-sub\-al\-ge\-bra
$\Mult{A}\otimes \K\subset \Mult{A\otimes \K}$ maps onto
$\mathcal{J}(E)$ (cf.\ the remark
following Definition \ref{D:LOp/K}).
Clearly $\ell_\infty\left(\Mult{A}\right)\subset 
\Norm{\Mult{A}\otimes \K} \cong \Mult{\Mult{A}\otimes \K}$.
Let, for $a\in \Mult{A}$,
$$ 
h^\infty (a) := \left(a,h(a),h^2(a),h^3(a),\dots \right)
\in \ell_\infty\left(\Mult{A}\right)\subset \LOp{\HM{A}}\,, 
$$
and let $\TT\in \LOp{\HM{A}}$ denote the forward shift
$\TT (a_0,a_1,\ldots):=(0,a_0, a_1, \ldots)$, i.e.\
$\TT=1\otimes\TT_0$ where $\TT_0$ is a 
Toeplitz operator (forward
shift) on $\ell_2(\N)$. Then 
$\TT$ is in 
$\Norm{\Mult{A}\otimes \K}\cong \Mult{\Mult{A}\otimes \K}$, 
and
$T_e$ decomposes
as 
$T_e=\TT h^\infty (e)$ for $e\in A$.
Note that 
$U:=\TT + \left(\Mult{A}\otimes \K\right)$  is a unitary in
the stable corona 
$Q^s\left(\Mult{A}\right):=
\Mult{\Mult{A}\otimes \K}/\left(\Mult{A}\otimes \K\right)$
of $\Mult{A}\otimes \K$, that
$$h_{1,\infty}\colon a\in \Mult{A} \mapsto  
h^\infty(a)+\left(\Mult{A}\otimes \K\right)
\in Q^s\left(\Mult{A}\right)$$
is a *-homo\-mor\-phism from $\Mult{A}$ into the corona 
$Q^s(\Mult{A})$ of 
$\Mult{A}\otimes \K$.

It follows that \emph{the \Name{Cuntz--Pimsner} algebra  
$\OOO{\Hm{A,h}}$ is the
\Cast-sub\-al\-ge\-bra $\Cast (Uh_{1,\infty}(A))$ of
$Q^s\left(\Mult{A}\right)$
which is generated by the elements
$Uh_{1,\infty}(e)$ for $e\in A$}.

Recall that the \Name{Toeplitz} operator $\TT$ is related to
the forward shift on 
$\ell_\infty \left(\Mult{A}\right)$ as follows:
Let 
$B:=
\ell_\infty \left(\Mult{A}\right)/c_0\left(\Mult{A}\right)
\subset Q^s\left(\Mult{A}\right)$ 
and $\sigma$ denote the auto\-mor\-phism of $B$
induced by the forward shift
$(a_0,a_1,\ldots)\to (0,a_0,a_1,\ldots)$ on 
$\ell_\infty\left(\Mult{A}\right)$.
Then $UbU^*=\sigma (b)$ for $b\in B$ by 
Lemma \ref{L:B-rtimes-Z}.

On the other hand, $h_{1,\infty}\left(h(a)\right)=
\sigma^{-1}\left(h_{1,\infty}(a)\right)$
for $a\in \Mult{A}$.
Thus,
$h_{1,\infty}\left(h(a)b\right)=
\sigma^{-1}\left(h_{1,\infty}(a)\right)h_{1,\infty}(b)\,$,
\,\,and, hence,
$$
h_{1,\infty}\left(h(a)b\right)=
U^{-1}h_{1,\infty}(a)Uh_{1,\infty}(b)\quad
for \;\,a,b\in \Mult{A}.
$$
\end{REM}
%%%
%
%
\begin{PRP}\Label{P:Pims-alg}
Suppose that $h\colon A\to \Mult{A}$ is a non-degenerate
*-homo\-mor\-phism. Let 
$$B:= \ell_\infty \left(\Mult{A}\right) /
c_0 \left(\Mult{A}\right)\subset Q^s\left(\Mult{A}\right)$$
and $\sigma\in \Aut{B}$ induced by the forward shift on
on $\ell_\infty \left(\Mult{A}\right)$.

There are a *-homo\-mor\-phism 
$\varphi\colon A\to B$
and a unitary $U\in Q^s\left(\Mult{A}\right)$ 
such that 
\begin{Aufl}
\item $\sigma(b)=Ub\,U^*$ for $b\in B$,
\item $ \varphi \left( h(a)b \right) =
U^*\varphi(a)U \varphi(b)$ for all
$a,b\in A$, and
\item
$U\varphi(A)$ generates the Cuntz--Pimsner
algebra $\OOO{\Hm{A,h}}$
as \Cast-sub\-al\-ge\-bra of $Q^s\left(\Mult{A}\right)$,
in particular $\varphi(A)\subset\OOO{\Hm{A,h}}$.
\end{Aufl}

Let $D:=\varphi(A)\subset B$ and let $[D]_\sigma$ 
denote the smallest
$\sigma$-invariant \Cast-sub\-al\-ge\-bra of $B$
containing $D$. Then
the \Cast-subalgebra $E$ of 
$Q^s\left(\Mult{A}\right)$ that is generated by
$\bigcup _{k\in \N} U^{-k}D$ 
is naturally isomorphic to
$[D]_\sigma\rtimes _\sigma \Z$, and
$\OOO{\Hm{A,h}}$ is 
the full hereditary \Cast-sub\-al\-ge\-bra of 
$E$ generated
by $D$.

$[D]_\sigma\rtimes _\sigma \Z$ and 
$\OOO{\Hm{A,h}}$ are isomorphic if $A$ is
stable and contains a strictly positive element.

$\varphi$ is a mono\-mor\-phism if 
$h$ is moreover faithful.
\end{PRP}
\begin{proof} Let $U\in Q^s\left(\Mult{A}\right)$, 
$\sigma\in \Aut{B}$,
$\varphi:=h_{1,\infty}$ and $D:=\varphi(A)$, 
where $h_{1,\infty}$
and $\sigma$ are
as in Remark \ref{R:h-non-deg}.
 
Then, by Remark \ref{R:h-non-deg},
$D\subset B\subset Q^s\left(\Mult{A}\right)$,
$U$ is a unitary in $Q^s\left(\Mult{A}\right)$ with 
$UbU^{-1}=\sigma(b)$
for $b\in B$,
and $\varphi^{-1}(D)D=D$, because 
$h_{1,\infty}\left(h(a)b\right)=
\sigma^{-1}\left(h_{1,\infty}(a)\right)h_{1,\infty}(b)$
and $h(A)A=A$. 

$\sigma^{-k}(D)D=D$, because 
$\sigma^{-j}(D)\sigma^{1-j}(D)=\sigma^{1-j}(D)$
for $1\leq j\leq k$.
It means $U^{-k}DU^kD=D$, $DU^kD=U^kD$ and 
$DU^{-k}D=DU^{-k}$ for $k\in \N$. 
It shows for $i,j,m,n\in \Z$:
\begin{equation}\Label{Eq:UD} 
%\[
U^iDU^jU^mDU^n =
\left\{ {\begin{array}{r@{\quad \text{if}\;\;}l} 

U^{i+(j+m)}DU^n 
&
j+m\ge 0,
\\
U^iDU^{n+(j+m)}
%\quad \text{if}\;\,
& 
j+m<0. 
\end{array}
}\right.
%\]
\end{equation}

Thus, the sum $C:=\sum _{i,j\in \Z} U^iDU^j$ 
is a *-sub\-al\-ge\-bra of 
$Q^s\left(\Mult{A}\right)$, 
and $U$ is in the two-sided
normalizer $\Norm{C}$ of $C$. 
It follows $U\in \Norm{\overline{C}}$.
\FRem{$\Norm{C}$ later defined}

$C$ is *-algebraically generated by 
$\bigcup _{k\in \N} U^{-k}D$,
because the \Cast-algebra generated by 
$U^{-1}D, U^{-2}D, \ldots$
contains also $D=(U^*D)^*(U^*D)$, $DU^{-k}D=DU^{-k}$ and 
$U^kD=\left(DU^{-k}\right)^*$ for 
$k\in \N$, and finally,
$\left(U^i D\right)\left(U^{-j}D\right)^*
= U^iDU^j$ for $i,j\in \Z$.

$\sum _{j\in \Z} \sigma^{j}(D)$
is a *-sub\-al\-ge\-bra of $B\cap C$, because 
$U^kDU^{-k}=\sigma^k(D)$
and $\sigma^j(D)\sigma^k(D)=U^jDU^{k-j}DU^k
=\sigma^{\max(j,k)}(D)$ by (\ref{Eq:UD}).
Let $D_{-\infty,\infty}$ denote its closure. 
Then $D_{-\infty,\infty}$
is the smallest $\sigma$-invariant 
\Cast-sub\-al\-ge\-bra of $B$ that
contains $D$, i.e.~$D_{\infty,\infty}=[D]_\sigma$. 
$\overline{C}$  is the \Cast-sub\-al\-ge\-bra of 
$Q^s\left(\Mult{A}\right)$
generated by $UD_{-\infty,\infty}
%=U\sigma(D_{-\infty,\infty})
=D_{-\infty,\infty}U\,$,
because $C$ is also generated by 
$\bigcup_{k\in \N} \sigma^{k}(D)U$ for $k\in\N$ (note 
$U^{-k}D=
\sigma^k(D)U\cdots \sigma^2(D)U\sigma(D)U$).

Since $\Cast (B,U)\subset Q^s\left(\Mult{A}\right)$ 
is naturally isomorphic
to $B\rtimes _\sigma Z$ (by Lemma \ref{L:B-rtimes-Z}),
$\overline{C}$ is naturally isomorphic to 
$[D]_\sigma\rtimes_\sigma \Z$
(by Remark \ref{Rcross-prod}(iii)).

$G:=\OOO{\Hm{A,h}}$ is the 
\Cast-sub\-al\-ge\-bra of 
$Q^s\left(\Mult{A}\right)$
that is
generated by $UD$, cf.~Remark \ref{R:h-non-deg}.
It follows 
$D=(UD)^*UD\subset G$, and
$U^nD\subset G$ by induction, because
$U^{n+1}D=U^nDUD$.
Hence, $U^mDU^{-n}\subset G$ for $m,n\ge 0$.
The sum $C_1:=\sum _{m,n\ge 0} U^mDU^{-n}$ is identical
to the *-sub\-al\-ge\-bra 
of $C$ given  by 
$DCD)$, because
$DU^iDU^jD=U^mDU^{-n}$ with 
$m,n\ge 0$ for $i,j\in \Z$ by Equation (\ref{Eq:UD}).
Thus, $\overline{C_1}$ is the hereditary 
\Cast-sub\-al\-ge\-bra of
$\overline{C}$ generated by $D$, and 
$G=\overline{C_1}$.

The *-ideal of $C$ generated by $D\subset C_1$ is 
is obviously identical with $C$.
Thus $G$ is the full hereditary \Cast-sub\-al\-ge\-bra of 
$\overline{C}$ generated by $D$. 

$\left\| h_{1,\infty} (a)\right\| =
\lim _{n\to\infty} \left\| h^n(a)\right\|$
for $a\in \Mult{A}$. 
$h^n\colon \Mult{A}\to \Mult{A}$
is faithful for every $n\in \N$ if $h\colon A\to \Mult{A}$
is faithful. Thus, 
$\left\| \varphi (a)\right\| =\left\|h^n(a)\right\|
=\left\| a\right\|$ 
if $h$ is faithful (in addition).

$E=\overline{C}$ and $\OOO{\Hm{A,h}}=DED$ contain strictly
positive elements if $A$ contains 
a strictly positive  element $a$, e.g.\
$\sum _{n\in \N} 2^{-n}\sigma^{-n}(\varphi(a))$
respectively $\varphi(a)$. 

$E$ and $\OOO{\Hm{A,h}}$ are stable 
if $A$ is stable,
because then $D=\varphi(A)$ is stable,
$\OOO{\Hm{A,h}}=DED$ is stable (by 
Remark \ref{R:delta-infty}), and 
$E$ is stable as the inductive
limit of the stable hereditary \Cast-sub\-al\-ge\-bras 
$U^{-n}DEDU^n$
($n=1,2,\ldots$), cf.\ \cite{HjelRor}.

Thus $\OOO{\Hm{A,h}}$ is isomorphic to
$[D]_\sigma\rtimes_\sigma \Z$ by Lemma \ref{L:Brown}
if $A$ is stable and contains a strictly positive element.
\end{proof}
%%%

\begin{COR}\Label{C:case-A-sepa}
$$
\OOO{\Hm{A,h}}\otimes \K \cong 
\left([D]_\sigma \rtimes_\sigma \Z \right)\otimes \K
$$
if $A$ is separable and $h\colon A\to \Mult{A}$ is
non-degenerate. Where $D:=h_{1,\infty}(A)$.
\end{COR}
\begin{proof} By the proof of Proposition \ref{P:Pims-alg}, 
$\OOO{\Hm{A,h}}\otimes \K$
is a full hereditary \Cast-subalgebra of 
$\left([D]_\sigma \rtimes_\sigma \Z \right)\otimes \K$,
and both algebras are separable.
The algebras are isomorphic by Lemma \ref{L:Brown}.
\end{proof}

%%%
\section{Ideals of certain crossed products}
\Label{sec:ideal-of-cross}
Throughout this section we assume that 
$\sigma$ is an auto\-mor\-phism
of a \Cast-algebra $B$ and that 
$D$ is a \Cast-sub\-al\-ge\-bra of $B$.
 
We call a \Cast-sub\-al\-ge\-bra $E\subset B$
\emph{$\sigma$-invariant} if $\sigma(E)=E$.
In this case we will denote the restriction of 
$\sigma$ to $E$ again
by $\sigma$. We also write $\sigma$ 
for the second
adjoint $\sigma^{**}\in \mathrm{Aut}(B^ {**})$ of $\sigma$. 
(It can not cause confusions by
Proposition \ref{P:a4.15}.)

\begin{DEF}\Label{D:Norm/Ann}
For a \Cast-subalgebra  $A$ of a \Cast-algebra $B$
the \emph{normalizer} of $A$ in $B$ is defined by
$$\Norm{B,A}:= \{ b\in B\fdg bA+Ab\subset A \},$$
\FRem{Normalizer to other place?}
and the \emph{annihilator} of $A$ in $B$ by
$$\Ann{B,A}:= \{ b\in B\fdg bA=0=Ab \}\,.$$
We write also $\Norm{A}$ (respectively $\Ann{A}$)
instead of $\Norm{B,A}$ (respectively $\Ann{B,A}$)
if no confusion can arise.
\end{DEF}
%
%
%
%\subsection{$D\subset B$, $\sigma\in\Aut{B}$}
\subsection{Identification of some crossed product}
\Label{ssec:id-of-crossed}
%We use the following conditions:
\begin{DEF}\Label{D:sigm-mod}
Let $\sigma\in \Aut{B}$ and $D\subset B$ a 
\Cast-sub\-al\-ge\-bra.
We say that
$D$ is \emph{$\sigma$-modular} if $D$ has
following properties ($\alpha$), ($\beta$) and ($\gamma$).
\begin{itemize}
\item[($\alpha$)] $D\sigma(D)=\sigma(D)$,
\item[($\beta$)]  $\Ann{D+\sigma(D),\sigma(D)}=\{0\}$,
\item[($\gamma$)] $D\cap\sigma(D)=\{0\}$. 
\end{itemize}

{}For $D\subset B$, let
$[D]_\sigma$ denote the smallest 
$\sigma$-invariant \Cast-sub\-al\-ge\-bra
of $B$ containing $D$. 
\end{DEF}

\begin{REM}\Label{R:a4.1}\mbox{ }
%
%\noindent
(i)
Property ($\alpha$) means equivalently that
$\sigma(D)$ is the closure of $\Span{D\sigma(D)}$, 
i.e.~$D\sigma(D)\subset \sigma(D)$ 
and every approximate unit of $D$
is also an (outer) approximate unit for $\sigma(D)$. 
(Special case of the 
Cohen factorization theorem, cf.\ \cite{Cohen}, 
\cite[thm.\ I,~\S 11.10]{BoDu}.) 

\noindent
(ii)
Note that ($\alpha$) implies $D\sigma^k(D)=\sigma^k(D)$
for $k\in\N$ by induction, because
$D\sigma^{k+1}(D)=D\sigma^k(D\sigma(D))=
(D\sigma^k(D))\sigma^{k+1}(D)$.
(Below we show that $[D]_\sigma$ is the
closure $D_{-\infty,\infty}$
of $\sum _{k\in \Z} \sigma^k(D)$ if 
$D\sigma^k(D)\subset \sigma^k(D)$
for $k\in \N$.)

\noindent
(iii)
Under assumption ($\alpha$) the assumption ($\beta$) means
that $\sigma(D)$ is an essential ideal of $D+\sigma(D)$.

\noindent
(iv)
Property ($\gamma$)  equivalently means
that 
$a+\sigma(b)=0$ implies $a=b=0$ for $a,b\in D$.
Thus, under assumption ($\gamma$), the algebra $D$
satisfies
($\beta$) if and only if 
$a,b\in D$ and $\left(a+\sigma(b)\right)\sigma(D)=\{ 0\}$ 
together imply that  $a=0$.
\end{REM}

In this section we prove the following propositions.
\begin{PRP}\Label{P:a4.15}
Suppose that $D \subset B$ and $\sigma\in\Aut{B}$ satisfy
properties ($\alpha$) and ($\gamma$) of 
Definition \ref{D:sigm-mod},
that $D_1\subset B_1$ and $\sigma_1\in\Aut{B_1}$ satisfy
$D_1\sigma_1(D_1)=D_1$, and that
$\varphi\colon D\to D_1$ is a *-homo\-mor\-phism with
$\varphi \left( \sigma ^{-1}(a)b \right) 
 =\sigma_1 ^{-1} \left( \varphi(a) \right) \varphi(b)$ 
for all $a,b\in D$. 
Then
$\varphi$ extends to a *-homo\-mor\-phism
$\varphi_e\colon [D]_{\sigma}\to [D_1]_{\sigma_1}$ with
$\varphi_e\left(\sigma(c)\right)
=\sigma_1\left(\varphi_e(c)\right)$ for 
$c\in [D]_{\sigma}$.

If $D$ is $\sigma$-modular in the sense of
Definition \ref{D:sigm-mod} and if
$\varphi$ is a *-mono\-mor\-phism 
(respectively a *-iso\-mor\-phism from $D$ onto $D_1$)
then $\varphi_e$ is a mono\-mor\-phism
(respectively  $\varphi_e$  is an iso\-mor\-phism
from $[D]_{\sigma}$ onto $[D_1]_{\sigma_1}$).
\end{PRP}
\begin{PRP}\Label{P:a4.7}
Suppose that $\sigma\in \Aut{B}$ and $D\subset B$
is $\sigma$-modular in the sense of Definition
\ref{D:sigm-mod}.
%
%Suppose $\sigma\in\Aut{B}$ and that $D\subset B$ is
%$\sigma$-modular (cf.~Definition \ref{D:sigm-mod}).
Every non-zero closed ideal $I$ of 
$[D]_\sigma\rtimes_\sigma\Z$
has non-zero intersection $J=I\cap D$ with $D$, and $J$
satisfies $J\sigma(D)\subset\sigma(J)$.
\end{PRP}
In particular, if 
$\varrho\colon [D]_\sigma\rtimes_\sigma\Z\to
\LOp{H}$ is a *-representation with $\varrho|D$ faithful, 
then $\varrho$ is faithful.
\begin{COR}\Label{C:simple}
Suppose $\sigma\in\Aut{B}$ and that $D\subset B$ is
$\sigma$-modular
(cf.~Definition \ref{D:sigm-mod}).
If $J\sigma(D)\not\subset\sigma(J)$ holds for every
closed non-trivial ideal $J$ of $D$
then ${[D]}_\sigma\rtimes_{\sigma}\Z$ is a simple 
\Cast-algebra.
\end{COR}

\medskip

We need some notation 
for the proofs of Propositions \ref{P:a4.15} and \ref{P:a4.7}.
For $m\le n$ and any 
\Cast-sub\-al\-ge\-bra $D\subset B$ let:
\FRem{leftdot and rightbrace to add for number}
\begin{eqnarray}
D_{m,n}&:=&\sigma^{m-1}(D)+\cdots+\sigma^{n-1}(D)\,,\nonumber\\
% D_n&:=&D_{n,n}=\sigma^{n-1}(D), \nonumber\\
D_{-\infty,\infty}&:=&\overline{\bigcup_{k\in\N} D_{-k,k}}\,,
  \Label{Eq:a4.0}\\
D_{-\infty,n} \,:=\,
\overline{\bigcup_{k\in\N} D_{n-k,n}}\,, &\quad&\quad
D_{n,\infty} \,:=\,
\overline{\bigcup_{k\in\N} D_{n,n+k}}\nonumber\,.
\end{eqnarray}
It will be shown in Lemma \ref{L:a4.3} 
that these vector spaces are
\Cast-sub\-al\-ge\-bras of $B$ under the additional assumption 
\begin{equation}\Label{Eq:alpha-subset}
D\sigma^k(D)\subset\sigma^k(D) \qquad\text{for}\quad k\in\N.
\end{equation}

\begin{LEM}\Label{L:D-sigma-k-E}
Under the assumption (\ref{Eq:alpha-subset}) holds: 
\begin{Aufl}
\item $\sigma^{j}(D_{m,n})=D_{m+j,n+j}$   
for $j, m,n\in \Z$, $m\leq n$.
\item $D_{m,n}D_{i,j} \subset D_{\max(m,i),\max(n,j)}$
for $m\leq n$ and $i\leq j$, $i,j, m,n\in \Z$.
\end{Aufl}
\end{LEM}

\begin{proof} (i) is obvious from definition.

\Ad{(ii)} $\sigma^i(D)\sigma^j(D)=
\sigma^i \left( D\sigma^{j-i}(D) \right)$ 
for $i\leq j$.
Since $\left(\sigma^i(D)\sigma^j(D) \right)^*=
\sigma^j(D)\sigma^i(D)$, we get
$\sigma^i(D)\sigma^j(D)\subset \sigma^{\max(i,j)}(D)$
for all $i,j\in\Z$.

This implies (ii) by definition of $D_{m,n}$.
\end{proof}
\begin{LEM}\Label{L:a4.2}
  Let $E, F$ be 
        \Cast-sub\-al\-ge\-bras of $B$.  If
  $EF\subset F$ then the sum $E+F$ is a 
        \Cast-subalgebra of \(B\).
\end{LEM}
\begin{proof}
  $E+F$ is a *-subalgebra, $F$ is a closed ideal of
  $E+F$ and
  $E\cap F$ is a closed ideal in $E$, because
  $FE=F^*E^*=(EF)^*\subset F^*=F$. The map 
  $$
  \varphi\colon c+(E\cap F)\in E/(E\cap F) 
        \mapsto  c+F\in (E+F)/F
  $$
  is a well-defined *-iso\-mor\-phism from $E/(E\cap F)$
  onto $(E+F)/F$.  
This shows, that
  $(E+F)/F\subset \Norm{B,F}/F$ is a \Cast-algebra.
  Therefore the pre-image $E+F$ 
  of the quotient map is closed.
\end{proof}
\begin{LEM}\Label{L:a4.3}
Let $D\sigma^k(D)\subset\sigma^k(D)$ for $k\in\N$. Then, for
$m,n\in\Z$ with $m\le n\,$:
\begin{Aufl}
\item $D_{m,n}$ is a \Cast-sub\-al\-ge\-bra of $B$ and
  $D_{i,n}$ is an ideal of $D_{m,n}$ for $m\le i\le n$.
\item $D_{-\infty,\infty}$ is the smallest $\sigma$-invariant
  \Cast-sub\-al\-ge\-bra $[D]_\sigma$ of $B$ containing $D$,
  and is the inductive limit of the \Cast-algebras
  $ \left( D_{-k,k} \right) _{k\in\N}$.

        If, in addition, $D$ is of type~I, 
        then $[D]_\sigma$ is the 
        inductive limit of type~I algebras.

\item $D_{-\infty,n}$ is a \Cast-algebra and 
  satisfies 
        $\sigma^{-1} \left( D_{-\infty,n} \right) 
        \subset D_{-\infty,n}$.
\item $D_{n,\infty}$ is a closed ideal of 
  $[D]_\sigma$ and
  satisfies 
  $\sigma \left( D_{n,\infty} \right)  \subset D_{n,\infty}$.
\end{Aufl}
\end{LEM}
\begin{proof}
{(i)} follows from Lemma \ref{L:D-sigma-k-E} and
by induction over $n\ge m$ by applying Lemma \ref{L:a4.2}
to $D_{m,n+1}=D_{m,n}+\sigma^n(D)$.

\Ad{(ii)} 
It is obvious from the definition of 
$D_{-\infty,\infty}$ and (i)
that it is the
inductive limit of the sequence 
$D_{-1,1}\subset D_{-2,2}\subset\dots\,$
of \Cast-algebras.
The $\sigma$-invariance follows from 
Lemma \ref{L:D-sigma-k-E}(i).  
Clearly $D_{-\infty,\infty}$
is the smallest 
$\sigma$-invariant closed vector subspace of $B$
containing $D$.

By (i), $D_{-k,k}$ has an ideal decomposition series 
$$
\sigma^{k-1}(D)=D_{k,k}\subset D_{k-1,k}\subset \dots
\subset D_{k-j,k} \subset \dots\subset 
D_{-k+1,k}\subset D_{-k,k}\,.
$$
The intermediate factors are isomorphic to 
$D/\left( D\cap D_{2,j+1} \right)$,
which are of type~I if $D$ is type~I.
Thus, $D_{-k,k}$ is of type I if $D$ is of type I
in addition to (\ref{Eq:alpha-subset}).

\Ad{(iii),(iv)} 
In the same way one gets that $D_{-\infty,n}$ (respectively 
$D_{n,\infty}$) is the inductive limit of $D_{-k,n}$
(respectively of $D_{n,k}$) for $k\to\infty$.
Observe $\sigma^{-1} \left( D_{-\infty,n} \right) 
=D_{-\infty,n-1}\subset D_{-\infty,n}$ and
$\sigma \left( D_{n,\infty} \right) 
=D_{n+1,\infty}\subset D_{n,\infty}$.

$D_{n,\infty}$ is a closed ideal of $D_{-\infty,\infty}$
by Lemma \ref{L:D-sigma-k-E}(ii).
\end{proof}
\begin{LEM}\Label{L:a4.5}
If $D\sigma^k(D)\subset\sigma^k(D)$ for $k\in\N$, 
$D\cap\sigma(D)=0$, and if
$\sigma(D)$ is an essential ideal of 
$D+\sigma(D)$ then
$D_{j,n}$ is an essential ideal of $D_{m,n}$ 
for all $m\le j\le n$,
$j,n\in\Z$, $m=-\infty$ or $m\in\Z$.

If $I$ is a non-zero $\sigma$-invariant closed ideal of 
$[D]_\sigma$, then $J:=I\cap D$ is non-zero
and satisfies $J\sigma^k(D)\subset \sigma^k(D)$ for
$k\in \N$.
\end{LEM}
\begin{proof}
First we show that the ideal 
$D_{n,n}$ of $D_{1,n}$ is essential 
by induction:
$n=1$ is trivial and $n=2$ is true by assumption. 
Suppose it is true for
$n$.  Let $x\in D_{1,n+1}$ with $x\cdot D_{n+1,n+1}=0$.
$x$ can be written as the sum of $y\in D_{1,n}$ and 
$z\in D_{n+1,n+1}$.
$ay\in D_{n,n}$ and $az\in D_{n+1,n+1}$ for all 
$a\in D_{n,n}$, i.e.
$ax\in D_{n,n}+D_{n+1,n+1}=\sigma^{n-1}\left(D_{1,2}\right)$.
But
$D_{n+1,n+1}=\sigma^{n-1}\left(D_{2,2}\right)$ is essential in 
$\sigma^{n-1} \left( D_{1,2} \right) $ by assumption, 
thus $ax\cdot D_{n+1,n+1}=0$
implies $ax=0$. Then 
$ay=-az\in D_{n,n}\cap D_{n+1,n+1}=\{0\}$ shows 
$ay=az=0$, but $D_{n,n}$ was assumed essential in $D_{1,n}$ 
implying $y=0$ and $x=z$ and this again implies $x=0$.

By applying $\sigma^{1-m}$ to $D_{m,n}$, we see that
$D_{n,n}$ is essential in $D_{m,n}$.  
$D_{j,n}$ is
an ideal of $D_{m,n}$ by Lemma \ref{L:a4.3}(i).
$D_{j,n}$ is essential, because 
$D_{j,n}\supset D_{n,n}$.

It follows that $D_{j,n}$ is an essential ideal of 
$D_{-\infty,n}
=\mathrm{indlim}_{m\to -\infty} D_{m,n}$.
Indeed, $D_{j,n}$ is a closed ideal of $D_{-\infty,n}$, and
the natural *-homo\-mor\-phism from 
$D_{-\infty,n}$ to $\Mult{D_{j,n}}$
is isometric on 
$\bigcup_{m\le j} D_{m,n}\subset D_{-\infty,n}$,
thus is faithful on $D_{-\infty,n}$.

Since $D_{-\infty,\infty}$ is the inductive limit of
$D_{-k,k}$ there exists $k\in\N$ with $I\cap D_{-k,k}\neq\{0\}$
if $I$ is non-zero.
If $I$ is non-zero and $\sigma$-invariant, then there exists
$n=2k+1\in\N$ with non-zero intersection $D_{1,n}\cap I$.
Since $D_{n,n}$ is an essential ideal
of $D_{1,n}$, the intersection 
$D_{n,n}\cap I=\sigma^{n-1}(D\cap I)$ 
is non-zero.

Let $J:=I\cap D$ then $J\sigma^k(D)\subset \sigma^k(J)$
for $k\in \N$,
because
$J\sigma^k(D)\subset I$, $J\sigma^k(D)\subset \sigma^k(D)$
and $I\cap \sigma^k(D)=\sigma^k(J)$.
\end{proof}

\begin{LEM}\Label{L:direct-sum}
Suppose 
$D\sigma^k(D)\subset\sigma^k(D)$ for $k\in\N$ and
$D\cap\sigma(D)=\{0\}$.
If $\sigma(D)$ is essential in $D+\sigma(D)$ or if
$D\sigma(D)=\sigma(D)$,
then $D_{m,n}$ is as a vector space the direct sum of the
subspaces $D_{j,j}=\sigma^{j-1}(D)\cong D$ for
$m\le j\le n$.

In particular, 
$D_{m,n}=D_{m,j}+D_{j+1,n}$
and $D_{m,j}\cap D_{j+1,n}=0$ for $m\le j< n$. 
\end{LEM}
\begin{proof}
It suffices to check the case $m=1$ by induction over $n$, 
because
$D_{m,n}=\sigma^{m-1} \left( D_{1,n-m+1} \right)$.
For $n=1$ this is trivial and for $n=2$ it follows from
$D\cap\sigma(D)=\{0\}$.
Let $d_j\in D_{j,j}$ for $j=1,\dots,n+1$ with
$d_1+\dots+d_n+d_{n+1}=0$.  If $d_1+\dots+d_n=0$ we get by
induction assumption that $d_1=d_2=\dots=d_n=0$ and thus
$d_{n+1}=0$.
Suppose that $d_1+\dots+d_n\neq 0$.  Then $d_{n+1}\neq 0$.
If $\sigma(D)$ is essential in $D+\sigma(D)$, then 
$D_{n,n}=\sigma^{n-1}(D)$ is essential in $D_{1,n}$ by 
Lemma \ref{L:a4.5}.
Thus there is $e\in D_{n,n}$ with 
$f:=
\sigma^{1-n}\left(\left( d_1+\dots+d_n\right)\cdot e\right)
\neq 0$, and
$g:=\sigma^{1-n} \left( d_{n+1}e \right) \neq 0$ and $f=-g
\in D\cap\sigma(D)$ which contradicts $D\cap \sigma(D)=\{0\}$.

If $D\sigma(D)=\sigma(D)$ 
then there exists $e\in D_{n,n}$ with
$ed_{n+1}\neq 0$, because 
$$D_{n,n}\cdot D_{n+1,n+1}
=\sigma^{n-1} \left( D\sigma(D) \right) 
=\sigma^n(D)=D_{n+1,n+1}$$
and the latter implies that 
$D_{n,n}$ contains an approximate 
unit for $D_{n+1,n+1}$.
With $f$ and $g$ as above, we get again 
$f=-g\in D\cap\sigma(D)$.
\end{proof}

\begin{proof}[Proof of Proposition \ref{P:a4.15}]
By induction over $n=j-i$ we get
$\varphi \left( \sigma^{i-j}(f)g \right) 
=\sigma_1^{i-j} \left( \varphi(f) \right) \varphi(g)$
for $i\le j$, $f,g\in D$, indeed:

By assumption we find $f_1,g_1\in D$ with
$g=\sigma^{-1}\left(f_1\right)g_1$ because 
$D=\sigma^{-1} \left( D\sigma(D)\right)$.
The induction step uses
$$ 
\varphi\left(\sigma^{-n-1}(f)g\right) 
=
\varphi\left(
\sigma^{-1}\left(\sigma^{-n}(f)f_1\right)g_1
\right), 
$$
where $\sigma^{-n}(f)f_1$ is in $D$.

Thus
\begin{equation}\Label{Eq:varphi-product}
 \varphi \left( \sigma^{1-j} \left( d_id_j \right)  \right) 
=\sigma_1^{i-j} \left( \varphi \left( \sigma^{1-i}
 \left( d_i \right)  \right)  \right) 
\varphi \left( \sigma^{1-j} \left( d_j \right)  \right) 
\end{equation}
for $d_i\in D_{i,i}$, $d_j\in D_{j,j}$, 
and $i\le j$.

Let $d_j\in D_{j,j}$ and 
$\varphi_e \left( d_j \right) 
:=
\sigma_1^{j-1}\varphi\left(\sigma^{1-j}\left(d_j\right)\right) 
\in \left( D_1\right)_{j,j}$ 
for $-k\le j\le k$, and let
$$ 
\varphi_e \left( d_{-k}+d_{-k+1}+\dots+d_k \right) 
:=
\varphi_e \left(d_{-k}\right)+
\varphi_e \left(d_{-k+1}\right) + \dots
+\varphi_e \left( d_{k} \right).
$$
Then $\varphi_e$ is a well-defined linear map from 
$D_{-k,k}$ into $ \left( D_1 \right) _{-k,k}$
by Lemma \ref{L:direct-sum}, 
because $D\cap\sigma(D)=\emptyset$.
%$\varphi_e \left( d^* \right) =
%\varphi_e \left( d \right) ^*$ for 
%$d\in D_{-k,k}\,$, because 
$\varphi_e$ is a *-homo\-mor\-phism, 
because
$\varphi_e ( d_id_j ) =\varphi_e ( d_i ) \varphi_e( d_j )$
by Equation (\ref{Eq:varphi-product}) and
and $\varphi_e \left( d_j^* \right) =
\varphi_e \left( d_j \right)^*\,$.

Since $[D]_\sigma= D_{-\infty,\infty}$
and 
$\left[D_1\right]_{\sigma_1}= 
\left( D_1 \right) _{-\infty,\infty}$
are the inductive limits of the \Cast-algebras 
$D_{-k,k}$
respectively of $\left( D_1 \right) _{-k,k}$, we get that
$\varphi_e$ uniquely extends to a *-homo\-mor\-phism from
$[D]_\sigma$ to $\left[D_1\right]_{\sigma_1}$.  
$\varphi_e$ satisfies 
$\varphi_e\sigma=\sigma_1\varphi_e$ because
$\varphi_e \left(\sigma^n(d) \right) 
=\sigma_1^n \left(\varphi(d)\right)$
for $n\in\Z$, $d\in D$.

If $\varphi$ is a 
*-mono\-mor\-phism from $D$ into $D_1$, then
the kernel of $\varphi_e$ is a 
$\sigma$-invariant closed ideal
$I$ of $\left[D\right]_{\sigma}$ with $I\cap D=\{0\}$.
It implies $I=\{0\}$ by Lemma \ref{L:a4.5} 
if $D$ is $\sigma$-modular,
i.e.~if $D$ satisfies ($\beta$) in addition.

In particular, $\varphi_e$ is a 
*-mono\-mor\-phism with dense image in
$ \left( D_1 \right) _{-\infty,\infty}=
\left[D_1\right]_{\sigma_1}$ 
if $\varphi$ is an iso\-mor\-phism from $D$ onto $D_1$.
\end{proof}

\begin{REM}\Label{R:type-I-endo}
Under the assumption of ($\alpha$) and ($\gamma$),
$ \{ D_{-k,1} \} $ is a decomposition series of ideals of
$D_{-\infty,1}$ 
with intermediate factors isomorphic to $D$,
and $\sigma^{-1}$ defines an endo\-mor\-phism of 
$D_{-\infty,1}$
such that $[D]_\sigma$ is the inductive limit of 
$\sigma^{-1}\colon D_{-\infty,1}\hookrightarrow 
D_{-\infty,1}\,$,
i.e.\ $ \left( [D]_\sigma ,\sigma^{-1} \right) $ 
is the natural 
dynamical system corresponding to the endo\-mor\-phism
$\sigma^{-1}\colon D_{-\infty,1}
\hookrightarrow D_{-\infty,1}$.
In particular, $D_{-\infty,1}$ is of type~I if 
$A$ is of type~I.

Indeed, the \Cast-algebras $D_{-k,1}$ are closed ideals of 
$D_{-n,1}$ for $n\ge k$ by Lemma \ref{L:a4.3}(i).  
Since 
$D_{-\infty,1}$ is the inductive limit of $D_{-n,1}$ 
we get
that $D_{-k,1}$ is also an ideal of $D_{-\infty,1}$.
The quotient $D_{-(k+1),1}/ D_{-k,1}$ is isomorphic to $D$
by Lemma \ref{L:direct-sum}.
$\sigma^{-1}$ is an endo\-mor\-phism of 
$D_{-\infty,1}$ because
$\sigma^{-1} \left( D_{-\infty,1} \right) 
=D_{-\infty,0}
\subset D_{-\infty,1}$.  
Clearly, the smallest $\sigma$-invariant
\Cast-sub\-al\-ge\-bra of $B$ containing $D_{-\infty,1}$ is
the inductive limit of $D_{-\infty,1}
\stackrel{\sigma^{-1}}{\longrightarrow}
D_{-\infty,1}
\stackrel{\sigma^{-1}}{\longrightarrow}D_{-\infty,1}
\stackrel{\sigma^{-1}}{\longrightarrow}\dots$ 
and is the same as $D_{-\infty,\infty}=[D]_\sigma$.
And $\sigma^{-1}$ is the unique auto\-mor\-phism of 
$[D]_\sigma$ that
induces $\sigma^{-1}$ on $D_{-\infty,1}$.
\end{REM}

\begin{LEM}\Label{L:secdual}
Suppose $\sigma\in \Aut{B}$ and that $D\subset B$ is 
$\sigma$-modular.
Let $e_n$ denote the unit of 
$\wcl{D_{n,n}}\subset B^{**}$, then
\begin{Aufl}
\item $e_n\ge e_{n+1}$, 
$\sigma \left( e_n \right) =e_{n+1}$,
\item $ \left( [D]_\sigma \right)^{**}\cong
  \wcl{D_{-\infty,\infty}}\subset B^{**}$,
\item 
        $e_n$ is the unit of 
        $\left( D_{n,\infty} \right) ^{**}$
  and is in the center of $ \left( [D]_\sigma \right) ^{**}$,
\item $\left\| a(e_{-n}-e_{n+1}) \right\| =\| a\|$
for $a\in D_{-k,k}$ and $k,n\in \N$ with $k\leq n$.
\end{Aufl}
\end{LEM}
Here $\wcl{X}$ denotes the ultra-weak closure of a 
subspace $X$ of 
the $W^\ast$-algebra $B^{**}$. $\wcl{X}$ is
naturally isomorphic to $X^{**}$ if $X$ is a closed subspace 
of $B$.  We write also $\sigma$ for
the second conjugate $\sigma^{**}$ of $\sigma$.
\FRem{check}
\begin{proof}
\Ad{(i)}
$\sigma\left(e_n\right)=e_{n+1}$, because 
$\sigma\left(D_{n,n}\right)=D_{n+1,n+1}$ and
$\sigma^{**}$ is weakly continuous on $B^{**}$.
Since $D\sigma(D)=D$, $D_{n,n}$ contains an approximate unit
for $D_{n+1,n+1}$, which implies $e_{n+1}\le e_n$.

\Ad{(ii)}
(ii) follows from Lemma \ref{L:a4.3}(ii).

\Ad{(iii)}
$e_n$ is in $\wcl{D_{n,\infty}}$.
If $a\in D_{k,k}$ and $k\ge n$, then $e_k\le e_n$,
$a\in e_k B^{**} e_k$, and $e_nae_n=a$.  
Thus $e_nae_n=a$ for $a\in D_{n,k}$, $k\ge n$, which
implies $e_nbe_n=b$ for $b\in \wcl{D_{n,\infty}}\,$,
i.e.\ $e_n$ is the unit of $\wcl{D_{n,\infty}}\,$.

$D_{n,\infty}\,$ is a closed ideal of
$[D]_\sigma\subset B$
by Remark \ref{R:a4.1}(ii) and Lemma \ref{L:a4.3}(iv),
and $e_n$ is the (open) support projection of
$D_{n,\infty}\,$ in 
$\wcl{D_{-\infty,\infty}}\cong 
\left([D]_\sigma\right)^{**}\,$.
Thus $e_n$ is in the
center of $\left([D]_\sigma\right)^{**}\,$.

%If $a\in D_{k,k}$ and $k\le n$,
%then $D_{n,n}$ is an ideal of $D_{k,n}\supset D_{k,k}$
%by Lemma \ref{L:a4.3}(i).  Thus $e_n$ is in the center of
%$\left(D_{k,n}\right)^{**}\cong\wcl{D_{k,n}}\subset B^{**}$, 
%and
%$e_n$ commutes with $a\in D_{k,n}$.
%It follows $e_na=ae_n$ for every $k\in\N$ and
%$a\in D_{-k,k}$, and that $e_n$ commutes with
%all $b$ in the ultra-weak closure of 
%$\bigcup_{k>0} D_{-k,k}$.

\Ad{(iv)} 
The map
$\rho\colon a\mapsto a\left(e_{-k}-e_{k+1}\right)$ 
is a *-homo\-mor\-phism
from $D_{-k,k}$ into 
$\wcl{D_{-k,k+1}} \cong \left(D_{-k,k+1}\right)^{**}$.
If $\rho(a)=0$ then $a=ae_{k+1}$.
If $ \{ b_\tau \} $ is an approximate unit of $D$ 
then $\sigma^k \left( b_\tau \right) $ tends weakly to 
$e_{k+1}$.
By Hahn-Banach separation this implies that $a$ is in the
closed span of $D_{-k,k}\sigma^k(D)\subset\sigma^k(D)$.  
Since 
$D\cap \sigma(D)=\{0\}$, we get 
$D_{-k,k}\cap\sigma^k(D)=\{0\}$
by Lemma \ref{L:direct-sum}. Hence $a=0$
and $\rho$ is isometric on $D_{-k,k}$. 
Now 
$\left\| a(e_{-n}-e_{n+1})\right\| = \| a\|$ for 
$a\in D_{-k,k}$
follows from 
$(e_{-k}-e_{k+1})(e_{-n}-e_{n+1})=e_{-k}-e_{k+1}$
for $0\leq k\leq n$.
\end{proof}

We now return to the Proposition \ref{P:a4.7} 
stated at the beginning
of this section:
\begin{proof}[Proof of Proposition \ref{P:a4.7}]

If $I$ is a closed ideal of $[D]_\sigma\rtimes \Z$
then $I_1:=[D]_\sigma\cap I$ is a $\sigma$-invariant
closed ideal of $[D]_\sigma$. The closed ideal
$J:=I\cap D=I_1\cap D$ satisfies 
$J\sigma(D)\subset \sigma(J)$
by Lemma \ref{L:a4.5}.

Suppose that 
$\varrho\colon [D]_\sigma\rtimes_\sigma\Z\to\LOp{H}$ is a
*-representation of $[D]_\sigma\rtimes_\sigma\Z$ such that
$\varrho|D$ is faithful, and let $I:=\ker(\varrho)$.
Then $I_1:=[D]_\sigma \cap I$ satisfies $I_1\cap D=\{0\}$.

By Lemma \ref{L:a4.5} it follows that $I_1=\{0\}$.  
Now Lemma \ref{L:secdual} shows that 
Proposition \ref{P:e-n-in-second} applies.
Hence, $\varrho$ is faithful
on $[D]_\sigma\rtimes_\sigma\Z$ 
by Proposition \ref{P:e-n-in-second}.

Equivalently: Every non-zero
closed ideal $I$ of 
$[D]_\sigma\rtimes_\sigma\Z$ 
has non-zero intersection $J=I\cap D$
with $D$.
\end{proof}

\subsection{Semi-invariant ideals}
\Label{ssec:semi-inv-ideals}
\begin{DEF}\Label{D:s-inv/cancel}
We say that a closed ideal $J$ of $D$ is 
\emph{$(D,\sigma)$-semi-invariant},
if $J\sigma(D)\subset \sigma(J)$.
$J$ has \emph{$(D,\sigma)$-cancellation}
if $a,b\in D$ and 
$(a+\sigma(b))\sigma(D)\subset \sigma(J)$ implies 
that $a\in J$. 
\end{DEF}
If $J$ is semi-invariant and has
cancellation,
then $(a+\sigma(b))\sigma(D)\subset \sigma(J)$  implies 
that $a,b\in J$, because $bD\subset J$ implies 
$b\in J$.

\begin{LEM}\Label{L:invIdeal}
Suppose that $D\subset B$ and $\sigma\in \Aut{B}$ satisfy
conditions ($\alpha$) and ($\gamma$) of 
Definition \ref{D:sigm-mod},
and that $J$ is a closed ideal of $D$ with 
$J\sigma(D)\subset \sigma(J)$.
\begin{Aufl}
\item 
  $J\sigma^k(D)\subset \sigma^k(J)$ for $k\in \N$, and
  $[J]_\sigma$ is a closed $\sigma$-invariant ideal of 
        $[D]_\sigma$.
\item
  $\left(D+\sigma(D)\right)\cap [J]_\sigma=J+\sigma(J)\,$.
\item
  The system 
  $D_1:=\left(D+[J]_\sigma\right)/[J]_\sigma \subset 
   B_1:=[D]_\sigma/[J]_\sigma\,$,
  $\sigma_1\in \Aut{B_1}$ with 
  $\sigma_1\left(a+[J]_\sigma\right):=\sigma(a)+[J]_\sigma$
  satisfies again ($\alpha$) and ($\gamma$) of
  Definition \ref{D:sigm-mod}, 
        i.e.~$D_1\sigma_1\left(D_1\right)=D_1$
  and  $D_1\cap \sigma_1\left(D_1\right)=\{ 0\}$.
\item
  $D_1$ and $\sigma_1$ satisfy also condition ($\beta$), 
  i.e.\
$\Ann{D_1+\sigma_1\left(D_1\right),\sigma_1\left(D_1\right)}
=\{ 0\}\,$, 
  if and only if
  $J$  has $(D,\sigma)$-cancellation.
\item If $J$  has $(D,\sigma)$-cancellation, then
$K=[J]_\sigma$ for every $\sigma$-invariant closed
ideal $K$ of $[D]_\sigma$ with $K\cap D=J$.
\end{Aufl}
\end{LEM}

\begin{proof} First note that
$J\sigma^k(J)\subset\sigma^k(J)$ and
$J_{m,n}$ (respectively
$J_{-\infty,\infty}$) is a closed ideal of 
of $D_{m,n}$ 
(respectively of $D_{-\infty,\infty}$),
if $J$ is a closed ideal of $D$ with 
$J\sigma^k(D)\subset\sigma^k(J)$
for all $k\in\N$. 
This follows from Lemma \ref{L:a4.5}(i) and
$(\sigma^i(D)\sigma^j(J))^*=\sigma^j(J)\sigma^i(D)
\subset\sigma^{\max(i,j)}(J)$
for $i,j\in\Z$:
If $i\le j$, then
$\sigma^{j-i}(J)D= \sigma^{j-i}(JD)D 
  = \sigma^{j-i}(J)\sigma^{j-i}(D)D
  \subset \sigma^{j-i}(J)\sigma^{j-i}(D) =\sigma^{j-i}(J)\,
$,
and, if $i\ge j$, then
%\begin{eqnarray*}
$$
  \sigma^j(J)\sigma^i(D)=\sigma^j\left(J\sigma^{i-j}(D)\right)
  \subset\sigma^j\left(\sigma^{i-j}(J)\right)=\sigma^i(J)\, 
.
$$
%\end{eqnarray*}

\Ad{(i)} 
$J\sigma^{k+1}(D)=
J\sigma^k(D)\sigma^{k+1}(D)\subset \sigma^k(J\sigma(D))$
by ($\alpha$) and induction assumption. 
The right side is contained in
$\sigma^{k+1}(J)$. Thus $[J]_\sigma=J_{-\infty,\infty}$
is a closed ideal of $[D]_\sigma=D_{-\infty,\infty}$.

\Ad{(ii)} Above we have seen that
$J_{-k,k}$ is a closed ideal of $D_{-k,k}$ for $k\in \N$. 
$[J]_\sigma=J_{-\infty,\infty}$ is
the inductive limit of $\{ J_{-k,k} \}_{k\in \N}$ 
by Lemma \ref{L:a4.3}(ii).
It follows 
$\textrm{dist}\left(a, [J]_\sigma\right)=
\lim_k \textrm{dist}\left(a, J_{-k,k}\right)$ 
and 
$\textrm{dist}\left(a, J_{-k,k}\right)=
\textrm{dist}
\left(a,\left(D+\sigma(D)\right)\cap J_{-k,k}\right)$ 
for $a\in D+\sigma(D)$.
But 
$\left( D+\sigma(D) \right) \cap J_{-k,k}=J+\sigma(J)$ 
by Lemma \ref{L:direct-sum}.

\Ad{(iii)}
Let $\pi\colon [D]_\sigma\to B_1=[D]_\sigma/[J]_\sigma$
the quotient map. Then $\pi\circ\sigma=\sigma_1\circ\pi$
and $\pi(D)=D_1$. 
Thus, 
$D_1\sigma_1\left(D_1\right)=\pi\left(D\sigma(D)\right)=D_1$
by ($\alpha$). If $d,e\in D$, 
$d_1=\pi (d)$ and $e_1=\pi (e)$ 
satisfy $d_1=\sigma_1(e_1)$, then 
$d-\sigma(e)\in [J]_\sigma$.
There are $f,g\in J$ with $d-\sigma(e)=f+\sigma(g)$, 
by part (ii).
$d-f=\sigma(e+g)$ implies $d=f$ and $d_1=\pi(f)=0$ by 
($\gamma$)
for $(D,\sigma)$.

\Ad{(iv)} 
$\sigma(D)\cap [J]_\sigma 
=\sigma(D)\cap\left( J+\sigma(J) \right)=\sigma(J)$
by (ii) and property ($\gamma$).
Let $d,e\in D$.  
$\left( \pi(d)+\sigma_1(\pi(e))\right)\sigma_1\left(D_1\right)
=\{ 0\}$
is equivalent to 
$\left(a+\sigma(e)\right)\sigma(D)
\subset \sigma(D)\cap [J]_\sigma$.
Thus, 
$\Ann{D_1+\sigma_1\left(D_1\right),\sigma_1\left(D_1\right)}
=\{ 0\}$ 
if and only if
$J$ has $(D,\sigma)$-cancellation.

\Ad{(v)} Clearly $[J]_\sigma\subset K$ if
$K$ is a $\sigma$-invariant closed ideal with $J=D\cap K$.
Then $\pi (K)$ is a closed ideal of 
$B_1 = [D]_\sigma/[J]_\sigma$ with $\pi(K)\cap \pi(D)=\{ 0\}$.
Since $D_1\subset B_1$ is $\sigma_1$-modular,
it follows that $\pi(K)=\{ 0\}$ by Lemma \ref{L:a4.5},
i.e.\ $K=[J]_\sigma$.
\end{proof}

\begin{PRP}\Label{P:ideal}
Suppose that $\sigma\in \Aut{B}$ and $D\subset B$
is $\sigma$-modular in the sense of Definition
\ref{D:sigm-mod}.
%
%Suppose that $\sigma\in \Aut{B}$ and that $D\subset B$
%is $\sigma$-modular (cf.~Definition \ref{D:sigm-mod}).
%
If $I$ is a closed ideal of $[D]_\sigma \rtimes_\sigma \Z$
such that $J:=D\cap I$ has $(D,\sigma)$-cancellation
(cf.\ Definition  \ref{D:s-inv/cancel}),
then $I$ is the natural image of
$[J]_\sigma \rtimes_\sigma \Z$ in 
$[D]_\sigma \rtimes_\sigma \Z$.
\end{PRP}
\begin{proof} 
$K:=[D]_\sigma \cap I$ is a $\sigma$-invariant ideal
of $[D]_\sigma$. 
Thus $J=D\cap I$ satisfies $J\sigma(D)\subset \sigma(J)$
and $[J]_\sigma$ is a $\sigma$-invariant closed ideal of
$[D]_\sigma$. Let  
$D_1:=\left( D+[J]_\sigma\right)/[J]_\sigma$,
$B_1:=[D]_\sigma/[J]_\sigma\,$, and
$\sigma_1\in \Aut{B_1}$ defined by 
$\sigma_1\left(a+[J]_\sigma\right):=\sigma(a)+[J]_\sigma$.
The natural image $K$ of 
$[J]_\sigma \rtimes_\sigma \Z$ in 
$[D]_\sigma \rtimes_\sigma \Z$
is contained in $I$ 
and is the kernel of the natural epi\-mor\-phism
$\pi\rtimes \Z$ from $[D]_\sigma \rtimes_\sigma \Z$ onto 
$[D_1]_{\sigma_1}\rtimes _{\sigma_1} \Z$ 
(cf.~Remark \ref{Rcross-prod}(vi)).
Let $I_1:=(\pi\rtimes \Z)(I)$. $I_1$ is a closed ideal of 
$[D_1]_{\sigma_1}\rtimes _{\sigma_1} \Z$ with 
$I_1\cap D_1=\{ 0\}\,$.
$\,\, D_1\subset B_1$ is $\sigma_1$-modular by 
Lemma \ref{L:invIdeal}(iv),
because $J$ is $(D,\sigma)$-semi-invariant and has 
$(D,\sigma)$-cancellation.
Thus, $I_1=\{ 0\}$ by Proposition \ref{P:a4.7}, i.e.~$I=K$.
\end{proof}

\FRem{next implicitly contained in some Lemma ?}
\begin{REM}\Label{R:ideal} 
In the above proof we have also shown:

If $K$ is a $\sigma$-invariant
closed ideal of $D_{-\infty,\infty}$ such that
$J:=D\cap K$ has $(D,\sigma)$-cancellation,
then 
$K= J_{-\infty,\infty}$ and every *-representation
$$\varrho \colon 
D_{-\infty,\infty}\rtimes_\sigma\Z\to\LOp{H}$$
with 
$\ker \left( \varrho |D \right) =D\cap K$ has kernel
$K\rtimes_\sigma\Z$.

Thus, $\varrho$ defines a faithful representation of
$\left( D_{-\infty,\infty}/K \right) \rtimes_{[\sigma]}\Z$,
where $[\sigma]$ is induced by $\sigma$ on 
$D_{-\infty,\infty}/K$.
\end{REM}
\begin{COR}\Label{C:ideals}
Suppose that $\sigma\in \Aut{B}$ and $D\subset B$
is $\sigma$-modular (cf.\ Definition  \ref{D:sigm-mod}).

If every closed ideal $J$ of $D$ with 
$J\sigma(D)\subset\sigma(J)$
has $(D,\sigma)$-cancellation 
(cf.\ Definition  \ref{D:s-inv/cancel}),
then the map
$$
I\in \latI{[D]_\sigma \rtimes_\sigma\Z}\mapsto
I\cap D \in \latI{D},
$$
is a lattice iso\-mor\-phism from
$\latI{[D]_\sigma \rtimes_\sigma \Z}\cong
\latO{\Prim{[D]_\sigma \rtimes_\sigma \Z}}$
onto the lattice $\latI{D}^\sigma$
of closed ideals 
$J$ in $D$ with
$J\sigma(D)\subset\sigma(J)$.

The inverse lattice isomorphism is given by
$\,J\mapsto 
%[J]_\sigma \mapsto 
[J]_\sigma \rtimes_\sigma \Z\,$.

$D$ is a regular \Cast-subalgebra of
$[D]_\sigma \rtimes_\sigma \Z$ in the sense of
Definition \ref{D:regular}.
\end{COR}
(Here we identify 
$[J]_\sigma \rtimes_\sigma \Z$ with its natural image
in $[D]_\sigma \rtimes_\sigma \Z$.)
\begin{proof} 
We show that
$$J\mapsto [J]_\sigma \mapsto [J]_\sigma \rtimes_\sigma \Z$$
define a lattice iso\-mor\-phisms from the lattice 
$\latI{D}^\sigma$
of $(D,\sigma)$-semi-invariant closed ideals of $D$
onto the lattice $\latI{[D]_\sigma}^\sigma$ of 
$\sigma$-invariant
closed ideals of $[D]_\sigma\,$, and from 
$\latI{[D]_\sigma}^\sigma$
onto $\latI{[D]_\sigma \rtimes_\sigma \Z}$, and that
the map 
$$I\in \latI{[D]_\sigma \rtimes_\sigma \Z}\mapsto D\cap I
\in \latI{D}^\sigma $$
is the inverse of 
$J\in \latI{D}^\sigma \mapsto [J]_\sigma \rtimes_\sigma \Z$.

If $J$ is a closed ideal with 
$J\sigma(D)\subset \sigma(J)$, then
$[J]_\sigma$ is a 
$\sigma$-invariant closed ideal of $[D]_\sigma$
by Lemma \ref{L:invIdeal}(i).
Thus $[J]_\sigma \rtimes_\sigma \Z$ is a closed ideal of
$[D]_\sigma \rtimes_\sigma \Z$.
Clearly, the maps 
$J\in \latI{D}^\sigma \mapsto [J]_\sigma$ and 
$K\in \latI{[D]_\sigma}^\sigma\mapsto K\rtimes _\sigma \Z$
are order preserving (i.e.~``$\subset$''-preserving).
$J\mapsto[J]_\sigma$ is injective, because by Lemma
\ref{L:invIdeal}(ii) and ($\gamma$) of Definition 
\ref{Eq:alpha-subset}:
$J\subset[J]_\sigma\cap D\subset \left(J+\sigma(J)\right)
\cap D\subset J.$

If $K$ is a $\sigma$-invariant closed ideal of $[D]_\sigma$,
then $K=[D]_\sigma \cap 
\left(K\rtimes_\sigma \Z\right)$ by 
Remark \ref{Rcross-prod}(iii). 
In particular, $K\mapsto K\rtimes _\sigma \Z$
is an injective map from the lattice of
$\sigma$-invariant closed ideals of $[D]_\sigma$
into the lattice of closed ideals of $[D]_\sigma\rtimes \Z$.

If $I$ is a closed ideal of 
$[D]_\sigma \rtimes_\sigma \Z$,
then $J:=D\cap I$ is a closed ideal with 
$J\sigma(D)\subset \sigma(J)$. 
By Proposition \ref{P:ideal}, 
$I=[D\cap I]_\sigma \rtimes _\sigma \Z$, because
$J$ has $(D,\sigma)$-cancellation
(by assumption).
 
Hence, 
$I\in \latI{[D]_\sigma \rtimes_\sigma \Z}\mapsto 
D\cap I\in \latI{D}^\sigma$
is an order iso\-mor\-phism which
is invertible with inverse map 
$J\mapsto [J]_\sigma \rtimes \Z$.
It follows that $J\mapsto [J]_\sigma$ is injective.
If the superposition of two injective maps is surjective, 
then
the maps are both invertible.

If $I_1,I_2$ are closed ideals of 
$[D]_\sigma \rtimes \Z$, then $J_1:=I_1\cap D$,
$J_2:=I_2\cap D$ and $J_3:=J_1+J_2$ are 
$(D,\sigma)$-semi-invariant.

$I_1+I_2$ is the l.u.b.\ of $\{ I_1,I_2\}$ in
the lattice $\latI{[D]_\sigma \rtimes \Z}$,
and $J_3$ is the l.u.b.\ of 
$\{ J_1,J_2\}$ in
the lattice $\latI{D}^{\sigma}$. 
Thus $D\cap (I_1+I_2)=(D\cap I_1)+(D\cap I_2)$.
It follows that $D$ is a regular
\Cast-subalgebra of $[D]_\sigma\rtimes_\sigma \Z$.
%
%\FRem{next proof with $h$ !}
%
%Let $h\colon D\to \Mult{D}$ be defined by
%$h(d)e:=\sigma^{-1}(d)e$.
%We have seen in Remark \ref{R:Omega-from-H0}
%that the lattice of closed ideals
%$J$ of $D$ with $h(J)D\subset J$ is closed with respect to 
%closures of  sums (= l.u.b.) of families of 
%closed ideals of $D$.
%Since $I\mapsto D\cap I$ is a lattice iso\-mor\-phism, 
%it follows
%that this map must respect sums of ideals, i.e.~that 
%$D\cap(I_1+I_2)=(D\cap I_1)+(D\cap I_2)$.
%Hence $D$ is a regular \Cast-subalgebra of 
%$[D]_\sigma\rtimes_\sigma \Z$.
\end{proof}

\begin{REM}\Label{R:h-from-sigma} 
Suppose that $D\subset B$ and $\sigma\in \Aut{B}$
satisfies condition ($\alpha$). Define a 
natural *-homo\-mor\-phism $h\colon D\to\Mult{D}$  by 
$h(a)b:=\sigma^{-1}(a)b$.
Clearly, $h$ is non-degenerate (i.e.~$h(D)D=D$) and 
the kernel of $h$ is equal to 
$D\cap \Ann{D+\sigma(D),\sigma(D)}$.

Then the conditions
($\beta$) and ($\gamma$), i.e.~
$\Ann{D+\sigma(D),\sigma(D)}=\{ 0\} $
and $D\cap \sigma (D)=\{ 0\}$, 
are equivalent to $h^{-1}\left( h(D)\cap D \right) =0$, i.e.,
$h\colon D\to \Mult{D}$ is faithful and $h(D)\cap D=\{ 0\}$.

(Indeed: 
$h^{-1}(h(D)\cap D) =\{ a\in D\,\colon\,\; h(a)\in D\}$
is the set of all $a\in D$ with the property that there is
$b\in D$ with 
$\sigma^{-1}\left((a-\sigma(b))\sigma(D)\right)
=(h(a)-b)D=\{0\}$,  
i.e.~with $a+\sigma(-b)\in \Ann{D+\sigma(D),\sigma(D)}$.)

A closed ideal $J$ of $D$ satisfies 
$J\sigma(D)\subset \sigma(J)$
if and only if $h(J)D\subset J$, i.e.~, 
$h(J)\subset \Mult{D,J}$.

$J$ has $(D,\sigma)$-cancellation 
in the sense of Definition \ref{D:s-inv/cancel},
if and only if, the existence of $b\in D$ with
$\left(\sigma^{-1}(a)+b\right)D\subset J$ implies $a\in J$.
Thus,  $J$ has $(D,\sigma)$-cancellation,
if and only if,
$$h^{-1}\left(h(D)\cap\left(D+\Mult{D,J}\right)\right)
\subset J\,.$$

Hence, a closed ideal 
$J$ of $D$ is $(D,\sigma)$-semi-invariant and has
$(D,\sigma)$-cancellation if and only if
$h(D)\cap \left(D+\Mult{D,J}\right) =h(J)$.
\end{REM}

\begin{COR}\Label{C:mod-to-CP}
Suppose that $\sigma\in\Aut{B}$ and that $D\subset B$
is $\sigma$-modular in the sense of 
Definition \ref{D:sigm-mod}. 
Let $h\colon D\to\Mult{D}$ denote the 
non-degenerate *-mono\-mor\-phism given by 
$h(a)b:=\sigma^ {-1}(a)b$ for $a,b\in D$.  

Then 
the hereditary \Cast-sub\-al\-ge\-bra of 
$[D]_\sigma\rtimes_\sigma\Z$ generated by 
$D$ is full in $[D]_\sigma\rtimes_\sigma\Z$ 
and is isomorphic
to the \Name{Cuntz--Pimsner} algebra $\OOO{\Hm{D,h}}$.

If every closed ideal $J$ of $D$ with 
$h(J)D\subset J$ satisfies
$$h(J)=h(D)\cap \left(D+\Mult{D,J}\right),$$ 
then $I\mapsto J:=D\cap I$ defines an isomorphism
from 
$\latI{\OOO{\Hm{D,h}}}$ 
onto the lattice 
$$
\left\{ 
J\in \latI{D}\fdg  h(J)D\subset J 
\right\}
.$$
\end{COR}
\begin{proof}
By Remark \ref{R:h-from-sigma} and Proposition 
\ref{P:Pims-alg},
there exists a *-mono\-mor\-phism $\varphi$
from $D$ into 
$\OOO{\Hm{D,h}}
\subset Q^s\left(\Mult{D}\right)$
and a unitary $U$ in the stable corona 
$Q^s\left(\Mult{D}\right):=\Mult{\Mult{D}\otimes \K}/
\left(\Mult{D}\otimes \K\right)$
of $\Mult{D}$
such that 
$U \varphi(D)$ generates $\OOO{\Hm{D,h}}$
and that $ \varphi \left( \sigma^{-1}(a)b \right) =
U^*\varphi(a)U\varphi(b)$ for all
$a,b\in D$.  
\FRem{notation $E$ OK? compare proof Prop.\ref{P:Pims-alg}}
Moreover, the \Cast-algebra $E$ generated
by $\bigcup _{n\in \N} U^{-n}\varphi(D)$ is
naturally isomorphic to 
$[\varphi(D)]_{\sigma_1}\rtimes_{\sigma_1}\Z$
for $\sigma_1(b):=UbU^*$, 
and $\OOO{\Hm{D,h}}$
is the hereditary \Cast-subalgebra 
of $E$ generated by 
$\varphi(D)$.

By Proposition \ref{P:a4.15}
it follows that $\varphi$ extends to a *-mono\-mor\-phism
$\varphi_e$ from $[D]_\sigma$ into 
$E\subset Q^s\left(\Mult{D}\right)$
with 
$\varphi_e\left(\sigma(a)\right)=U\varphi_e(a)U^*$
for $a\in [D]_\sigma$.
By Remark \ref{Rcross-prod}(vi)
there is a *-homo\-mor\-phism $\varrho$ 
from $[D]_\sigma\times_\sigma \Z$ onto $E$
with $\varrho|[D]_\sigma=\varphi_e$. 
$\varrho$ is an iso\-mor\-phism 
by Proposition \ref{P:a4.7} 
because $\varrho|D$
is faithful. $\varrho$ maps the full 
hereditary \Cast-subalgebra
of $[D]_\sigma\times_\sigma \Z$ generated by
$D$ onto $\OOO{\Hm{D,h}}$,
thus, the restriction of $\varrho$
to $D([D]_\sigma\times_\sigma \Z)D$ 
is an isomorphism onto $\OOO{\Hm{D,h}}=\varphi(D)E\varphi(D)$.

The lattice iso\-mor\-phisms follow from 
Corollary \ref{C:ideals}
by Remark \ref{R:h-from-sigma}. 
\end{proof}

\begin{COR}
Suppose that $D\subset B$ and $\sigma\in\Aut{B}$ satisfy
$D\sigma(D)=\sigma(D)$. 
If for the
(non-degenerate) *-homo\-mor\-phism 
$h\colon D\to\Mult{D}$ defined by
$h(a)b:=\sigma^{-1}(a)b$ holds 
$h^{-1}\left( h(D)\cap D \right)=\{ 0\}$ and 
$h(J)D\not\subset J$ for every non-trivial closed ideal 
$J$ of $D$,
then $[D]_\sigma\rtimes_\sigma\Z$ 
is simple and contains a copy of  the
\Name{Cuntz--Pimsner} algebra 
$\OOO{\Hm{D,h}}$ as a hereditary 
\Cast-subalgebra.

If, in addition, 
$D$ is $\sigma$-unital and stable, then 
$[D]_\sigma\rtimes_\sigma\Z$ 
is  isomorphic to $\OOO{\Hm{D,h}}$.
\end{COR}

\subsection{$(D,\sigma)$ and  $\Hm{A,h}$}
Here we translate the
results on
$\sigma$-modular algebras $D$ 
in subsections \ref{ssec:semi-inv-ideals} and
\ref{ssec:id-of-crossed} 
to the situation of our
special case of Cuntz--Pimsner algebras.

\begin{NOR} \Label{Nr:a4.16}
We consider in this section the following 
situation as an example of the previous section:\\
Suppose that $A$ is a \Cast-algebra and 
$h\colon A\to\Mult{A}$ a *-homo\-mor\-phism with
\begin{Aufl}
\item[(ND)] $h$ is \emph{non-degenerate}, i.e.~$h(A)A=A$,
\item[(GP)] $h(A)$ is in \emph{general position}, 
i.e.~$h(A)\cap A=0$,
\item[(FF)] $h$ is \emph{faithful}, i.e.~$\ker(h)=\{ 0\}$.
\end{Aufl}

Then $h(A)+A$ is a \Cast-sub\-al\-ge\-bra of $\Mult{A}$ and 
$A$ is an
essential ideal of $h(A)+A$, 
because $A$ is essential in $\Mult{A}$.

Since $h$ is non-degenerate, 
$h$ extends uniquely to a *-homo\-mor\-phism 
$\Mult{h}\colon \Mult{A}\to\Mult{A}$ with 
$\Mult{h}(a)h(b)c=h(ab)c$
for $a\in \Mult{A}$ and $b,c\in A$.
We denote $\Mult{h}$ also by $h$ to keep notation simple.
It follows from properties (ND) and (FF) that 
$h\colon \Mult{A}\to \Mult{A}$ becomes a
strictly continuous \emph{unital *-mono\-mor\-phism}.
\end{NOR}

\FRem{cross-references to above are desirable !!}

Let $B:=\ell_\infty (\Mult{A})/c_0(\Mult{A})$ and
$\sigma\in\Aut{B}$ be the auto\-mor\-phism induced by the 
forward-shift
$( a_1,a_2,\dots ) \mapsto ( 0,a_1,a_2,\dots )$
on
$\ell_\infty \left( \Mult{A}\right)$.

Now consider the natural embedding of the
inductive limit $G$ of
$$\Mult{A} \stackrel{h}{\longrightarrow} 
\Mult{A}\stackrel{h}{\longrightarrow} \cdots\, $$
into $B$,
given by the unital \emph{*-monomorphisms}
$h_{n,\infty}\colon \Mult{A}\to B$ with: 
\begin{eqnarray*}
h^\infty\colon a\in \Mult{A} &\mapsto&
  \left(a,h(a),h^2(a),\dots\right)
        \in\ell_\infty\left(\Mult{A}\right)\,, \\
h_{n,\infty}(a) &:=&
  \sigma^{n-1}\left(h^\infty(a)
        +c_0(\Mult{A})\right) \,,
\end{eqnarray*}
compare Remark \ref{R:h-non-deg} and the notations in
(i)--(vi) above Lemma \ref{L:B-rtimes-Z}.
\begin{eqnarray*}
G&:=&\indlim{h\colon \Mult{A}\to\Mult{A}}
=\overline{\bigcup h_{k,\infty} \left( \Mult{A} \right)}\\
&\subset& B:=
\ell_\infty \left(\Mult{A}\right)/c_0 \left(\Mult{A}\right).
\end{eqnarray*}

$\sigma$ induces an auto\-mor\-phism on $G$ with
$$
\sigma \left( h_{k,\infty}\left(\Mult{A}\right)\right)=
h_{k+1,\infty} \left( \Mult{A} \right).
$$
Then 
$h_{k,\infty}\colon \Mult{A} \stackrel{\sim}{\to} 
h_{k,\infty} (\Mult{A})
\subset h_{k+1,\infty}(\Mult{A})\subset B$ 
define unital *-mono\-mor\-phisms with
\begin{Aufl}
\item[(HR)] $\sigma h_{k,\infty}=h_{k+1,\infty}\,$,
$\,h_{k+1,\infty} h=h_{k,\infty}\,$,\,
and  $\,\sigma^{-1}h_{k,\infty}=h_{k,\infty} h\,$
for $k\in \N\,$.
\end{Aufl}
Let $D:=h_{1,\infty} (A)$. 
We have $h_{2,\infty}( h(A)+A ) 
=
h_{1,\infty}(A)+h_{2,\infty}(A)=D+\sigma (D)$,
$\sigma^{n-1}(D)=h_{n,\infty}(A)$ for $n\in\N\,$.

The following Lemma \ref{L:h-and-D} translates
(ND), (GP), (FF) to
the terminology in  subsections
\ref{ssec:semi-inv-ideals} and
\ref{ssec:id-of-crossed}. 

\begin{LEM}\Label{L:h-and-D}
With $D:=h_{1,\infty}(A)\subset B$, 
$\sigma\in\Aut{B}$ as above we get from properties
(ND), (GP) and (FF) of \ref{Nr:a4.16}:
\begin{Aufl}
\item $D$ is $\sigma$-modular in the sense
        of Definition \ref{D:sigm-mod}.
\item 
  $J\in \latI{A}$ 
  satisfies $h(J)A\subset J$, if and only if,
  $J_1:=h_{1,\infty}(J)$ is  
        $(D,\sigma)$-semi-invariant
  in the sense of Definition \ref{D:s-inv/cancel}.
\item
  $J\in \latI{A}$ 
  satisfies $h(J)=h(A)\cap \left( A+\Mult{A,J}\right)$, 
if and only if,
  $J_1:=h_{1,\infty}(J)$ is $(D,\sigma)$-semi-invariant
  and has $(D,\sigma)$-cancellation
  in the sense of Definition \ref{D:s-inv/cancel}.
\end{Aufl}
\end{LEM}

\begin{proof} \Ad{(i)}
$\sigma(D)=h_{2,\infty}(A)$
and $h_{2,\infty}h=h_{1,\infty}$ by (HR), thus: 

\Ad{($\alpha$)} $D\sigma (D)=h_{2,\infty} (h(A)A) 
=h_{2,\infty}(A)=\sigma(D)$.

\Ad{($\beta$)}  
  $\sigma(D)=h_{2,\infty}(A)$ is essential in $D+\sigma(D)=
  h_{2,\infty}( h(A)+A ) 
        \subset h_{2,\infty} \left( \Mult{A} \right)$
        because $A$ is essential in $h(A)+A\subset \Mult{A}$
  and $h_{2,\infty}$ is faithful on $\Mult{A}$.

\Ad{($\gamma$)} 
$D\cap \sigma (D)= h_{2,\infty}(h(A)\cap A )=\{0\}$.

\Ad{(ii)} Since $h_{1,\infty}|A$ is a \Cast-isomorphism
from $A$ onto $D$, $J\mapsto J_1:=h_{1,\infty}(J)$
is a lattice isomorphism from $\latI{A}$
onto $\latI{D}$. 
It holds
\begin{eqnarray*}
J_1\sigma(D)
&=&
h_{1,\infty}(J)h_{2,\infty}(A)=h_{2,\infty} ( h(J)A )\, \\
\sigma(J_1)
&=&
\sigma \left( h_{1,\infty}(J) \right) =
h_{2,\infty}(J)\quad \textrm{and}\\
h^{-1} \left( h(A)\cap\Mult{A,J} \right) &=&
   \{  a\in A \fdg h(a)A\subset J \}\,.
\end{eqnarray*}
Therefore, the following equivalences hold:
$$J_1\sigma(D)\subset \sigma \left( J_1 \right) 
\Longleftrightarrow h(J)A\subset J \Longleftrightarrow 
J\subset h^{-1}\left( h(A)\cap\Mult{A,J}\right) \,.$$

\Ad{(iii)} $a\in A$ satisfies $h(a)\in \Mult{A,J}+A$,
if and only if, there is $b\in A$ with
$(h(a)+b)A\subset J$. If we apply 
$h_{2,\infty}=\sigma\circ h_{1,\infty}$, we see that 
$(h(a)+b)A\subset J$ is equivalent to
$(a_1+\sigma(b_1))\sigma(D)\subset \sigma(J_1)$
for $J_1:=h_{1,\infty}(J)$, $a_1:=h_{1,\infty}(a)$
and $b_1:=h_{1,\infty}(b)$.
Hence, $J_1$ has $(D,\sigma)$-cancellation,
if and only if, $h^{-1}(h(A)\cap(A+\Mult{A,J})\subset J$.
\end{proof}

\begin{REM}\Label{R:D-to-h}
Conversely, if $B_1$ is a \Cast-algebra, 
$\sigma_1\in \Aut{B_1}$ and if $A\subset B_1$
is $\sigma$-modular in the sense of
Definition \ref{D:sigm-mod} (i.e.\ satisfies 
($\alpha$), ($\beta$) and ($\gamma$) for
$A$ in place of $D$), then 
$h\colon A\to \Mult{A}$ defined by 
$h(a)b:=\sigma^{-1}(a)b$ has properties
(ND), (GP) and (FF).
The map $\varphi:=h_{1,\infty}$ is a *-isomorphism
from $A$ onto $D$ with 
$\varphi(\sigma_1^{-1}(a)b)=
\sigma^{-1}(\varphi(a))\varphi(b)$.
It extends to
an isomorphism $\varphi_e$ from
$[A]_{\sigma_1}$ onto $[D]_\sigma$
with $\varphi_e \circ \sigma_1=\sigma \circ \varphi_e$
(see Proposition \ref{P:a4.15}).

A closed ideal $J$ of $A$ is 
$(A,\sigma_1)$-semi-invariant
(respectively has $(A,\sigma_1)$-cancellation)
if and only if $h(J)A\subset J$
(respectively $h(A)\cap (A+\Mult{A,J})\subset J$)
by Remark \ref{R:h-from-sigma}.
\end{REM}

\begin{REM}\Label{R:h-D(infty,1)} 
By (HR), it holds
$$
\sigma^{-n}(D)+\dots+\sigma^{-1}(D)+D
=
h_{1,\infty}(h^n(A)+\dots+h(A)+A)\,.$$
Hence, $D_{-\infty,1}=h_{1,\infty}(C)$ for
the closure $C$ of $A+h(A)+h^2(A)+\cdots\,$.

$h\colon C\to C$ is a non-degenerate faithful
endomorphism, and $[D]_\sigma$ is
isomorphic to the inductive limit of
$C \stackrel{h}{\longrightarrow}
C \stackrel{h}{\longrightarrow} \cdots\,$,
because $h_{1,\infty}\circ h=\sigma^{-1}\circ h_{1,\infty}$
and $[D]_\sigma$ is in a natural way the inductive limit of
$D_{-\infty,1} \stackrel{\sigma^{-1}}{\longrightarrow}
D_{-\infty,1} \stackrel{\sigma^{-1}}{\longrightarrow} \cdots\,$
(cf.\ Remark \ref{R:type-I-endo}).

$C$ is of type I if $D$ is of type I.
\end{REM}

%%%%%%%%%%%%%%%%%%%%%%%%%%%%%
%\Rem{changes from now, state: 24.5.04, 17:30}
%%%HEREYYY23.5.
%\YYY

\begin{COR}\Label{C:char-[D]sigma-by-h}
Suppose that $h\colon A\to \Mult{A}$
is a non-degenerate *-mono\-mor\-phism
with $h(A)\cap A=\{ 0\}$.
Let $B$,
$h_{1,\infty}\colon \Mult{A}\to B$, and
$\sigma\in \Aut{B}$ as above.
Then:
\begin{Aufl}
\item $D:=h_{1,\infty}(A)$ is 
  $\sigma$-modular in the sense of Definition
  \ref{D:sigm-mod}.
\item
  The \Name{Cuntz--Pimsner} algebra
  $\OOO{\Hm{A,h}}$ is the full hereditary
  \Cast-sub\-al\-ge\-bra of
  $[D]_\sigma\rtimes_\sigma\Z$ 
  generated by $D$.
\item
  Every non-zero closed ideal $I$ of $\OOO{\Hm{A,h}}$ 
  has non-zero intersection $D\cap I$ with $D$.
  (I.e.\ every *-representation 
  $\varrho \colon \OOO{\Hm{A,h}} \to\LOp{H}$
  with faithful restriction
  $\varrho |D$ is itself faithful.)
\item 
  $J:=(h_{1,\infty})^{-1}(D\cap I)$
  satisfies $h(J)A\subset J$.
\end{Aufl}
\end{COR}

\begin{proof}
(i) is Lemma \ref{L:h-and-D}(i).

(ii) is part of Proposition \ref{P:Pims-alg}.

\Ad{(iii)}
Since $\OOO{\Hm{A,h}}$ is a full hereditary \Cast-subalgebra
of $[D]_\sigma\rtimes_\sigma\Z$, every non-zero closed ideal
is the intersection of a non-zero closed ideal $K$ of
$[D]_\sigma\rtimes_\sigma\Z$ with $\OOO{\Hm{A,h}}$.
By Proposition \ref{P:a4.7} the intersection $I\cap D=K\cap D$
is non-zero.

\Ad{(iv)}
Since $I\cap D=K\cap D$ for a closed ideal $K$ of 
$[D]_\sigma\rtimes_\sigma\Z$, $I\cap D$ is $(D,\sigma)$-semi-invariant
in the sense of Definition \ref{D:s-inv/cancel} by
Proposition \ref{P:a4.7}.
Thus, $J:=h_{1,\infty}^{-1}(I\cap D)$ satisfies 
$h(J)A\subset J$ by Lemma \ref{L:h-and-D}(ii).
\end{proof}

\FRem{Next necessary?}
\begin{REM}\Label{R:converse}
Conversely, if $B_1$ and $\sigma_1\in \Aut{B_1}$
and $\sigma_1$-modular $D_1\subset B_1$ are given,
then $h_1\colon D_1\to \Mult{D_1}$  defined
by $h_1(a)b:=\sigma_1^{-1}(a)b$ is a non-degenerate
*-monomorphism with $h_1(D_1)\cap D_1=\{ 0\}$.
Then $h_{1,\infty}\colon D_1\to D$ 
extends to an isomorphism from
$([D_1]_{\sigma_1},\sigma_1)$ onto
$([D]_\sigma,\sigma)$ where $h_{1,\infty}$ and 
$\sigma\in\Aut{B}$ 
are as in Corollary \ref{C:char-[D]sigma-by-h}.
(See Remark \ref{R:h-from-sigma}, 
Corollary \ref{C:mod-to-CP}
and \ref{P:a4.15}.)
\end{REM}

\begin{COR}\Label{C:stable-isomorphism}
If $h\colon A\to \Mult{A}$
is a non-degenerate *-mono\-mor\-phism
with $h(A)\cap A=\{ 0\}$, and $A$ is
stable and $\sigma$-unital,
then the \Name{Cuntz--Pimsner} algebra $\OOO{\Hm{A,h}}$ 
is stable and is isomorphic to the crossed product 
of the inductive limit 
$C \stackrel{h}{\longrightarrow}
C \stackrel{h}{\longrightarrow} \cdots\,$ by $\Z$,
where $C$ is the closure of $A+h(A)+h^2(A)+\cdots$.

$C$ has a decomposition series with intermediate factors 
isomorphic to $A$.  In particular $C$ is of type I if
$A$ is of type I.
\end{COR}

\begin{proof}
Since $\OOO{\Hm{A,h}}$ is the full hereditary \Cast-subalgebra
of $[D]_\sigma\rtimes_\sigma\Z$ generated by $D$ (cf.\ 
Corollary \ref{C:char-[D]sigma-by-h}) and since $D$ is
isomorphic to $A$, we have that $[D]_\sigma\rtimes_\sigma\Z$ 
is stable and is isomorphic to $\OOO{\Hm{A,h}}$ by
Proposition \ref{P:Pims-alg}. 
On the other hand $[D]_\sigma$ is isomorphic to the inductive limit
$C \stackrel{h}{\longrightarrow}
C \stackrel{h}{\longrightarrow} \cdots\,$
by Remark \ref{R:h-D(infty,1)}.
$C$ has a decomposition series given by the closed ideals
$C_n:=A+h(A)+\cdots+h^n(A)$ with $C_{n+1}/C_n\cong h^{n+1}(A)\cong A$.
\end{proof}

\begin{REM}
More generally, if $\mathcal{P}$ is a property of 
\Cast-algebras
that is preserved under split-extensions, 
\Name{Morita} equivalence, crossed products by $\Z$, and
inductive limits, and if $A$ has property 
$\mathcal{P}$,
then $\OOO{\Hm{A,h}}$ has property $\mathcal{P}$.  
This happens
for example for the property $\mathcal{P}$ of being locally-reflexive,
exact, nuclear or weakly-injective.
\end{REM}

\begin{COR}\Label{C:inv-ideal}
Suppose that $h\colon A\to \Mult{A}$
is a non-degenerate *-mono\-mor\-phism
with $h(A)\cap A=\{ 0\}$.
Build $B$,
$\sigma\in \Aut{B}$ and 
$D\subset B$ as above from $h$.

If $I$ is a closed ideal of $\OOO{\Hm{A,h}}$ 
such that $J:=(h_{1,\infty})^{-1}(D\cap I)$
satisfies 
$h(J)=h(A)\cap (A+\Mult{A,J})$,
then $I$ is the closed linear span of
$\OOO{\Hm{A,h}}(D\cap I)\OOO{\Hm{A,h}}$.
\end{COR}
In particular, every ideal $K$ of $\OOO{\Hm{A,h}}$ 
with $D\cap I=D\cap K$ coincides with $I$. 
\begin{proof}
By Corollary \ref{C:char-[D]sigma-by-h}(iv)
and Lemma \ref{L:h-and-D}(ii) $J_1:=D\cap I$ is
$(D,\sigma)$-semi-invariant in the sense of 
Definition \ref{D:sigm-mod}.
It has $(D,\sigma)$-cancellation by Lemma \ref{L:h-and-D}(iii).
Thus, there is only one closed ideal $I_1$ of 
$[D]_\sigma\rtimes_\sigma\Z$
with $I_1\cap D=h_{1,\infty}(J)=D\cap I$.
Since $\OOO{\Hm{A,h}}$ is a full hereditary \Cast-subalgebra
of $[D]_\sigma\rtimes_\sigma\Z$ by 
Corollary \ref{C:char-[D]sigma-by-h}(ii),
it follows that the closed ideal of $\OOO{\Hm{A,h}}$ generated by 
$D\cap I$ coincides with $I$.
\end{proof}

\begin{COR}\Label{C:ideals-of-O(H(A,h))}
Suppose that $h\colon A\to \Mult{A}$
is a non-degenerate *-mono\-mor\-phism
with $h(A)\cap A=\{ 0\}$.
Build $B$,
$\sigma\in \Aut{B}$ and 
$D\subset B$ as above from $h$.

If every closed ideal $J$ of $A$ with 
$h(J)A\subset J$ satisfies 
$$h(J)=h(A)\cap( A+ \Mult{A,J}),$$
then the map
$$
I\in \latI{\OOO{\Hm{A,h}}} \mapsto I\cap D \in \latI{D},
$$
is a lattice iso\-mor\-phism from
$\latI{\OOO{\Hm{A,h}}}\cong
\latO{\Prim{\OOO{\Hm{A,h}}}}$
onto the Hausdorff lattice of closed ideals 
$J$ of $A$  that satisfy
$h(J)A\subset J$.

In particular, 
$D\cong A$ is a regular \Cast-subalgebra of $\OOO{\Hm{A,h}}$.
\end{COR}

Corollary \ref{C:ideals-of-O(H(A,h))} derives from 
Corollary \ref{C:inv-ideal} as 
Corollary \ref{C:ideals} derives from 
Proposition \ref{P:ideal}.

In particular, we get the following corollary.
\begin{COR}\Label{C:O(H(A,h))-simple}
Suppose that $h\colon A\to \Mult{A}$
is a non-degenerate *-mono\-mor\-phism
with $h(A)\cap A=\{ 0\}$.
If every closed ideal with $h(J)A\subset J$
is trivial, then $\OOO{\Hm{A,h}}$ is simple.
\end{COR}

%%%
\section{Proofs of the main results}\Label{sec:main-results}
We give here the proof of Theorem \ref{T:main}.
Further we
derive two corollaries, whose contents we have 
mentioned already in the introduction.

\begin{proof}[Proof of Theorem \ref{T:main}]
Recall that $\latO{X}$ denotes the open subsets of $X$.
By (I)--(IV), $\Psi(\latO{X})$ is a sublattice of $\latO{P}$
that contains $\emptyset$, $P$ and is closed under l.u.b.\
and g.l.b.  Let $A:=\Cont[0]{P,\K}$ and $h\colon A\to\Mult{A}$
the non-degenerate *-monomorphism of Corollary \ref{C:H0-from-Omega}
for $\Omega:=\Psi(\latO{X})$.
Since $h$ is unitarily equivalent to $\delta_\infty\circ h$
by (ii) of Corollary \ref{C:H0-from-Omega} and since
$\delta_\infty(\Mult{A})\cap A=\{0\}$ by Remark \ref{R:delta-infty}(iii),
it follows that $h(A)\cap A=\{0\}$.

By Remark \ref{R:H0-inv-ideal} it holds for every $J\in\latI{A}$ with
$h(J)A\subset J$ that
$$ h(J)=h(A)\cap (A+\Mult{A,J}).$$
Thus Corollary \ref{C:char-[D]sigma-by-h} and 
Corollary \ref{C:ideals-of-O(H(A,h))} apply,
and we get: 
\begin{Aufl}
\item $\OOO{\Hm{A,h}}$ is isomorphic to the crossed product of
  $E:=[D]_\sigma$ by $\Z$ with respect to the shift 
  automorphism $\sigma$ restricted to $E$, where $D:=h_{1,\infty}(A)$.
\item The map
  $$
  I\in \latI{\OOO{\Hm{A,h}}} \mapsto I\cap D \in \latI{D},
  $$
  is a lattice iso\-mor\-phism from
  $\latI{\OOO{\Hm{A,h}}}\cong
  \latO{\Prim{\OOO{\Hm{A,h}}}}$
  onto the Hausdorff lattice of closed ideals 
  $J$ of $A$  that satisfy
  $h(J)A\subset J$.
\end{Aufl}
By construction of $h\colon A\to\Mult{A}$ the Hausdorff lattice
of closed ideals $J$ of $A$ with $h(J)A\subset J$ is naturally
isomorphic to $\Omega$, and $\Omega$ is isomorphic to the lattice
$\latO{X}$ by $\Psi$.  Thus the composition of the isomorphisms
define a lattice isomorphism from $\latO{\Prim{\OOO{\Hm{A,h}}}}$
onto $\latO{X}$.  Since $X$ is point-complete by 
assumption and since $\OOO{\Hm{A,h}}$ is separable, we get
that $X$ and $\Prim{\OOO{\Hm{A,h}}}$ are homeomorphic by
Corollary \ref{C:lat-iso}.

The properties (i)--(iii) follow from Corollary \ref{C:char-[D]sigma-by-h}, 
Corollary \ref{C:ideals-of-O(H(A,h))}, and 
Corollary \ref{C:stable-isomorphism}.
\FRem{check}
\end{proof}

\begin{REM}\Label{R:zuTh-main}
\begin{Aufl}
\item
Since $E=[A]_\sigma$ in 
our proof and since every $\sigma$-invariant ideal
$I$ of $[A]_\sigma$ is determined by its intersection with
$A=\Cont[0]{P,\K}$, we have that 
$\Cont[0]{P}$ is isomorphic
to a regular Abelian \Cast-sub\-al\-ge\-bra of 
$E\rtimes_\sigma\Z$.
\item
If $B$ is a separable \Cast-algebra such that 
$B\otimes\OO{2}$ contains
an Abelian regular sub\-al\-ge\-bra then 
Theorem \ref{T:main} applies to 
$X:=\Prim{B\otimes\OO{2}}\cong\Prim{B}$ by
Lemma \ref{L:reg-abel-psi}.  
Further, in the case where $B$ is
in addition nuclear, we get 
that $B\otimes\OO{2}\otimes\K$ is isomorphic to 
$\left(E\rtimes_\sigma\Z\right)\otimes\OO{2}\otimes\K$
from \cite[chp.~1, cor.~L]{K.book}
(cf.\ Remark \ref{R:iso-of-prim}).
\end{Aufl}
\end{REM}

Theorem \ref{T:main} and Remark \ref{R:zuTh-main} imply:
\begin{COR}\Label{C:reg-abel}
If $B$ is a separable \Cast-algebra such that 
$B\otimes\OO{2}$ 
contains an Abelian regular \Cast-sub\-al\-ge\-bra $C$, 
then the primitive ideal 
space of $B$ is isomorphic to 
the primitive ideal space of the crossed 
product $E\rtimes_\sigma\Z$ of an inductive limit 
$E$ of \Cast-algebras 
of type~I by an auto\-mor\-phism $\sigma$ of $E$.

In particular, if $B$ is in addition nuclear, then
$$B\otimes\OO{2}\otimes\K \cong 
\left( E\rtimes_\sigma\Z \right)  
\otimes \OO{2}\otimes\K.$$
\end{COR}

%%%
%%%
\begin{COR} \Label{C:exist.regular}
A separable \Cast-algebra $B$ has a primitive ideal space 
$X=\Prim{B}$ 
which is the continuous pseudo-open 
and pseudo-epimorphic image 
of a locally compact Polish space $P$ 
if and only if $B\otimes\OO{2}$
contains a ``regular'' Abelian 
\Cast-sub\-al\-ge\-bra in the sense of
Definition \ref{D:regular}. 
\end{COR}
\begin{proof}
If $B\otimes\OO{2}$ contains an Abelian \Cast-subalgebra
$C\cong\Cont[0]{P}$ which is regular in $B\otimes\OO{2}$,
then $X=\Prim{B}\cong\Prim{B\otimes\OO{2}}$ is the pseudo-open 
and pseudo-epimorphic image of $P$ by 
Corollary \ref{C:char-prim-nuc}.

Conversely if $\pi\colon P\to X$ is pseudo-open and pseudo-epimorphic,
then there is a *-monomorphism $\varphi$ from $\Cont[0]{P}$
into $B\otimes\OO{2}$ which is given by the embeddings
$$ \Cont[0]{P}\cong\Cont[0]{P,\C\cdot 1\otimes e_{11}}
\subset \Cont[0]{P,\OO{2}\otimes\K}
\hookrightarrow B\otimes\OO{2}\otimes\K
\subset B\otimes\OO{2} $$
such that $\Psi(U):=\pi^{-1}(U)$ is induced by
$J\mapsto \varphi^{-1}(\varphi(\Cont[0]{P})\cap J)$.
Here the inclusion $\Cont[0]{P,\OO{2}\otimes\K}
\hookrightarrow B\otimes\OO{2}\otimes\K$
comes from \cite[chp.~1, thm.~K]{K.book}
because $B\otimes\OO{2}\otimes\K$ is stable and
strongly purely infinite.
\end{proof}

\begin{REM}
The construction used in the Theorem \ref{T:main} implies,
for $X=\{\text{point}\}$, that 
$E\rtimes_\sigma\Z$ is isomorphic to 
$\OO{\infty}\otimes\K 
=\OOO{\ell _2}\otimes\K$. 
If $X=P$ is \Name{Hausdorff} and
$\Psi =\id{P}$, then 
$E\rtimes_\sigma\Z$ is isomorphic
to $\Cont[0]{P,\OO{\infty}\otimes\K}$.
\end{REM}

\section{\Name{Dini} spaces}\Label{sec:7}
Throughout this section we require that 
$X$ is a point-complete
and second countable $\T{0}$-space. 
(If one of this conditions
is not satisfied, 
then Definitions and results become more complicate.)
\begin{DEF}\Label{D:Dini}
A function $g\colon X\to [0,\infty)$ is a 
\emph{Dini function} on $X$ 
if $g$ is a lower semi-continuous and
$\sup g \left( \bigcap_n F_n \right) =
\inf_n\sup g \left( F_n \right) $
  for every decreasing sequence 
$F_1\supset F_2\supset\dots$
  of closed subsets $F_n$ of $X$.
(Here we use the convention $\sup\emptyset:=0$.)

We call $X$ a \emph{Dini space} if the supports of 
Dini functions build a base of the topology of $X$.
\end{DEF}

It turns out that
\begin{Aufl}
\item[(i)]  If $g$, $h$ are Dini functions on $X$ then 
$g$ is bounded 
  and $\max ( g,h ) $ is a Dini functions.
\item[(ii)]  The set of Dini functions on $X$ is 
closed under uniform convergence.
\item[(iii)]  If $g$ is a Dini function and 
$f\colon \R_+\to\R_+$ 
is an increasing
  continuous function with $f(0)=0$ then 
$x\mapsto f(g(x))$
  is a Dini function.
\end{Aufl}

In \cite{DiniEK1} the second-named author 
has shown that for a 
non-negative lower semi-continuous
function $g$ on $X$
the following three properties (iv)-(vi) are equivalent
(and justify the name ``Dini'' function): 
\begin{Aufl}
\item[(iv)]  $g$ is a Dini function.
\item[(v)] 
Every increasing sequence  
$0\leq f_1\leq f_2\leq \ldots$ of
non-negative lower semi-continuous functions
on $X$ with $g(x)=\sup _n f_n(x)$ for all $x\in X$
converges uniformly to 
$g$ (i.e.~$\lim _n \| g-f_n \|_\infty =0$)
\item[(vi)] 
For every $\gamma >0$
the $G_\delta$-set $\{ y\in X\fdg g(y)\ge \gamma \}$
is quasi-compact and $g$ is bounded.
\end{Aufl}

In general the set of Dini functions is not convex
and is not closed under multiplication or minima,
because of the following result in \cite{DiniEK2}:\\
For  separable \Cast-algebra $A$ the Dini functions $g$
on $\Prim{A}$ are nothing else than 
the generalized Gelfand transforms $g=N(a)$ of elements
$a\in A$ given by $N(a)(J)=\| a+ J\|$ 
for primitive ideals $J$ of $A$.

\medskip

It has been shown in \cite{DiniEK1}
that a $X$ point-complete second-countable $\T{0}$ space is a Dini space 
if and only if $X$
locally quasi-compact. In particular, primitive ideal spaces of separable 
\Cast-algebras are Dini spaces.

All Dini spaces are continuous and open images 
of Polish spaces \cite{DiniEK3}.
Continuous and open images $X$ of Polish spaces have
l.u.b.\ and g.l.b.\ compatible isomorphic embeddings of 
their lattices of
open subsets into lattices of open subsets of 
Polish spaces.  
Unfortunately
a Polish space $P$ is the pseudo-open and pseudo-epimorphic 
image of a locally
compact space $Q$ if and only if 
$P$ is itself locally compact.
Thus there is still an \emph{open question} 
whether every Dini space is isomorphic to 
the primitive ideal space of a separable 
\Cast-algebra or not.

\FRem{below good??}

Our results yield equivalent
descriptions of primitive ideal 
spaces of separable nuclear
\Cast-spaces among the Dini spaces, in fact,
the following properties
(a)--(e) of a Dini space $X$ are equivalent:

\begin{Aufl}
\item[(a)]
$X$ is isomorphic to the primitive ideal space
of a separable \emph{nuclear} \Cast-algebra.
\item[(b)]
$\latF{X}$ is lattice-isomorphic to a 
sub-lattice $\mathcal{G} $
of $\latF{Y}$ which is closed under
forming of l.u.b.\ and g.l.b.\ 
for some \emph{locally compact} 
Polish space $Y$. (Here 
$\latF{Y}$ means the lattice of
closed subsets of $Y$. The g.l.b.\ is just the
intersection, and the l.u.b.\ is the closure
of the union of a family in $\latF{Y}$.)
I.e., there is a map
$\Psi$
from the open subsets $\latO{X}$ of $X$ into 
the open subsets $\latO{Y}$ of $Y$
with properties (I)--(IV) of Definition \ref{D:l-g-preserv}.
%
%\begin{Aufl}
%\item[(I)] $\Psi^{-1}\{Y\} = X$,
%               and $\Psi (\emptyset )=\emptyset$.
%\item[(II)] $\Psi ((\bigcap _\alpha U_\alpha)^\circ)=
%               (\bigcap _\alpha \Psi (U_\alpha))^\circ$. 
%               ($Z^\circ$ denotes the interior of $Z$.)
%\item[(III)] $\Psi (\bigcup _\alpha U_\alpha)=
%               \bigcup _\alpha \Psi (U_\alpha)$.
%\item[(IV)] $\Psi (U)=\Psi (V)$ implies $U=V$.
%\end{Aufl}
%

\item[(c)]
$\latF{(0,1]_{lsc}\times X}$ is 
the projective
limit of $\latF{P_n\setminus \{q_n \}}$
for pointed finite one-dimensional polyhedral 
$(P_n,q_n)$ (in a lattice sense). The connecting maps 
$\Phi_n \colon \latF{P_{n+1}\setminus \{q_{n+1} \}}\to 
\latF{P_n\setminus \{q_n \}}$
satisfy (with $Y_n=P_n \setminus \{ q_n\}$):
\begin{Aufl}
\item[(I$_0$')   ] $\Phi_n(Y_{n+1})=Y_n\,$, 
$\,\,\Phi_n (\emptyset)=\emptyset$.
\item[(II')  ] $\Phi_n (\overline{\bigcup_\tau F_\tau})=
\overline{\bigcup_\tau \Phi_n (F_\tau)}$
for every family $\{ F_\tau \}_\tau$ of closed
subsets of $\latF{Y_{n+1}}$,
\item[(III$_0$')] $\Phi_n(\bigcap _k F_k)=
\bigcap _k\Phi_n (F_k)$ for every \emph{decreasing} 
sequence
$F_1\supset F_2\supset \cdots\,$ in $\latF{Y_{n+1}}$, and
\end{Aufl}

\item[(d)]
There are a locally compact Polish space $Y$ and a
continuous map $\pi\colon Y\to X$
such that, for closed subset $F\subset G$ of $X$ 
with $F\not=G$, the set $G\setminus F$ contains
a point of $\pi (Y)$, and that
$$\overline{\bigcup_n \pi ^{-1}(F_n)}=
\pi ^{-1}\left( \, \overline{\bigcup_n F_n}\, \right)$$
for every increasing sequence $F_1\subset F_2\subset \cdots\,$
of closed subsets of $X$. 

\item[(e)]  
$\latO{X}$ is the projective limit
of a sequence of maps 
$$\Psi_n\colon \latO{X_{n+1}} \to \latO{X_n}$$
with properties (I), (II) and (III$_0$)
and $X_n\cong \Prim{A_n}$ for a separable exact \Cast-algebra.
\end{Aufl}

\begin{proof}
\Ad{(b)$\Longleftrightarrow$(d)} by Proposition \ref{P:map-Psi-pi} (note
that we can restrict to increasing sequences of closed subsets
because $X$ is second countable).
%The implication (b)$\Rightarrow$(a) follows from Theorem \ref{T:main}.

\Ad{(a)$\Longleftrightarrow$(d)} by Corollary \ref{C:char-prim-nuc}.

\Ad{(a)$\Longleftrightarrow$(e)}
Clearly, (a) implies (e) with $X_n=X$ and $\Psi_n=\id{X}$.

(e) implies (a)
because for every separable
exact \Cast--algebra $A$ there is a \emph{nuclear}
stable separable  \Cast-algebra
$B$ with the same primitive ideal space and 
$B\cong B\otimes \OO{2}$ (cf. \cite[cor.\ 12.2.20]{K.book}).
Then $\Psi_n$ is induced by a non-de\-gen\-erate
*-mono\-mor\-phism
$h_n\colon B_n\hookrightarrow B_{n+1}$ 
(cf.\ \cite[thm.\ K]{K.book})
and
$B:=\indlim{h_n\colon B_n\to B_{n+1}}$ has 
primitive ideal space $\cong X$. 

\Ad{(a)$\Longrightarrow$(c)}
by \cite[Thm.\ 5.12, Prop. 6.2]{KR03},
because $(0,1]_{lsc}\times X$ is the primitive ideal space of
$\mathcal{A}_{[0,1]}\otimes B$ where $B$ is separable and nuclear and
$X\cong \Prim{B}$.

\Ad{(c)$\Longrightarrow$(d)}
(c) implies that $Z=(0,1]_{lsc}\times X$ satisfies (e) (with $Z$ in place
of $X$).
Thus, there is a separable nuclear \Cast-algebra $A$ with primitive ideal
space isomorphic to $Z$ by the implication (e)$\Rightarrow$(a).
From the implication (a)$\Rightarrow$(d) it follows that $Z$ is a
pseudo-epimorphic and pseudo-open image of a locally compact
Polish space $Y$.  The composition $p_2\circ\pi$ of the 
map $\pi\colon Y\to Z$ with the 
projection $p_2\colon (t,x)\in Z\to x\in X$ is again 
pseudo-epimorphic and pseudo-open.
\end{proof}

\begin{REM}\Label{R:Psi-for-Dini}
For every Dini space $X$ there is a map 
$\Psi\colon \latO{X}\to \latO{\R_+\times \N}$ with
properties (I), (II), (III$_0$) and (IV).

If, in addition, $X$ is isomorphic to the primitive
ideal space of a separable \Cast-algebra, then
there exists a separable \emph{nuclear} \Cast-algebra
$A$ and an epimorphism
$\Phi\colon \latO{\Prim{A}}\to \latO{X}$ 
with (I), (II) and (III$_0$). The latter map $\Phi$ is
unknown for general Dini spaces.
\end{REM}

\appendix

%%%
\section{Preliminaries}
\Label{sec:Topol}
\subsection{$\T{0}$-spaces}
\Label{ssec:2.1}
\begin{DEF}\Label{D:MscTopSpc}
Suppose that $X$ is a $\T{0}$-space 
(i.e.\ $\overline{\{ x\}}=\overline{\{ y\}}$ implies $x=y$).
$X$ is called

\noindent
(i)
\emph{second countable} if the topology of $X$ contains a
  countable base,

\noindent
(ii)
\emph{prime} if it is not the union of two closed true subsets
  of $X$ (and a subset $F\subset X$
  is called prime, if it is prime in its relative
  topology, \Name{Hausdorff} \cite[p.~231]{Hausd} 
  calls a non-prime closed subset 
  \emph{decomposable}), and

\noindent
(iii) 
\emph{point-complete} if every closed prime subset of $X$
  is the closure of a singleton
  (the name ``spectral space'' is used in 
   \cite[def.~4.9]{HoffKeim}, others use the terminology
   ``sober space'' for our point-complete spaces).
\FRem{ref}
\end{DEF}

For a topological space 
$X$ to be prime it is of course equivalent
that it does not contain two disjoint open subsets.  
Equivalently
every open subset of $X$ is dense in $X$.  
If a subspace $F$ of $X$ is prime, then so is its closure 
$\overline{F}$, and, by that,
the closure of a singleton is prime.
In the case where 
$X$ is the primitive ideal space of a separable
\Cast-algebra $A$ every prime closed set is 
the closure of a singleton
because there is an open and continuous map from the 
Polish space of pure states on $A$ onto $X$.

\subsection{Maps related to $\Psi$}\Label{ssec:maps}
Here we give some results on the relation between 
lattice maps, point maps, and lower semi-continuous
selections.
%Lemma \ref{L:Phi},
We use in this paper Remark \ref{R:sublattice}, 
Lemma \ref{L:char-pseudo-open},
Proposition \ref{P:map-Psi-pi},
Corollary \ref{C:lat-iso},
and Lemma \ref{L:exist-Psi-eqiv.-cp-map}.

We adopt a more general viewpoint and
suppose that $X$, $Y$ are $\T{0}$~spaces, that 
$\Psi\colon \,\latO{X}\to\latO{Y}$ satisfies property 
(III) of Definition \ref{D:l-g-preserv} and 
the following
weaker conditions (I$_0$) and (II$_0$) instead of (I) and (II): 
\begin{Aufl}
\item[(I$_0$)] $\Psi(X)=Y$, $\Psi(\emptyset)=\emptyset$,
\item[(II$_0$)] 
  $\Psi(U\cap V) =\Psi(U)\cap\Psi(V)$ for all open
  subsets $U,V\subset X$.
\end{Aufl}
For example, 
if $\pi\colon Y\to X$ is a continuous map, 
then
the map $\Psi\fdg   \latO{X}\to \latO{Y}$ defined by 
$\Psi(U):=\pi^{-1}(U)$
obviously satisfies 
properties (I$_0$), (II$_0$) and (III). 
Conversely, for every
map $\Psi\fdg   \latO{X}\to \latO{Y}$ with 
(I$_0$), (II$_0$)
and (III) there is a unique continuous map $\pi\colon Y\to X$
with $\Psi(U)=\pi^{-1}(U)$ for open $U\subset X$ 
(cf.\ Proposition \ref{P:map-Psi-pi} below).

On account of properties (I$_0$) and (II$_0$)
one can define a new topology on $Y$ by an interior operation
\begin{eqnarray*}
Z^{\circ\Psi} &:=&\bigcup \{ \Psi(U)\fdg  U\in\latO{X},\, 
  \Psi(U)\subset Z \}\subset Z^\circ 
\end{eqnarray*}
for every subset $Z\subset Y$, 
or the corresponding closure operation:
\begin{eqnarray*}
\overline{Z}^\Psi &:=& 
\bigcap \{ Y\setminus\Psi(U)\fdg U\in\latO{X},\, 
  \Psi(U)\subset Y\setminus Z \} \subset \overline{Z}\,.
\end{eqnarray*}
If $\Psi$ satisfies property (III), then
$\Psi \left( \latO{X} \right) $ is not only a base of this 
$\Psi$-topology but
is the set of all $\Psi$-open sets.
The $\Psi$-topology is coarser than the given 
$\T{0}$-topology of $Y$ 
and is in general not $\T{0}$.

\begin{DEF}\Label{D:pseudo-inv}
Suppose that $X$, $Y$ are $\T{0}$ and 
that $\Psi\colon\latO{X}\to\latO{Y}$ is an arbitrary map.
We call the (order preserving) map
$$ 
\Phi\colon V\in \latO{Y}
\mapsto 
\bigcup 
\left\{ U\in\latO{X}\fdg \Psi(U)\subset V \right\}  
\in \latO{X}
$$
the \emph{pseudo-left-inverse} of $\Psi$.
(It is a left-inverse $\Phi$ of $\Psi$ 
if $\Psi$ 
satisfies the properties (I), (III) and (IV) 
of Definition \ref{D:l-g-preserv}.) 
\end{DEF}
\begin{LEM}\Label{L:Phi}
Suppose that $\Psi:\latO{X}\to\latO{Y}$ satisfies
properties (I$_0$) and (III) of Definition \ref{D:l-g-preserv}. 
Let $\Phi$ denote the pseudo-left-inverse of $\Psi$,  then
\begin{Aufl}
\item[(a)] $\Psi(\Phi(V)) \subset V$ for $V\in \latO{Y}$,
        $U\subset \Phi ( \Psi(U) )$ for $U\in \latO{X}$.
\item[(b)] $\Phi^{-1}(X)=\{ Y\}$ and $\Phi(\emptyset)$
        is the biggest open subset $U_0$ of $X$ with 
$\Psi\left( U_0 \right)=\emptyset$,
\item[(c)] 
$
\Phi\left(\right(\bigcap_\alpha V_\alpha\left)^\circ\right) 
  =\left( \bigcap_\alpha\Phi(V_\alpha)\right) ^\circ
$ 
for every family 
  $\{ V_\alpha \}_\alpha \subset \latO{Y}$.
\item[(d)] $\Phi\circ \Psi\circ \Phi=\Phi$ and
  $\Psi\circ \Phi\circ \Psi=\Psi$.
\item[(e)] $\Psi(\Phi(V))=V^{\circ\Psi}$ and
  $\Phi (V)=\Phi\left(V^{\circ\Psi}\right)$ for any
  $V\in\latO{Y}$
  if $\Psi$ has in addition property (II$_0$).
\end{Aufl}
$\Phi\circ \Psi=\id{\,\,\latO{X}}$ if $\Psi$ is injective.
\end{LEM}
Clearly (e) means 
$\Phi \left( Y\setminus\overline{Z} \right) 
= \Phi \left( Y\setminus\overline{Z}^\Psi \right) $ 
for any subset 
$Z\subset Y$.
\begin{proof}
\Ad{(a)} $U\subset \Phi ( \Psi(U) )$ by Definition
\ref{D:pseudo-inv}.

$\Psi ( \Phi(V)) 
   =\bigcup \{ \Psi(U)\fdg U\in\latO{X},
  \Psi(U)\subset V \}\subset V$ 
by (III) and \ref{D:pseudo-inv}.

\Ad{(b)} If $\Phi(V)=X$, then $Y=\Psi(X)\subset V$.
$\Psi(U)= \emptyset$ implies $U\subset \Phi(\emptyset)$.
 
\Ad{(c)} 
By monotony of $\Phi$ it follows 
for 
$W:=\left(\bigcap_\alpha V_\alpha\right)^\circ $ and
$U:=
\left(\bigcap_\alpha\Phi\left(V_\alpha \right)\right)^\circ$
that
$\Phi(W)
\subset \Phi\left(V_\alpha\right)$ 
and 
$\Phi(W) \subset U$,
because $\Phi(W)$ is open.
Clearly, $U\subset \Phi\left(V_\alpha\right)$ for all $\alpha$ 
and $U\in\latO{X}$. 

Then $\Psi(U)\subset 
\Psi \left( \Phi\left(V_\alpha\right)\right) \subset V_\alpha$
for all $\alpha$, 
because $\Psi$ is increasing by property (III).
Since $\Psi(U)$ is open, this implies
$\Psi(U)\subset W$. Thus  $U\subset\Phi(W)$
by definition of $\Phi$. Hence, $\Phi(W)=U$.

\Ad{(d)} Part (a) implies
$\Phi(V)\subset 
\Phi \left( \Psi \left( \Phi(V) \right)\right) \subset \Phi(V)$,
for $U=\Phi(V)$, because $\Phi$ is increasing.
In the same way:
$\Psi(U)\subset 
\Psi\left(\Phi\left(\Psi(U)\right)\right)\subset \Psi (V)$
with $V:=\Psi(U)$, because $\Psi$ is increasing by 
property (III).

\Ad{(e)}
$\Psi \left( \Phi(V)\right) =V^{\circ\Psi}$ by property (III).
Thus $\Phi(V)=\Phi\left(V^{\circ\Psi}\right)$ by (d).

If $\Psi$ is injective, then 
$\Psi(U)=\Psi\left(\Phi\circ \Psi(U)\right)$
implies $U=\Phi\circ \Psi(U)$, i.e.~$\Phi\circ \Psi=\id{}$ 
by (d).
\end{proof}

\begin{REM}\Label{R:sublattice}
Suppose that $\mathcal{Z}$ is a sub-lattice of $\latO{Y}$
that contains $Y$, $\emptyset$ and
l.u.b.~and g.l.b.~of families $\{ U_\alpha\}$  in $\mathcal{Z}$.
Then Lemma \ref{L:Phi} implies the existence
of an order-preserving map $\Theta\colon \latO{Y}\to \latO{Y}$
with $\Theta|\mathcal{Z}=\id{\mathcal{Z}}$, 
$\Theta(\latO{Y})=\mathcal{Z}$,
$\Theta\circ \Theta= \Theta$, 
$\Theta(V)\subset V$ for $V\in \latO{Y}$ and
$\Theta\left(\left(\bigcap_\alpha V_\alpha\right)^\circ \right) 
= 
\left(\bigcap_\alpha \Theta\left(V_\alpha\right)\right)^\circ$ 
for every family 
$\{ V_\alpha \}_\alpha \subset \latO{Y}$.

Indeed, $\mathcal{Z}$ is the set 
of all open sets of a coarser topology on $Y$.
Let $R$ denote the equivalence relation 
$x\sim _{\mathcal{Z}}y $
given by 
$ \overline{\{ x\}}^{\mathcal{Z}}
=\overline{\{ y\}}^{\mathcal{Z}}$.
The quotient space $X:=Y/R$ is a $\T{0}$-space
and the 
map $\pi\colon y\in Y\to [y]_R\in Y/R$ to the
equivalence classes is continuous. Moreover 
$\Psi(U):=\pi^{-1}U$
defines a lattice iso\-mor\-phism $\Psi$ from
$\latO{X}$ onto $\mathcal{Z}\subset \latO{Y}$.
$\Psi\colon \latO{X}\to \latO{Y}$ satisfies (I)--(IV)
of Definition \ref{D:l-g-preserv}, because
$\mathcal{Z}$ is closed under unions and 
$\latO{Y}$-interiors of
intersections of families in $\mathcal{Z}$.
Let $\Phi\colon \latO{Y}\to \latO{X}$ the left-inverse
defined by Lemma \ref{L:Phi} and let $\Theta:=\Phi\circ \Psi$.
Then $\Theta\colon \latO{Y}\to \latO{Y}$ has the above listed
properties by Lemma \ref{L:Phi}.
\end{REM}

\begin{LEM} \Label{L:on-pi-def}
Suppose that $X$ and $Y$ are $\T{0}$ spaces and that 
$\Psi\colon\latO{X}\to\latO{Y}$ 
satisfies properties (III) of
Definition \ref{D:l-g-preserv},
\begin{Aufl}
\item[(I$_0$)] $\Psi(X)=Y$, $\Psi(\emptyset)=\emptyset$, and
\item[(II$_0$)] 
  $\Psi ( U\cap V ) 
  =\Psi(U)\cap\Psi(V)$ for all open
  subsets $U,V\subset X$.
\end{Aufl}
Let $\Phi$ denote the pseudo-left-inverse of $\Psi$.  
Then
for any prime closed set $F\subset Y$ the set
$X\setminus\Phi ( Y\setminus F ) $ is a prime closed
subset of $X$.
 
In particular, if $X$ is point-complete, then 
the complement $F_y$ of
$$\Phi\left(Y\setminus  \overline{\{y\}}\right)= 
\bigcup  \left\{  U\in\latO{X}\fdg y\not\in\Psi(U) \right\}$$
is the closure of a singleton $x=:\pi(y)$ for every $y\in Y$.
\end{LEM}
\begin{proof}
Suppose $X\setminus\Phi(Y\setminus F)$ is not prime, 
i.e.\ $\Phi(Y\setminus F)$ is the intersection of two open
subsets $U_1,U_2\subset X$
both different from $\Phi(Y\setminus F)$.  Then 
$\Psi\left( \Phi(Y\setminus F) \right) 
\subset \Psi\left(U_j\right)$ 
for $j=1,2$ and
$\Psi \left( \Phi(Y\setminus F) \right) 
=\Psi\left(U_1\cap U_2\right)
=\Psi\left(U_1\right)\cap \Psi\left(U_2\right)$ 
by property (II$_0$) and monotony of $\Psi$.
Further 
$\Psi\left(U_1\right)\cap\Psi\left(U_2\right)=
\Psi \left( \Phi(Y\setminus F)\right) 
\subset Y\setminus F$.  This is equivalent to
$F\subset Y\setminus \left( \Psi(U_1\cap\Psi(U_2) \right)$ 
and implies 
$F= \left( F\cap \left( Y \setminus \Psi(U_1) \right) \right) 
\cup \left( F\cap \left( Y\setminus\Psi(U_2)\right)\right)$,
i.e.~$F$ is the union of two closed sets.  But both sets
$F\cap \left( Y\setminus\Psi(U_j)\right)$, 
$j=1,2$ are different from $F$,
since $U_j\not\subset\Phi(Y\setminus F)$ implies
$\Psi(U_j)\not\subset Y\setminus F$ for $j=1,2$.
So $F$ is not prime.
\end{proof}
By Lemma \ref{L:on-pi-def} we get a well-defined  map 
$$ 
\Phi'\colon F\in Y^c:=\Prime{\latF{Y}}\mapsto 
X\setminus\Phi ( Y\setminus F )  \in X^c:=\Prime{\latF{X}}
$$
One can define maps $\eta$ 
that map every point to its closure, i.e.
\begin{eqnarray*}
\eta_Y \colon\; y\in Y &\mapsto& 
\overline{\{y\}}\in\Prime{\latF{Y}}\\
\eta_X \colon\; x\in X &\mapsto& 
\overline{\{x\}}\in\Prime{\latF{X}},
\end{eqnarray*}
which are injective because $X$, $Y$ are $\T{0}$-spaces.
If 
$\eta_X(X)=\Prime{\latF{X}}$), i.e.~if $X$ is point-complete,
then the map $\pi\colon Y\to X$ of 
Lemma \ref{L:on-pi-def} is given by 
$\pi:=\left(\eta_X \right)^{-1}\circ\Phi'\circ\eta_Y$ and
satisfies 
$\overline{\{\pi(y)\}}=\Phi'(\overline{\{y\}})$ for 
$y\in Y$.

\begin{LEM}\Label{L:pi-formula}
Suppose that 
$\Psi\colon \, \latO{X}\to \latO{Y}$ satisfies the
assumptions of Lemma \ref{L:on-pi-def} 
and that $X$ is point-complete.
Let $\pi\colon\, Y\to X$  as above. Then
$$\pi^{-1}U=\Psi(U)\qquad\text{for}\quad U\in\latO{X}.$$
In particular, $\pi\colon Y\to X$ is continuous.
\end{LEM}

\begin{proof} $\pi^{-1} ( X\setminus F ) 
=Y\setminus\pi^{-1}(F)$ for any subset $F\subset X$,
$\pi^{-1}(X)=Y$, and $\pi^{-1}(\emptyset)=\emptyset$ hold
for any map $\pi\colon Y\to X$.

Consider a prime set $F\in\Prime{\latF{Y}}$. Then, for 
$U\in\latO{X}$,
\begin{eqnarray*}
\lefteqn{\Phi'(F)=
X\setminus\Phi(Y\setminus F)\subset X\setminus U} \\
&\Longleftrightarrow & U\subset \Phi(Y\setminus F) 
\quad\Longleftrightarrow \quad \Psi(U)\subset Y\setminus F \\
&\Longleftrightarrow & F\subset Y\setminus\Psi(U)\in\latF{Y} 
\quad\Longleftrightarrow \quad \forall y\in F\colon y\in
 Y\setminus\Psi(U),
\end{eqnarray*}
where the second equivalence comes from Definition 
\ref{D:pseudo-inv} of $\Phi$ and property (III) of $\Psi$.
In particular,
$$ \overline{ \{ \pi(y) \} } =
\Phi'\left( \overline{\{y\}} \right) \subset X\setminus U
\quad\Longleftrightarrow \quad y\in Y\setminus\Psi(U), $$
and this shows that $\pi(y)\in X\setminus U\in\latF{X}$ 
if and only if 
$y\in Y\setminus\Psi(U)$. Thus $\pi^{-1} U=\Psi (U)$.
\end{proof}

\bigskip

%
%\begin{Aufl}
%\item[(I$_0$)] $\Psi(\emptyset)=\emptyset$, $\Psi(X)=Y$,
%\item[(II$_0$] $\Psi ( U\cap V ) =\Psi(U)\cap\Psi(V)$ 
%  for $U,V\in\latO{X}$,
%\item[(III)] $\Psi ( \bigcup_\alpha U_\alpha ) 
%  =\bigcup_\alpha \Psi ( U_\alpha ) $ 
%  for $ \{ U_\alpha \} \subset\latO{X}$.
%\end{Aufl}
%
%
Let $X$ and $Y$ $\T{0}$-spaces, we let
$p_1\colon Y\times X\to Y$ and $p_2\colon Y\times X\to X$
the maps $p_1(y,x):=y$ and $p_2(y,x):=x$ to the components.
\begin{LEM}\Label{L:lsc-relation}
Suppose that $R\subset Y\times X$ satisfies 
$p_1(R)=Y$. Let $\lambda(y)$ denote
the subset $p_2(p_1^{-1}(y)\cap R)$ of $X$ for  $y\in Y$.
\begin{Aufl}
\item 
        $p_1\colon R\to Y$ is open as a map from $R$ to $Y$,
        if and only if,
        $$
     \lambda(y)\subset \overline{\bigcup _{v\in V} \lambda(v)}
        $$
        for every subset $V\subset Y$ and $y\in \overline{V}$.
\item 
        If $X$ is a compact convex set, 
$\partial X$ the set of its
        extreme points  and if
        $\lambda(y)$ is the closed convex span of 
        $\lambda(y)\cap \partial X$
        for every $y\in Y$, then 
        $p_1\colon R\to Y$ is open if and only if
        $p_1\colon R\cap (Y\times \partial X)\to Y$ is open.
\end{Aufl}
\end{LEM}

\begin{proof} 
\Ad{(i)} Let $\alpha$ denote the restriction 
of $p_1$ to $R$. 
Note that 
$\bigcup _{v\in V} \lambda(v)$
is equal to 
$p_2\left(\alpha^{-1}(V)\right)=
p_2\left(p_1^{-1}(V)\cap R\right)$ 
for every subset $V$ of $Y$,
because $\alpha^{-1}(V)$ is the union of the
sets $\{ v\} \times \lambda(v)$ with $v\in V$.

It is easy to see that the continuous
map $\alpha \colon R\to Y$ is open if and only if
$\left(\alpha ^{-1}(W)\right)^\circ
\subset \alpha ^{-1}(\, W^\circ\, )$
for every subset $W$ of $Y$
(cf.~eg.~\cite[lem.~6.9]{DiniEK1}). 
And this is equivalent to
$\alpha ^{-1}\left(\, \overline{V}\, )\right) 
\subset \overline{\alpha^{-1}(V)}$ for every subset 
$V$ of $Y$
(by $W:=X\setminus V$).
On the other hand, 
$p_2 \left( \overline{S} \right)\subset \overline{p_2(S)}$ 
for every subset $S$ of $R$, because
$p_2\colon R\to X$ is continuous.
Thus $\lambda(y)$ is contained in the closure of
$p_2\left(\alpha^{-1}(V)\right)$  
for every subset $V\subset Y$ and every
$y\in \overline{V}$
if $\alpha\colon R\to Y$ is open.

Conversely, suppose that 
$p_2\left(\alpha^{-1}(\,\overline{V}\,)\right)$ 
is contained in $\overline{p_2\left(\alpha^{-1}(V)\right)}$  
for every subset $V\subset Y$.

The complement of $\overline{\alpha^{-1}(V)}$ is the union  
of Cartesian products $U_1\times U_2$ for $U_1\subset Y$
and $U_2\subset X$ open with 
$(U_1\times U_2)\cap \alpha^{-1}(V)=\emptyset$.

Then $\alpha^{-1}(V\cap U_1)\subset Y\times (X\setminus U_2)$,
thus 
$p_2\left(\alpha^{-1}(V\cap U_1)\right)\subset X\setminus U_2$.
By assumptions, $\lambda(y)\subset X\setminus U_2$ for
$y\in \overline{V}\cap U_1$ because $X\setminus U_2$ is
closed and $\overline{V}\cap U_1\subset \overline{V\cap U_1}$.
It  implies 
$p_2(\alpha^{-1}(\overline{V})\cap (U_1\times X))
\subset X\setminus U_2$ and, finally,
$\alpha^{-1}(\overline{V})\cap (U_1\times U_2)=\emptyset$.
It shows $\alpha^{-1}(\overline{V})\subset 
\overline{\alpha^{-1}(V)}$ for every subset of $V$ of $Y$,
i.e.~ $\alpha$ is open.

\Ad{(ii)} Recall that a compact convex set
$C$ is the closed convex hull of
the set $\partial C$ of the extreme points of $C$
(Krein-Milman theorem), and that $\partial C$ is contained
in the closure $\overline{S}$ of every subset $S\subset C$
such that $C$ is the closed convex hull of $S$.
Further, 
$p_2\left(p_1^{-1}(y)\cap R\cap (Y\times \partial X)\right)$
is the same as $\lambda (y)\cap \partial X$ for $y\in Y$.

$\bigcup _{v\in V} \lambda(v)$ is
contained in the closed convex hull
$C$ of $\bigcup _{v\in V} (\lambda(v)\cap \partial X)$,
because $\lambda(v)$ is the closed convex hull of
$\lambda(v)\cap \partial X$. 

Thus $\lambda(y)\cap \partial X\subset \partial C$
is contained in the closure of 
$\bigcup _{v\in V} (\lambda(v)\cap \partial X)$,
if $\lambda(y)$ is contained in the closure of
$\bigcup _{v\in V} \lambda(v)$. 
Hence, $p_1\colon R\cap (Y\times \partial X)\to Y$
is open if $p_1\colon R\to Y$ is open by part(i).

We need a sort of ``uniform'' lower 
semi-continuity of  
$y\mapsto \lambda(y)\cap \partial X$
for the proof of the other direction:
Suppose that $p_1\colon R\cap (Y\times \partial X)\to Y$
is open, $V\subset Y$ and $y\in \overline{V}$,
$x_1, \ldots, x_n\in \lambda (y)\cap \partial X$ and that
$U_1,\ldots, U_n$ are
open subsets of $X$ with $x_j\in U_j$ for $j=1,\ldots,n$.
We show that the closure of the set 
$$
V_k:=\{ v\in V\fdg   
U_j\cap \lambda(v)\cap \partial X\not=\emptyset
\quad for\quad j=1,\ldots,k \}
$$
contains $y$ for $k\leq n$. 
We proceed by induction over $k=1,\ldots,n$.
(The case $k=1$ goes with $V_0:=V$ as the induction step.)

Let $U_0\subset Y$ denote an open neighborhood of $y$.
Since 
$\overline{V_k}\cap U_0\subset \overline{V_k\cap U_0}$,
$y$ is in the closure of $V_k\cap U_0$, thus,
$x_{k+1}\in \lambda (y)\cap \partial X$ is contained in the
closure of 
$\bigcup _{v\in V_k\cap U_0} (\lambda (v)\cap \partial X)$.
Hence, $V_{k+1}\cap U_0\not=\emptyset$. It shows that
$y$ is in the closure of $V_{k+1}$. 

Now, if $y\in \overline{V}$ and $x\in \lambda (y)$,
$W$ a neighborhood of $x$, then there
are $x_1,\ldots,x_n\in \lambda (y)\cap \partial X$,
$\mu_1,\ldots,\mu_n\in [0,1]$ with $\sum_k \mu_k=1$,
and open neighborhoods $U_k$ of $x_k$ in $X$
such that $\sum_k \mu_k z_k\in W$ whenever
$z_k\in U_k$ for $k=1,\ldots,n$, because $\lambda (y)$
is the closed convex span of $\lambda(y)\cap \partial X$
and $(z_1,\ldots,z_n)\in X^n\to \sum_k \mu_k z_k \in X$
is continuous.

There are $v\in V$ and  
$z_j\in U_j\cap \lambda (v)\cap \partial X$
by the ``uniform'' lower semi-continuity  of
$y\mapsto \lambda(y)\cap \partial X$.
$\sum_k \mu_k z_k$ is in  $\lambda(v) \cap W$.
Thus $\lambda(y)$ is contained in the
closure of $\bigcup _{v\in V} \lambda(v)$,
if $y\in \overline{V}$, and $p_1\colon R\to X$
is open by part (i).
\end{proof}

\begin{LEM}\Label{L:char-pseudo-open}
Let $\pi:Y\to X$ be a continuous and surjective map.  
Then the following are equivalent:
\begin{Aufl}
\item $\overline{\bigcup_\alpha \pi^{-1} F_\alpha}=
\pi^{-1}\left( \overline{\bigcup_\alpha F_\alpha}\right)$
holds for every increasing family 
$ \{ F_\alpha \} _{\alpha\in I}$
of closed subsets of $X$.
\item 
$\left( \bigcap_\alpha \pi^{-1} U_\alpha \right)^\circ
=
\pi^{-1}\left(
\left(\bigcap_\alpha U_\alpha \right)^\circ
\right)$
holds for every decreasing family $\left\{U_\alpha\right\}$
of open subsets of $X$.
\item Every $R_\pi$-invariant open subset of $Y$ is in 
  $\pi^{-1}\latO{X}$, and
  the map 
  $(y,z)\in R_\pi\mapsto y\in Y$ 
  is an open and surjective map.
\item For every closed subset $G$ of $Y$, the set
  $$F_G:= 
        \left\{ x\in X\fdg 
        \pi^{-1}\left(\, \overline{ \{x\} }\,\right) 
  \subset G  \right\} 
        $$ 
        is closed in $X$.
\end{Aufl}
\end{LEM}
Here $R_\pi\subset Y\times Y$ denotes the partial order on $Y$
defined by $\pi$.
Recall that
$$ R_\pi:= \left\{ (y,z)\in Y\times Y\fdg
  \overline{\{\pi(y)\}}\ni\pi(z) \right\}  $$
and that $V\subset Y$ is \emph{$R_\pi$-invariant} if
$z\in V$ and $(y,z)\in R_\pi$ implies $y\in V$, cf.~Definition
\ref{D:pseudo}. It means that $y\in Y\setminus V$ 
if and only if 
$\pi^{-1}\left(\overline{\{ y\}}\right) \subset Y\setminus V$.

Note that 
$\pi^{-1}\left(\overline{\{ y\}}\right) 
= p_2(p_1^{-1}(y)\cap R_\pi)$
where  $p_i\colon Y\times Y\to Y$ is the projection onto the 
$i$-th component.

\begin{proof}
%\Ad{(i)$\,\Leftrightarrow$(ii)}
The equivalence of (i) and (ii)
can be seen easily by passing to the 
complements.

\smallskip

We establish another equivalence for later use:
%Observe that ``$\subset$'' in (i) is always true.
(i) holds, if and only if,
\begin{equation}\Label{Eq:(i)-equi}
\overline{
\bigcup_{z\in Z} 
\pi^{-1} ( \, \overline{ \{ z \} }\, )
}
  = \pi^{-1}\left( \overline{Z} \right)  
 \qquad \text{for all }\quad Z\subset X\,.
\end{equation}

Indeed: (\ref{Eq:(i)-equi}) is a special case of (i)
with $F_z:=\overline{ \{z\} }$ for $z\in Z$,
because $\overline{Z}=
\overline{\bigcup_{z\in Z} \overline{ \{z\} } }$.
Conversely,
(\ref{Eq:(i)-equi}) implies (i) because
$\pi^{-1} \left( F \right) =
\bigcup_{z\in F} \pi^{-1}\left(\,\overline{\{z\}}\,\right)$
for closed $F\subset X$.

\Ad{(i)$\,\Rightarrow$(iii)} We show that
(\ref{Eq:(i)-equi}) implies (iii).
Let $V$ a subset of $Y$, $Z:=\pi(V)$ in (\ref{Eq:(i)-equi}), 
and let 
$\lambda(y):= p_2\left(p_1^{-1}(y)\cap R_\pi\right) 
=\pi^{-1}(\pi(y))$
for $y\in Y$. (\ref{Eq:(i)-equi}) implies
that $\lambda(y)$ is in the closure
of $\bigcup _{v\in V} \lambda(v)=
\bigcup_{z\in Z} \overline{ \{z\} } $
if $y$ is in the closure of $V$, because 
$\overline{\{\pi(y)\} }\subset \overline{Z}$.
Thus $p_1\colon R_\pi \to Y$ is 
open by Lemma \ref{L:lsc-relation}(i).

Suppose that $V$ is an $R_\pi$-invariant open
subset of $Y$, and let $F:=Y\setminus V\,$, 
$\, Z:=\pi(F)$.
Then $F$ is closed, and $y\in F$ implies 
$\pi^{-1}\left( \overline{ \{ \pi (y) \}}\right)\subset F$.
By (\ref{Eq:(i)-equi}), $\pi^{-1}(\overline{Z})\subset F$.
Since $F\subset \pi^{-1}(\pi(F))$, we get $Z=\overline{Z}$
$F=\pi^{-1}(Z)$, and $V=\pi^{-1}(U)$ for 
the open set $U=X\setminus Z$.

\Ad{(iii)$\,\Rightarrow$(iv)} Let 
$\lambda(y):=p_2(p_1^{-1}(y)\cap R_\pi)
=\pi^{-1}\left(\overline{\{\pi(y)\}}\right)$
for $y\in Y$.
If $G$ is a closed subset of $Y$ let
$W:=\{ y\in Y\,\colon \,\, \lambda(y)\subset G\}$.
Then $F_G=\pi(W)$, and $W$ is closed,
because $\lambda(y)$ is in the closure
of $\,\bigcup _{w\in W}\lambda(w)$ and $G$ is closed.
If $y\in W$ and $z\in \lambda(y)$ then 
$\pi(z)\in \overline{\{ \pi(y)\}}$. Hence
$\lambda(z)\subset \lambda (y)$ and $z\in W$.
It means that $V:=Y\setminus W$ is an
$R_\pi$-invariant open subset of $Y$. By (iii),
$V=\pi^{-1}(U)$ for an open subset $U$ of $X$,
i.e.~$\pi^{-1}(X\setminus U)=W$ and 
$F_G=X\setminus U$ is closed.

%Let $W\subset Y$ be an $R_\pi^t$-invariant set, 
%i.e.\ $R_\pi^t(W)=W$.  
%Suppose $p_1$ is open.
%By (\ref{ELchar-open}) and continuity of $p_2$
%$$ R_\pi^t \left( \overline{W} \right) 
%=p_2 \left( p_1^{-1}\left( \overline{W} \right)  \right) 
%=p_2 \left( \overline{p_1^{-1} ( W ) } \right) 
%\subset \overline{ p_2 \left( p_1^{-1} ( W ) \right) }
%=\overline{R_\pi^t ( W ) }=\overline{W}.$$
%, and this
%is contained in $\overline{W}$ 
%because $W$ is $R_\pi^t$-invariant.
%So $\overline{W}$ is again  $R_\pi^t$-invariant.
%
%Let $G$ be a closed subset of $Y$.  Then 
%\begin{eqnarray*}\pi^{-1} ( F_G ) 
%&=& \{ g\in G\fdg \pi^{-1} ( \overline{\{\pi(g)\}} ) 
%\subset G \}  \\
%&=& \{ g\in G\fdg p_2 ( p_1^{-1} ( g )  ) \subset G \} 
%=r_\pi^t(G)\subset G
%\end{eqnarray*}
%is the $R_\pi^t$-invariant kernel of $G$.
% 
%%%\Rem{more?}
%
%$G$ being closed implies 
%that
%$\pi^{-1} ( F_G ) $ 
%is an $R_\pi^t$-invariant set but also a subset of $G$,
%which shows 
%that $\overline{\pi^{-1} ( F_G ) }=
%R_\pi^t ( \overline{\pi^{-1} ( F_G ) } ) 
%\subset r_\pi^t(G) =\pi^{-1} ( F_G ) $,
%which shows that $\pi^{-1} ( F_G ) $ is closed.
%By that $W:=Y\setminus\pi^{-1} ( F_G ) =
%\pi^{-1} ( X\setminus F_G ) $ is an open $R_\pi$-invariant set.
%By (iii) there exists an open set $U\subset X$ with 
%$W=\pi^{-1}(U)$.
%But $\pi^{-1}$ is one-to-one on sets, which shows that 
%$U=X\setminus F_G$ 
%and $F_G$ must be closed.
%

\Ad{(iv)$\,\Rightarrow$(i)}
Let $Z\subset X$ and $G:=\overline{\bigcup_{z\in Z}
\pi^{-1}( \overline{ \{ z \} }) }$.  Then
$G\subset \pi^{-1}(\overline{Z})$, and
$Z\subset F_G$ and $\pi^{-1}(F_G)\subset G$ by
definition of $F_G$ in (iv).
Since $F_G$ is closed by (iv), $\overline{Z}\subset F_G$.
Thus $\pi^{-1}(\overline{Z})= G$. I.e. (iv) implies
(\ref{Eq:(i)-equi}).
%
%
%$R^t_\pi \left( \pi^{-1} ( Z ) \right) 
%\subset \pi^{-1} ( F_G ) $ and 
%$\pi^{-1} (\, \overline{Z}\, ) $
%closed in $Y$ imply $G\subset\pi^{-1} ( \overline{Z} ) $.
%But conversely
%$F_G$ being closed implies 
%$\overline{Z}\subset F_G$, and this shows
%$\pi^{-1}\left( \overline{Z} \right) \subset\pi^{-1} ( F_G ) 
%\subset G$.
%So $\pi^{-1} \left( \overline{Z} \right) 
%=G=\overline{ R_\pi ^t ( \pi^{-1} (Z) ) }$.
%
%Let $Z:=\bigcup_\alpha F_\alpha$, then 
%$\pi^{-1} \left( \overline{\bigcup_\alpha F_\alpha} \right) 
%=\overline{R_\pi^t ( \pi^{-1} ( \bigcup_\alpha F_\alpha )  ) }
%=\overline{\bigcup_\alpha R_\pi^t ( F_\alpha ) }$.
%But the last term is the same as
%$\overline{\bigcup_\alpha \pi^{-1} ( F_\alpha ) }$ because
%the pre-image of every closed subset $F$ of $X$ is 
%$R_\pi^t$-invariant:
%$R_\pi^t ( \pi^{-1}(F) ) 
%= \bigcup_{f\in F} \pi^{-1}\left( \overline{\{ f \}}\right) 
%\subset \pi^{-1}\left(\overline{F}\right)=\pi^{-1}(F)$.
\end{proof}

\begin{REM}
The property (iv) of Lemma 
\ref{L:char-pseudo-open} equivalently means that
$\widehat{f}$ 
is lower semi-continuous for every lower semi-continuous
function $f\colon Y\to [0,\infty)$.  
Here $\widehat{f}$ is defined as
$$
\widehat{f}(x):=
\sup f \left(\pi^{-1}\left( \overline{\{x\}} \right)\right).
$$
(Indeed: If $t\in [0,\infty)$, then
$\widehat{f}^{-1}[0,t]=F_G$ for $G:=f^{-1}[0,t]$.
One can take in place of $f$ the characteristic 
function of any open subset $U=Y\setminus G$ of $Y$.)
\end{REM}

\begin{LEM}\Label{L:pseudo-epic}
A continuous map $\pi\colon Y\to X$ is 
pseudo-epimorphic if and only if 
the map 
$\Theta \colon U\mapsto U\cap\pi(Y)$ 
defines a lattice iso\-mor\-phism
from $\latO{X}$ onto $\latO{\pi(Y)}$. 
\end{LEM}
\FRem{next shorter?}

\begin{proof}
$\pi$ is pseudo-epimorphic, if and only if, 
$U\setminus V$
contains a point of $\pi(Y)$ for 
all open subsets
$V\subset U$ of $X$ with $V\neq U$. 
%%HEREXXX next shorter ???
(Indeed,
let $F:=X\setminus V$, $G:=X\setminus U$, then 
$G\subset F$ are closed, $G\neq F$ and  
$F\setminus G=U\setminus V$. 
If $\pi$ is  pseudo-epimorphic, then
$F$ is the closure of $F\cap \pi(Y)$ and 
$F\setminus G$ must
contain a point of $\pi(Y)$. 
Conversely, if $\pi$ is not pseudo-epimorphic, then there
$F$ is a closed subset of $X$ such that 
$G:=\overline{F\cap \pi(Y)}\subset F$
is not equal to $F$. 
Then $U\setminus V=F\setminus G$ does not contain
a point of $\pi(Y)$.)

Clearly $\Theta$ is a lattice epi\-mor\-phism.

Let $\Theta (U_1)=\Theta (U_2)$ then $\Theta (U_1\cap U_2)=
\Theta (U_1\cup U_2)$.  
Suppose that $U_1\neq U_2$.  Then 
$ ( U_1\cup U_2 )  \setminus  ( U_1\cap U_2 ) $ is not empty
and does not contain a point of $\pi(Y)$.  
Thus 
$\Theta$ must be an iso\-mor\-phism if 
$\pi$ is pseudo-epimorphic.

Conversely suppose that $\Theta$ is a lattice iso\-mor\-phism 
and 
$V\subset U$
with $V\neq U$ then 
$V\cap\pi(Y)=\Theta(V)\neq \Theta(U)=U\cap\pi(Y)$.
It yields that $U\setminus V$ contains a point of $\pi(Y)$.
Hence, $\pi$ is pseudo-epimorphic.
\end{proof}

\FRem{Change $P,p$ to $Y,y$ below (and above?)}

\begin{PRP}\Label{P:map-Psi-pi}
Suppose that $X$ and $P$ are point-complete $\T{0}$-spaces.
Then there is a one-to-one correspondence between maps 
$\Psi\colon\latO{X}\to\latO{P}$ with properties 
\begin{Aufl}
\item[(I$_0$)] % $\Psi$ is non-degenerate, i.e.\ 
  $\Psi(X)=P$ and 
  $\Psi(\emptyset)=\emptyset$,
\item[(II$_0$)] $\Psi ( U\cap V ) 
  =\Psi(U)\cap\Psi(V)$ for all open
  subsets $U,V\subset X$, 
\item[(III)] 
  $\Psi \left( \left(\bigcup_\alpha U_\alpha\right) \right) 
  =\bigcup_\alpha\Psi\left(U_\alpha\right)$ 
  for every family of open subsets $U_\alpha\subset X$
\end{Aufl}
 and  continuous maps $\pi\colon P\to X$ given by 
$$\Psi(U):=\pi^ {-1}(U)$$
and
$$\pi(p):=x$$
where $x$ is defined by Lemma \ref{L:on-pi-def}.

$\pi$ is pseudo-open and pseudo-epimorphic
if and only if $\Psi$ satisfies in addition 
properties (II) and (IV) of Definition \ref{D:l-g-preserv},
i.e.~if and only if $\Psi$ is a
lattice mono\-mor\-phism from $\latO{X}$ into $\latO{P}$
that respects l.u.b.\ and g.l.b.\ and satisfies (I$_0$).
\end{PRP}
\begin{proof}
%Suppose $\pi$ given,
% then the properties (I), (II$_0$) and (III)
%hold for $\Psi$ and $\Psi(\latO{X})$ 
%is the coarsest topology on $P$
%induced by $\pi$. 
Suppose that $\pi\colon P\to X$ is continuous. 
Then $\Psi(U):=
\pi^{-1}U$ clearly satisfies (I$_0$), (II$_0$) and (III).

Suppose that $\pi_1$ is another map with 
$\Psi(U)=\pi_1^{-1}(U)$.
$P\setminus \pi_1^{-1}(U)=P\setminus \pi^{-1}(U)$ for 
$U:=X\setminus \overline{ \{ \pi(p)\} }$, implies
$$
\pi_1 \left( 
\pi^{-1}\left(\overline{ \{ \pi(p) \} }\right)
\right) 
=
\pi_1(P)\cap \overline{ \{ \pi(p) \} }.
$$
In particular, $\pi_1(p)$ is in $\overline{ \{ \pi(p) \} }$.
Similarly, $\pi(p)\in \overline{ \{ \pi_1(p) \} }$.
Since $X$ is $\T{0}$, $\pi_1(p)=\pi(p)$.

Suppose $\Psi$ with (I$_0$), 
(II$_0$) and (III) is given, and let 
$\pi$ be
as in Lemma \ref{L:on-pi-def}. 

By Lemma \ref{L:pi-formula}
$\pi^{-1}U=\Psi(U)$ holds for $U\in\latO{X}$. In particular
$\pi\colon P\to X$ is continuous.

Suppose that $\pi$ is in addition a
pseudo-open and pseudo-epimorphic map. Then
the inclusion map from $\pi(P)$ into 
$X$ induces an iso\-mor\-phism from
$\latO{X}$ onto $\latO{\pi(P)}$ by Lemma \ref{L:pseudo-epic}.  
On the other hand, the continuous 
epi\-mor\-phism from $P$ onto 
$\pi(P)$ is still a pseudo-open map, 
because Definition \ref{D:pseudo}(ii)
refers only to the pseudo-graph $R_\pi$
and the open subsets of $\pi(P)$.
Let $V$ be an $R_\pi$-invariant open subset of $P$.
Then $\pi(V)$ is an open subset of $\pi(P)$ by Definition
\ref{D:pseudo}. 
For $p\in \pi^{-1}(\pi(V))$ there is $q\in V$ such that 
$\pi(q)=\pi(p)$,
in particular, 
$(p,q)\in R_\pi$, thus $p\in V$ and  
$V=\pi^{-1}(\pi(V)$. 
This shows that 
$\pi\colon P\to \pi(P)$ satisfies
the conditions of Lemma \ref{L:char-pseudo-open}(iii).
$\Psi$ satisfies (I)--(IV) of 
Definition \ref{D:l-g-preserv}
by Lemma \ref{L:char-pseudo-open}(ii),
because 
$\Psi(U)=\pi^{-1}(U)=\pi^{-1}\left(U\cap\pi(P)\right)$, 
and $U\mapsto U\cap\pi(P)$ is a lattice isomorphism.

Conversely, if $\Psi$ satisfies (I)--(IV) of 
Definition \ref{D:l-g-preserv},
then the lattice mono\-mor\-phism 
$\Psi$ factorizes through
the lattice epimorphism 
$\Psi_0\colon U\mapsto U\cap \pi(P)$.
Thus $\Psi_0$ is a lattice isomorphism, and $\pi$ is
a pseudo-epimorphism by Lemma \ref{L:pseudo-epic}.
It follows that the continuous epimorphism $\pi$ from
$P$ onto $\pi(P)$ satisfies (ii) of 
Lemma \ref{L:pseudo-epic}
(by condition (II) of $\Psi$). 
Thus $\pi$ is  also pseudo-open.
\end{proof}
\begin{COR}\Label{C:lat-iso}
Suppose that $X$ and $Y$ are point-complete $\T{0}$-spaces. 
If $\Psi$ is an iso\-mor\-phism from
$\latO{X}$ onto $\latO{Y}$, 
then there is a unique 
homeo\-mor\-phism $\pi$ from $Y$ onto $X$ such that
$\Psi(U)=\pi^{-1}U$ for $U\in \latO{X}$.
\end{COR}

\medskip 

\begin{EXA}\Label{Ex:[0,1]lsc}
Let $(0,1]_{\mathrm{lsc}}$ 
denote the half-open interval $(0,1]$
with topology 
$$\latO{(0,1]_{\mathrm{lsc}}}:= 
\left\{ \emptyset, (t,1]\fdg   t\in [0,1) \right\}$$
The continuous epimorphism
$$t\in (0,1]\mapsto t\in (0,1]_{\mathrm{lsc}}$$ 
is pseudo-open but is not open.
\end{EXA}
\begin{LEM}\Label{L:reg-abel-psi}
If $C$ is a regular \Cast-sub\-al\-ge\-bra of $B$
(cf.~Definition \ref{D:regular}), then the map
$$ J\in\latI{B}\mapsto J\cap C\in\latI{C} $$
defines naturally a map $\Psi\colon \latO{\Prim{B}}\to
\latO{\Prim{C}}$ that satisfies properties (I)--(IV) 
of Definition \ref{D:l-g-preserv}.
\end{LEM}
\begin{proof} 
By the correspondence between closed ideals
and open subsets of primitive ideal spaces, 
(II) is equivalent to
$\left(\bigcap_\alpha J_\alpha\right)\cap C
=\bigcap _\alpha \left(J_\alpha\cap C\right)$, and
(IV) means that $J_1\cap C=J_2\cap C$ implies $J_1=J_2$ 
(and follows from (ii) of Definition \ref{D:regular}).
(I) holds for $\Psi$, because $\{ 0\}\cap C=\{ 0\}$ and 
$J\cap C=C$ implies $J=B$.
$\Psi$ satisfies (III), because 
$(J_1+J_2)\cap C=(J_1\cap C)+(J_2\cap C)$
for all $J_1,J_2$ in $\latI{B}$ 
by (i) of Definition \ref{D:regular},  
and because 
$\left(\overline{\bigcup_\alpha J_\alpha}\right)\cap C=
\overline{\bigcup_\alpha \left(J_\alpha\cap C\right)}$ 
holds for
every upward directed family 
$\{ J_\alpha \}_\alpha $ 
in $\latI{B}$.
The latter follows from 
$\textrm{dist}\left(c,\bigcup _\alpha J_\alpha\right)=
\inf_\alpha \textrm{dist}\left(c, J_\alpha\right)$ 
and 
$\textrm{dist}\left(c,J_\alpha\right)=
\textrm{dist}\left(c,J_\alpha\cap C\right)$ for $c\in C$.
\end{proof}

\FRem{replace $p,q$ by $y,z$ ?}
\begin{LEM}\Label{L:exist-Psi-eqiv.-cp-map}
Suppose that $A$ is a separable \Cast-algebra and 
that $Y$ is a locally compact Hausdorff space.  
Moreover suppose that
$$
\Psi\colon\latO{Y}\to\latI{A}\cong\latO{\Prim{A}}
$$
is an order preserving map.
Let $K_p\subset \mathcal{Q}(A)$ 
denote the convex set of quasi-states 
%(= positive functionals of norm $\le 1$) 
$\xi$ on $A$ with
$\,\,
\xi\left(\Psi\left(Y\setminus\{p\}\right)\right) 
=\left\{ 0 \right\}
$.

Then, 
        $\Psi$ satisfies property (II) of 
        Definition \ref{D:l-g-preserv}, 
if and only if,
  the map
  $( p,\xi ) \mapsto p\in Y$
  is an open map from 
  $R:= \left\{ (p,\xi)\fdg \xi\in K_p \right\} 
        \subset Y\times \mathcal{Q}(A)$ 
        onto $Y$.

If $\Psi$ satisfies (II) of Definition \ref{D:l-g-preserv}, 
then 
\begin{Aufl}
\item every continuous selection 
        $p\mapsto \xi_p\in K_p$ defines
  a completely positive contraction 
  $T\colon A\to \Cont[b]{Y}$
  with $T \left( \Psi(U) \right)\Cont[0]{Y}
  \subset \Cont[0]{U}\subset \Cont[0]{Y}$ for all
  $U\in\latO{Y}$, and,
\item
  for every $W\in \latO{Y}$ and every 
        $a\in A\setminus \Psi(W)$,
  there exists a completely positive contraction
  $T\colon A\to \Cont[0]{Y}$ 
        with $T(a)\not\in \Cont[0]{W}$
  and $T(\Psi(U))\subset \Cont[0]{U}$ for all $U\in \latO{Y}$.
\end{Aufl}
\end{LEM}
\begin{proof} 
Recall that the boundary $\partial \mathcal{Q}(A)$
of $\mathcal{Q}(A)$ is just $\{ 0\} \cup P(A)$,
where $P(A)$ denotes the pure states on $A$. 
Further recall
that $\pi\colon \xi\in P(A)\mapsto J_\xi \in \Prim{A}$
is an open and continuous epimorphism from
$P(A)$ onto $\Prim{A}$ 
(cf.\ \cite[thm.~3.4.11]{Dix}, \cite[thm.~4.3.3]{Ped}). 
Here $J_\xi$ is the kernel
of the GNS representation corresponding to $\xi$,
i.e.~$b\in J_\xi$ if and only if $\xi(AbA)=\{ 0\}$.

Let 
$\Psi'(F):=
\{ 0\} \cup 
\pi^{-1}\left(\hu \left(\Psi(Y\setminus F)\right)\right)$
for $F\in \latF{Y}$, then it is straight forward to check
that
$\Psi\colon \latO{Y}\to \latI{A}$ fulfills
condition (II) of Definition \ref{D:l-g-preserv} if and
only if $\Psi'\colon \latF{Y}\to \latF{\{0\}\cup P(A)}$
satisfies
\begin{Aufl}
\item[(II$\,'$)] 
        $\Psi' \left(\overline{\bigcup_\alpha F_\alpha} \right)
     =  \overline{\bigcup_\alpha \Psi' \left(F_\alpha\right)}$
     for every (non-empty) family 
        $\{ F_\alpha \}$ of closed subsets of $Y$.  
\end{Aufl}
It is easy to see that (II$\,'$) is equivalent to
$\Psi'(F)=\overline{\bigcup_{p\in F} \Psi'(\{ p \})}$
for every closed subset $F\not=\emptyset$ of $Y$. 
Since $K_p\cong
\mathcal{Q}\left(A/\Psi(Y\setminus \{ p\})\right)$ 
is a split face of $\mathcal{Q}(A)$ for $p\in Y$, 
it holds:
$$
\Psi'(\{p\})= K_p\cap (\{ 0\} \cup P(A))
=K_p\cap \partial \mathcal{Q}(A)
=\partial K_p\,.
$$

By Lemma \ref{L:lsc-relation}, 
$(p,\xi)\in R\mapsto p\in Y$
is open, if and only if, 
$$
(p,\xi)\in 
R\cap \left(Y\times (\{0\}\cup P(A))\right)
\to p\in Y
$$ 
is open, if and only if, 
$\Psi'(F)=\overline{\bigcup_{p\in F} \Psi'(\{ p \})}$
for every closed subset $F$ of $Y$.

\Ad{(i)} 
The function 
$p\mapsto T(a)(p):= \xi_p(a)$ is continuous
on $Y$ for every $a\in A$, because 
the selection is weakly continuous.
$T\colon A\to \Cont[b]{Y}$ is completely
positive, because $a\mapsto T(a)(p)=\xi_p(a)$ 
is completely positive for every $p\in Y$.

If $a\in \Psi(U)$ and $q\in Y\setminus U$ then
$\xi_q(a)=0$, because 
$\Psi(U)\subset \Psi(Y\setminus \{ q\})$
and $\xi_q\in K_q$.
Thus $p\mapsto T(a)(p)f(p)$ is in $\Cont[0]{U}$ for
$a\in \Psi(U)$ and $f\in \Cont[0]{Y}\cong C$.

\Ad{(ii)} By property (II),
$\Psi(W)=
\bigcap _{p\in Y\setminus W} \Psi(Y\setminus \{ p\})$.
Thus, there is $q\in Y\setminus W$ with 
$a\not\in \Psi(Y\setminus \{q\})$, and there is a
(pure) state $\chi\in K_q$ with $\chi(a)\not=0$.

Since $(p, \xi)\mapsto p$ is an open map from
$R$ onto $Y$,
$$p\in Y\mapsto K_p\subset \mathcal{Q}(A)$$
  is a lower semi-continuous family (with respect to the 
  $\sigma ( A^*,A ) $-topology on $\mathcal{Q}(A)$).
It follows by a selection theorem of Michael \cite{Michael}
that there is a continuous selection $p\mapsto \xi_p\in K_p$
with $\xi_q=\chi$.
Let $f\in \Cont[0]{Y}\cong C$ a function  with 
$0\leq f \leq 1$ and $f(q)=1$. 
Then 
$T\colon A\to \Cont[0]{Y}$ with $T(a)(p)=f(p)\xi_p(a)$
satisfies $T(a)(q)=\chi(a)\not=0$, i.e.\
$T(a)\not\in \Cont[0]{W}$. 
Hence, by (i), $T$ is as desired in (ii).
\end{proof}

\subsection{Hilbert \Cast-modules}
Now we recall the definitions of 
\Name{Hilbert} \Cast-modules,
\Name{Hilbert} $A$-bi-modules, its tensor products
and some basic results on its
module homo\-mor\-phisms  (cf.~\cite{La95} for details).

\begin{DEF} 
A \emph{pre-\Name{Hilbert} $A$-module} is 
a right $A$-module $E$ over a \Cast-algebra $A$
equipped with an $A$-valued sesqui\-linear
form
$\SP{\cdot}{\cdot}\colon E\times E\to A$ that is 
$A$--linear in the second
variable such that $\SP{e}{e}\ge0$ 
for all $e\in E$ and $\SP{e}{e}=0$ implies $e=0$.
%
%for all $a\in A$, 
%$e,f\in E$,
%\begin{Aufl}
%\item $\SP{e}{fa}=\SP{e}{f}a$,
%\item $\SP{e}{f}^*=\SP{f}{e}$,
%\item $\SP{e}{e}\ge0$ and if $\SP{e}{e}=0$ then $e=0$,
%\end{Aufl}
%

$E$ becomes a normed vector space with norm
$\|e\|:=\sqrt{\|\SP{e}{e}\|}$. 
A pre-\Name{Hilbert} 
$A$-module which is complete with respect to the norm   
induced by $\SP{.}{.}$ is called 
\emph{\Name{Hilbert} $A$-module}.

$E$ is \emph{full} if the span of 
$\SP{E}{E}$ is dense in $A$,
i.e., if
$A$ is the closure of the linear span of
$\{ \SP{e}{f}\,\colon \,\, e,f\in E \}$.

A \Name{Hilbert} $A$-module $E$ 
that is the closure of the $A$-linear span
(i.e.~of the set of finite sums $\sum e_j a_j$)
of a countable subset $\{ e_1,e_2,\ldots  \}$
of $E$ is called \emph{countably generated}
over $E$.
\end{DEF}

A map $\iota \colon E\to F$ from $E$ onto
another \Name{Hilbert} $A$-module $F$
is an \emph{isomorphism} of \Name{Hilbert} modules,
if $\iota$ is an epimorphism of $A$-modules
and satisfies $\SP{\iota(e_1)}{\iota(e_2)}=\SP{e_1}{e_2}$
for $e_2,e_2\in E$.

To keep notation simple, it is useful to note that
an \emph{isometric} $A$-module map $\iota$ from $E$ to $F$
automatically satisfies 
$\SP{\iota(e_1)}{\iota(e_2)}=\SP{e_1}{e_2}$. This follows
by polar decomposition from the fact that
for $a,b\in A_+$ holds $a\leq b$ if and only if
$\| (b+t)^{-1/2}a(b+t)^{-1/2} \| \leq 1$ for
$t\in (0,\infty)$.

\begin{EXA}\Label{Ex:H-module}\mbox{ }

\noindent
(i)     A \Cast-algebra $A$ itself becomes a (right)
  \Name{Hilbert} $A$-module $\HM[A]{A}$ with the 
  algebra-multiplication from the right
  and the scalar-product $\SP{a}{b}:=a^*b$.
  More generally, every closed right ideal 
  $R$ of $A$ becomes a
  \Name{Hilbert} $A$-module $\HM[R]{A}$ with this 
  $A$-valued sesqui\-linear form.

\smallskip
\noindent
(ii)    Sequences $( a_n ) \in A^\N$ build a right $A$-module
  with multiplication of the $a_n$ from the right by
  elements of $A$.  
  By restricting to sequences where the series
  $\sum_{n=1}^n a_n^*a_n$ converges in $A$ 
  and taking the scalar-product
  $\SP{(a_n)}{(b_n)}:=\sum_{n=1}^\infty a_n^*b_n$ we get a
  \Name{Hilbert} $A$-module, which we denote by $\HM{A}$.
  If $A$ is stable, then there is an isometric 
  $A$-module iso\-mor\-phism from $\HM{A}$ onto
  the \Name{Hilbert} $A$-module $A$ of (i) 
  (cf.~Remark \ref{R:A-cong-HA}).

\smallskip
\noindent
(iii)  More generally if $\{E_\xi \}_{\xi\in X}$ is any family
  of \Name{Hilbert} $A$-modules, then one can define a 
        Hilbert $A$-module
  $\bigoplus _{\xi\in X} E_\xi$ as the set of all maps 
  $e\colon X\to \bigcup_{\xi\in X} E_{\xi}$ with 
        $e(\xi)\in E_\xi$ and
  $\sum_\xi  e(\xi )^*e(\xi )$ convergent in $A$, 
        together with
  the obvious right $A$-module structure 
        and obvious sesqui\-linear form.

\smallskip
\noindent
(iv)    By a result of Kasparov \cite{Kas} 
        (cf.~\cite[thm.~13.6.3]{Black}, \cite[thm.~6.2]{La95})
        for every countably generated
  \Name{Hilbert} $A$-module $E$ holds 
        $E\oplus \HM{A}\cong \HM{A}$
  (by an isometric $A$-module iso\-mor\-phism), 
  i.e.,~every countably generated
  Hilbert $A$-module is the range of a self-adjoint projection 
  on $\HM{A}$.
\end{EXA}

\begin{DEF}\Label{D:LOp/K}
If $E$ is a \Name{Hilbert} A-module then 
$\LOp{E}$ denotes the set of adjoint-able maps over $E$,
i.e.~maps $T\colon E\to E$ for which there is a map 
$T^*\colon E\to E$ with
$$ \SP{e}{Tf}=\SP{T^*e}{f} \quad \text{for all } e,f\in E.$$

$\KOp{E}$ is defined as the closed linear span of the set of 
maps over $E$
given for $x,y \in E$ by
$$ \theta_{xy}\colon E\to E, e\mapsto x\SP{y}{e} .$$
\end{DEF}

Adjoint-able maps from (all of) $E$ into $E$ 
(with adjoint maps defined
on all of $E$)
are automatically $A$-module homo\-mor\-phisms, and
are bounded (by an application of the
\Name{Banach-Steinhaus} theorem).  Furthermore,
$\LOp{E}$ is a unital \Cast-algebra.  
It holds $\theta_{xy}^*=\theta_{yx}$, 
and $\KOp{E}$ is an essential closed ideal of $\LOp{E}$.
The natural *-mono\-mor\-phism from $\LOp{E}$
to $\Mult{\KOp{E}}$ is an isomorphism from $\LOp{E}$
onto $\Mult{\KOp{E}}$ that turns strong* topology
on (bounded parts of) $\LOp{E}$ to the strict topology
on (bounded parts of) $\Mult{\KOp{E}}$. Moreover,
$\KOp{E}$ is strongly Morita equivalent
to $A$ in a natural way.

\begin{EXA}\Label{Ex:M(A)}
For Examples \ref{Ex:H-module}(i) the sets are well-known 
\Cast-algebras:
$\theta_{xy}(a)=xy^*a$ and by that $\KOp{\HM[A]{A}}=A$.
The adjoint-able operators are the multipliers of $A$, 
$\LOp{\HM[A]{A}}=\Mult{A}$.

In the case of a closed right ideal $R$ of $A$,
$\KOp{\HM[R]{A}}= R^*\cap R=RR^*$ 
(a hereditary \Cast-subalgebra
of $A$),
and $\LOp{\HM[R]{A}}=\Mult{R^*\cap R}$.
\end{EXA}

Given a \Name{Hilbert} 
$B$-module $E$ and a *-representation 
$h \colon B\to\LOp{F}$ of $B$ on a 
\Name{Hilbert} $A$-module $F$, 
one can define a \Name{Hilbert} $A$-module
$E\otimes_h  F$, the \emph{(interior) tensor product}:

A right $A$-module structure $x\cdot b$ and
an $A$-valued sesqui\-linear form $\left<x,y\right>$ 
on the \emph{algebraic} vector space tensor product 
$E\odot F$ of $E$ and $F$ are given by 
$$\left<e_1\otimes f_1, e_2\otimes f_2\right>
:=\left<f_1,h\left( \left<e_1,e_2\right>\right)f_2\right>$$
and $(e\otimes f)b:=e\otimes (fb)$.
It can be shown that the linear subspace 
$L$ generated by elements of
the form $ea\otimes f - e\otimes h(a)f$ is equal to the set
$\left\{ x\in E\odot  F\fdg  <x,x>=0 \right\}$
(cf.\ proof of \cite[prop.~4.5]{La95}).
$E\otimes_h  F$ is defined as the completion of the quotient
$(E\odot F)/L $. 
Below we use the notation $e\otimes_h f$ 
(or $e\otimes _B f$ if no confusion can arise) 
for the element $(e\otimes f)+L$ in the quotient 
$(E\odot F)/L \subset E\otimes _h  F $.

\begin{REM}\Label{R:L-e}
(i)
There is a natural unital *-homo\-mor\-phism $\eta$ from 
$\LOp{E}$ into
$\LOp{E\otimes _h  F}$ given 
on elementary tensors by 
$\eta(S)\left(e\otimes _h  f\right):= S(e)\otimes _h  f$ for
$S\in \LOp{E}$, because 
%%HEREXXX change of reasoning
%
%
%\FRem{change reasoning?}
%
%
%$\eta(S)$  is well-defined  on 
%$\left(E\odot F\right)/L$, has there norm 
%$\leq \| S\|$
%and $\eta(S^*)$  is an adjoint of $\eta(S)$. 
%
%
%
%%HEREXXX
\FRem{next necessary? or cite?}
the map $T\mapsto T\otimes  \id{}\in \textrm{Lin}(E\odot F)$ 
is multiplicative and
satisfies 
$\left<x,T\otimes \id{}(y)\right>=
\left<T^*\otimes \id{}(x),y\right>$.
(It gives  that $\| T\otimes \id{}(y)\| \leq \| T\|$ if
applied to $S:=\left(\| T\|^2-T^*T\right)^{1/2}$, thus
$T\otimes \id{}(L)\subset L$, and 
$[T]_L(x+L):=T\otimes \id{}(x)+L$ extends to 
an adjoint-able operator $\eta(T)$ with norm $\leq \| T\|$.)

(ii)
Thus, if the \Name{Hilbert} $B$-module has a left $C$-module
structure given by a *-homo\-mor\-phism 
$k\colon C\to \LOp{E}$
then the \Name{Hilbert} $A$-module
$E\otimes _h  F$ has a natural left $C$-module
structure given by 
$\eta\circ k \colon C\to \LOp{E\otimes_h  F}$.

(iii)
Another map is given by 
$L_e\colon f\in F\to e\otimes f\in E\otimes_h  F$
for $e\in F$. It is easy to see that $L_e$
is a $A$-module map, $\| L_e\|\leq \| e\|$
and has the adjoint $(L_e)^*\colon E\otimes_h F\to F$
given elementary tensors by 
$(L_e)^*(e'\otimes_h f)=h(\SP{e}{e'})f$.
\end{REM}

\begin{EXA}\Label{Ex:tensor}
        If $h\colon B\to \Mult{A}$ is a *-homo\-mor\-phism, 
        then there is
  a natural isometric \Name{Hilbert} $A$-module 
        iso\-mor\-phism $I$ from
  $B\otimes_h A$ onto the closed right ideal 
  $R_h:=\overline{\Span{h(B)A}}$. 
  $I$ is given by $I\left(b\otimes_h a\right):=h(b)a$.

$\LOp{B\otimes_h A}\cong \Mult{R_h^*\cap R_h}$
(cf.\ Example \ref{Ex:M(A)}), and $I\eta (\,\cdot\,)I^{-1}$
is the unital strictly continuous
*-homo\-mor\-phism from $\Mult{B}$ into $\Mult{R_h^*\cap R_h}$
that extends the non-degenerate 
*-homo\-mor\-phism 
$\iota\circ h\colon B\to \Mult{R_h^*\cap R_h}$,
where $\iota$ is the restriction *-homo\-mor\-phism
from 
$\{ t\in \Mult{A}\fdg tA+t^*A\subset R_h \}$
into $\Mult{R_h^*\cap R_h}$.

$R_h=A$ if and only if 
$h\colon B\to \Mult{A}$ is non-degenerate.
Then $B\otimes _h A\cong A$, and
under this isomorphism 
$L_e$ becomes $h(e)$ and $(L_e)^*=h(e^*)$,
and $\eta\colon B\to \LOp{B\otimes_h A}$
becomes the natural extension of
$h$ to a unital strictly continuous map
$\Mult{h}\colon \Mult{B}\to \Mult{A}$.
%Since $h$ is non-degenerate, 
%$h$ extends uniquely to a *-homo\-mor\-phism 
%$\Mult{h}\colon \Mult{B}\to\Mult{A}$ with 
%$\Mult{h}(c)(h(b)a)=h(cb)a$
%for $c\in \Mult{B}$, $b\in B$ and $a\in A$.
(We denote 
$\Mult{h}$ also by $h$ to keep notation simple.)

%
%\begin{Aufl}
%\item 
%   $\HM{A}$ is the tensor product $\ell_2(\N)\otimes _h  A$
%   of the \Name{Hilbert} $\C$-module $\ell_2(\N)$
%   and the \Name{Hilbert} $A$-module $A$ with respect to
%   the *-homo\-mor\-phism 
%       $h \colon z\in \C\mapsto \Mult{A}=\LOp{\HM[A]{A}}$.
%
%\end{Aufl}
\end{EXA}
\begin{DEF}\Label{D:H-bi-module}
A \emph{\Name{Hilbert} ($B$,$A$)-bi-module} is
a right \Name{Hilbert} $A$-module $E$ together with a left
$B$-module structure given by a *-homo\-mor\-phism 
$h\colon B\to \LOp{E}$.
I.e.\ let $h \colon B\to\LOp{E}$ be a 
*-homo\-mor\-phism, then $E$ is an 
 ($B$,$A$)-bi-module
with the left multiplication
given by $a\cdot e:=h (a)e$.

$E$ is \emph{full} if $E$ is full as (right) 
\Name{Hilbert} $A$-module.
\FRem{Is this kind of "full" needed?}

$E$ is \emph{non-degenerate} if 
the linear span of $h(B)E$ is dense in $E$,
i.e.\ if $h(B)E=E$ (by Cohen factorization theorem).

We denote by $\Hm{A,h}$ the 
non-degenerate Hilbert $(A,A)$-bi-module that
is given by the Hilbert $A$-module $E=A_A$ of
Example \ref{Ex:H-module}(i) and a 
\emph{non-de\-gener\-ate}
*-homomorphism $h\colon A\to \LOp{E}=\Mult{A}$.
\end{DEF}
\medskip

\begin{REM}\Label{R:tensor} 
Suppose that $h_i\colon A\to \Mult{A}$ ($i=1,2$)
are  \emph{non-de\-gener\-ate}
*-homo\-mor\-phisms. 
Then $h_i$ uniquely extends to a strictly continuous
unital *-homo\-mor\-phism 
$h_i\colon \Mult{A}\to \Mult{A}$
(cf.\ Example \ref{Ex:tensor}).
Since $h_2\circ h_1$ is strictly continuous and unital, 
we get that 
$\left(h_2\circ h_1|A\right) \colon A\to \Mult{A}$ 
is non-degenerate.
Thus we can
define a non-degenerate *-homo\-mor\-phism
$h_2\circ h_1\colon A\to \Mult{A}$ as the 
restriction
of the strictly continuous unital *-homo\-mor\-phism 
$h_2\circ h_1$ from $\Mult{A}$ into $\Mult{A}$.

We have seen in Example \ref{Ex:tensor}, that there is
a \Name{Hilbert} $A$-module iso\-mor\-phism $I$ from 
$A\otimes_{h_2} A$ onto 
the (right) \Name{Hilbert} $A$-module
$A$ of Example \ref{Ex:H-module}(i). $I$ is
given by $I\left(a\otimes _{h_2} b\right):=h_2(a)b$.
The left $A$-module structure on 
$A\otimes_{h_2} A$ is defined by
$a\cdot \left(b\otimes_{h_2} c\right)
:=\eta (h_1(a))(b\otimes_{h_2} c)=
\left(h_1(a)b\right)\otimes_{h_2} c$
for $a,b,c\in A$.

Thus, 
$I\left(a\cdot \left(b\otimes _{h_2} c\right)\right)
=
h_2\left(h_1(a)b\right)c
=
h_2\circ h_1(a)I\left(b\otimes _{h_2} c\right)$, 
and $I$ is also a left $A$-module iso\-mor\-phism
from the left $A$-module $A\otimes_{h_2} A$ onto the 
left $A$-module $A$ given by 
$\left(h_2\circ h_1\right)\colon A\to \Mult{A}$.

The $n$-fold
tensor product 
$E_1\otimes _{h_2} E_2\otimes_{h_3}\dots \otimes_{h_n} E_n$ 
of the \Name{Hilbert} $A$-bi-modules $E_i$ given by 
non-degenerate $h_i\colon A\to \Mult{A}$ 
is defined  inductively by
$$
E_1\otimes _{h_2} E_2\otimes_{h_3}\dots 
\otimes_{h_n} E_n
:= 
\left(
E_1\otimes _{h_2} E_2\otimes_{h_3}\dots 
\otimes_{h_{n-1}} E_{n-1}
\right)
\otimes _{h_n}E_n\,.
$$
By induction, it is
isomorphic to the 
\Name{Hilbert} $A$-module $A$ 
of Example \ref{Ex:H-module}(i) with left $A$-module structure
given by 
$h_n\circ h_{n-1}\circ \dots \circ h_1\colon A\to \Mult{A}$,
and there is a natural \Name{Hilbert} $A$-module isomorphism
$$
E_1\otimes _{h_n\circ\dots\circ h_2} 
(E_2\otimes _{h_3}\dots\otimes _{h_n} E_n)\cong
E_1\otimes _{h_2} E_2\otimes_{h_3}\dots \otimes_{h_n} E_n\,.
$$
\end{REM}

\smallskip

Some \Name{Hilbert} bi-modules 
are related to completely positive maps:
%%%%%%%%%%%%%%%%%%%%%%%%%%%%%%%%%%%%
\begin{LEM}\Label{L:bi-mod-from-cp-map}
For every
completely positive map $T\colon A\to B$
there are a \Name{Hilbert} $B$-module $E^T$ and
a  *-homo\-mor\-phism 
$h^T\colon A\to \LOp{E^T}$, such that
$h^T$ is non-degenerate and $(E^T,h^T)$ satisfies:

A completely positive map $V\colon A\to B$
can be approximated in point-norm
by maps
$V_e\colon A\to B$ given by $V_e(a):=\SP{e}{h^T(a)e}$,
if and only if,
$V$ can be approximated by by completely positive
maps 
$T_{r,c}\colon A\to B$ given by
$T_{r,c}(a):=c^*T\otimes \id{n}(r^*ar)c$
for a row-matrix 
$r\in M_{1,n}(A)$ and a column-matrix 
$c\in M_{n,1}(B)$, $n\in \N$.

$E^T$ can be taken countably generated over $B$
if $A$ is separable.
\end{LEM}

\begin{proof} The algebraic vector space
tensor product 
$A\odot B$ has natural left $A$-module
and right $B$-module structures.
Define a $B$-valued sesqui\-linear form $\beta(x,y)$
on elementary tensors by
$$\beta(a_1\otimes b_1, a_2\otimes b_2)
:=b_1^*T(a_1^*a_2)b_2\,.$$
It is $B$-linear in the second variable,
satisfies $\beta(x,x)\ge 0$, 
$\beta(a\cdot x, a\cdot x)\leq \| a\|^2\beta(x,x)$,
and $\beta(a^*\cdot x,y)=\beta(x, a\cdot y)$
for all $x,y\in A\odot B$, $a\in A$.

Thus, the subspace 
$L:=\{ x\in A\odot B\fdg \beta(x,x)=0\}$
is $A$- and $B$-invariant, and $\beta$ defines
on $(A\odot B)/L$ a $B$-valued
scalar product $\SP{x}{y}$,which 
defines a pre-\Name{Hilbert} $B$-module. The
completion $E^T$ of $(A\odot B)/L$ and 
$\SP{\cdot }{\cdot }$ is a \Name{Hilbert}
$B$-module with a *-homo\-mor\-phism
$h^T\colon A\to \LOp{E^T}$.

By Cohen factorization, for every
$x\in A\odot B$ there is $y\in A\odot B$ and
$a\in A$ with $(a\otimes 1_{\Mult{B}})y=x$. 
It implies that $h^T(A)E^T$ is dense in $E^T$.

The characterization of c.p.~maps which can be
approximated by maps $V_e$ follows immediately from
the definition of $E^T$.

If $A$ is separable, let $c\in C_+$ a strictly
positive element of the separable \Cast-subalgebra $C$
of $B$ that is generated by $V(A)$, and let
$\{ a_1,a_2,\ldots \}$ a dense sequence in $A$.
Then the $B$-module generated by
$\{ (a_1\otimes c)+L, (a_2\otimes c)+L, \ldots \}$
is dense in $E^T$.
\end{proof}

\subsection{Crossed products by $\Z$}
\Label{ssec:crossed-prod}
\FRem{rewrite subsection?}
\begin{REM}\Label{Rcross-prod}
Suppose that $\varphi\colon E\to F$ a *-homo\-mor\-phism, 
$\alpha\in\Aut{E}$, $\beta\in\Aut{F}$ with 
$\varphi\circ\alpha=\beta\circ\varphi$. Recall the following
elementary properties of crossed products by $\Z\,$:
\begin{Aufl}
\item The universal crossed product $E\rtimes_\alpha\Z$ is the
  same as the reduced crossed product $E\rtimes_{\alpha,r}\Z$.
  
  $E\rtimes_{\alpha,r}\Z$ is the image 
  of the *-representation $d_\textrm{red}$ 
  from $E\rtimes_\alpha \Z$ into the von-Neumann algebra
  $M\overline{ \otimes }\LOp{\ell_2(\Z )}$ given
        by 
$$
        d_\textrm{red}(e):=
        (\ldots, \alpha^{-1}(e),e,\alpha(e),\alpha^2(e),\ldots )
        \in \ell_\infty (E^{**})
        \subset M\overline{ \otimes }\LOp{\ell_2(\Z )}
$$
        for $e\in E$ and
        $\Mult{d_\textrm{red}}(U_0):= 1\otimes S_1$. Here
  $M$ is any W*-algebra that contains $A$ as
  a weakly dense \Cast-subalgebra, and
  $S_1(f)(n):=f(n+1)$ for
        $f\in \ell_2(\Z)$ and $U_0$ is the canonical
        generator of $\Z$ in $\Mult{E\rtimes_\alpha \Z}$
        with $U_0eU_0^*=\alpha(e)$ for $e\in E$ 
        (see (v) below).

\item $E$ is naturally isomorphic to a \Cast-sub\-al\-ge\-bra
  of $E\rtimes_\alpha\Z$.
\item $\varphi$ defines a *-homo\-mor\-phism
  $ \varphi\rtimes\Z\colon E\rtimes_\alpha\Z \rightarrow 
  F\rtimes_\beta\Z$ with 
$$
\varphi(E)=F\cap 
\left(\varphi\rtimes\Z\left(E\rtimes_\alpha \Z\right)\right).
$$
  $\varphi\rtimes \Z$ is a 
        mono\-mor\-phism if $\varphi$ is a mono\-mor\-phism. 
\item If $\varphi(E)$ is an ideal of $F$, then 
        $E\rtimes_\alpha\Z$
  maps onto an ideal of $F\rtimes_\beta\Z$.  
%       This ideal is essential
%  in $F\rtimes_\beta\Z$ if $\varphi(E)$ is essential in $F$.
\item The non-degenerate *-representations 
  $\varrho\colon E\rtimes_\alpha\Z\to\LOp{H}$ 
        are in one-to-one
  correspondence to covariant
        representations  $\left( \varrho_E,U, H \right)$ 
(where 
        $\varrho_E\colon E\mapsto\LOp{H}$
        is a non-degenerate *-representation and
        $U\in\LOp{H}$ is a unitary with
        $\varrho_E \left( \alpha(e) \right) 
  =U\varrho_E(e)U^*$ for $e\in E$). 
  The image of $\varrho$ is generated by
  $U\varrho_E(E)$ as \Cast-algebra.
\item The crossed product
  $(E,\alpha)\mapsto E\rtimes_\alpha \Z$ defines
  in a natural manner an exact functor from the category of
  dynamical \Cast-systems into the  category of
  \Cast-algebras (as follows from (i), (iii) and (iv)).
\end{Aufl}
\end{REM}
\begin{LEM}\Label{L:inner-dual-action}
Let $\alpha\in\Aut{B}$, $B$ unital.  
If, for $\theta\in\C$ with
$|\theta|=1$,
there is a unitary $v(\theta)$ in the center of 
$B$ such that
$\theta \alpha( v(\theta)) = v(\theta)$, 
then every ideal $I$ of
$B\rtimes_\alpha\Z$ is invariant under the dual action 
$\hat\alpha$
of $\mathbb{T}=S^1$ on $B\rtimes_\alpha\Z$ 
and, thus, is determined
by its intersection $I\cap B$ with $B$, i.e.\ 
$I$ is the natural image of $(I\cap B)\rtimes _\alpha \Z$
in $B\rtimes _\alpha \Z$.
\end{LEM}

\begin{proof}
The ideals $I$ of $B\rtimes_\alpha\Z$  are invariant 
under the dual action $\widehat{\alpha}$ of
$\mathbb{T}=S^1$ on $B\rtimes_\alpha\Z$,
because  
$\widehat{\alpha}(\theta)=v(\theta)(\, .\, )v(\theta)^*$
for $\theta\in \mathbb{T}$.

There is no ideal $\not= 0$ on a crossed-product 
$B\rtimes_\alpha \Z$ that is 
invariant under the dual action of 
$\mathbb{T}$ on $B\rtimes_\alpha \Z$
\emph{and} is orthogonal to $B$, because the integral
over the dual action is a \emph{faithful} conditional
expectation from $B\rtimes _\alpha \Z$ onto $B$.
This remains true for every  quotient of $B$ by 
an $\alpha$-invariant  ideal $J$ of $B$. 
Thus, every non-zero closed ideal of
$$(B\rtimes_\alpha \Z)/(J\rtimes_{\alpha|J} \Z)
\cong (B/J)\rtimes _{[\alpha]}\Z$$
intersects $B/J$. 
(See Remark \ref{Rcross-prod}(vi) for the isomorphism.)

If $I$ is a closed ideal of $B\rtimes_\alpha\Z$  
then $J:=I\cap B$ is $\alpha$-invariant,
the natural image of $J\rtimes \Z$ is contained in
$I$ and the natural image of
$I$ in $(B/J)\rtimes \Z$
has zero intersection with $B/J$, thus
$I$ is the natural image of $(I\cap B)\rtimes \Z$.
\end{proof}

{}For the following Lemma \ref{L:B-rtimes-Z}
recall that:
\begin{Aufl}
\item 
$
        \ell_\infty \left( \Mult{A} \right)  
  \subset\Mult{\Mult{A}\otimes\K}
        \subset \Mult{A\otimes\K}
        \cong
        \LOp{\HM{A}}
$
        where $\HM{A}$ is as in Example \ref{Ex:H-module}(ii),
\item 
        $\ell_\infty \left( \Mult{A} \right)\cap 
        \left(\Mult{A}\otimes\K\right)=c_0(\Mult{A})$,  
\item   
        $\TT\in 1\otimes\LOp{\ell _2}
  =1_{\Mult{A}}\otimes \Mult{\K}
        \subset
        \Mult{\Mult{A}\otimes\K}$
        where  $\TT:=1\otimes \TT_0$ with 
$\TT_0\in \LOp{\ell_2}$
        is the Toeplitz operator (= forward shift) on 
        $\ell_2:=\ell_2(\N )$.
\item Let
  $Q^s\left( \Mult{A} \right)
  := \Mult{\Mult{A}\otimes\K}/\left(\Mult{A}\otimes\K\right)$
        the \emph{stable corona} of $\Mult{A}$, and
\item  let $B:= \ell_\infty(\Mult{A})/c_0(\Mult{A})$.
        By (ii), there is a natural unital *-monomorphism
  $\epsilon\colon B\hookrightarrow 
  Q^s\left(\Mult{A}\right)$.
\item
  $\sigma\in \Aut{B}$ denote the automorphism of $B$ 
        that is induced by the forward shift
  on $\ell_\infty \left(\Mult{A}\right):=
        \ell_\infty\left( \N,\Mult{A} \right)$.
%\item
\end{Aufl}

\begin{LEM}\Label{L:B-rtimes-Z}
  $$
        \Mult{A}\otimes\K \subset
  \Cast \left(\ell_\infty\left(\Mult{A}\right),\TT\right) 
        \subset \Mult{\Mult{A}\otimes\K},
        $$ 
        and
  $
    \Cast \left(\ell_\infty\left(\Mult{A} \right) ,\TT \right)/
    \left(\Mult{A}\otimes\K\right)
  $
  is naturally isomorphic to $B\rtimes_\sigma\Z$ by the
  *-homo\-mor\-phism that is defined by the covariant
        representation $(\epsilon, U)$ of $(B,\sigma)$ into 
        $Q^s\left( \Mult{A} \right)$.

  Here
  $
    B:= \ell_\infty \left( \Mult{A} \right) /
        c_0 \left( \Mult{A} \right)\, 
  $,
  $\sigma$ is induced by the forward shift on
  $\ell_\infty\left(\Mult{A}\right):=
        \ell_\infty\left( \N,\Mult{A}\right)$, and 
        $U:=\TT + \Mult{A}\otimes\K\,$.
\end{LEM}
\begin{proof} Since
$ 
1\otimes\K= \left( 1\otimes\LOp{\ell _2} \right) \cap
 \left( \Mult{A}\otimes\K \right)\, 
$ and $1-\TT_0\TT_0^*\in \K\,$,
the image $U:=\TT + \Mult{A}\otimes \K$ 
of $\TT=1\otimes\TT_0$ in
$Q^s\left(\Mult{A}\right)$
is unitary.  
A straight calculation shows that
$a\mapsto \TT a \TT ^*$ is the forward shift on 
$\ell_\infty \left( \Mult{A} \right)$, thus
$\epsilon(\sigma( b))=U\epsilon(b)U^*$
for $b\in B$.

$\Mult{A}\otimes \K$ is contained in
$\Cast \left(\ell_\infty\left(\Mult{A}\right),\TT\right)$,
because $\Cast \left(\TT _0 \cdot c_0\right)=\K$
and
%$c_0 \left( \Mult{A} \right)=
%\ell_\infty \left( \Mult{A} \right) \cap 
%\left( \Mult{A}\otimes\K \right) 
%$, 
$
\Cast \left(\TT c_0\left(\Mult{A}\right)\right)=
\Mult{A}\otimes \Cast \left(\TT _0 \cdot c_0\right)\,$.
$$ 
\Cast \left(\ell_\infty\left(\Mult{A}\right),\TT\right)/ 
\left(\Mult{A}\otimes\K\right) =\Cast  (U\epsilon(B))
$$
is the image of 
the natural *-homo\-mor\-phism $\rho$ from
$B \rtimes_\sigma \Z $ into $Q^s\left(\Mult{A}\right)$
defined by the
covariant representation $(\epsilon ,U)$.
The restriction $\rho|B=\epsilon$ of $\rho$ to
$B$ is faithful.

{}For $\theta \in\C$ with $|\theta|=1$, the image 
$v(\theta)\in B$ of
$( \theta 1,\theta^2 1,\dots )$ 
is in the center 
of $B$ and satisfies
$\theta \sigma(v(\theta)) = v(\theta)$.
Thus, the kernel of the natural homo\-mor\-phism
$\rho\colon B\rtimes_\sigma\Z \to Q^s\left(\Mult{A}\right)$
must be zero by Lemma \ref{L:inner-dual-action}.
\end{proof}

Let $C$ a \Cast-algebra. We denote by
$\wcl{X}$ the ultra-weak closure of a subspace 
$X$ of 
the $W^\ast$-algebra $C^{**}$.  Recall that $\wcl{X}$ is
naturally isomorphic to $X^{**}$ if $X$ is a closed subspace 
of $C$ (by Hahn-Banach extension).  
We write also $\alpha$ for
the second conjugate $\alpha^{**}\in \Aut{A^{**}}$ 
of $\alpha\in \Aut{A}$.

\begin{LEM}\Label{L:second-dual}
Let $\alpha\in \Aut{A}$. 
There is a natural *-homomorphism $\eta$
from $A^{**}\rtimes _\alpha \Z$ to the second
conjugate $(A\rtimes _\alpha \Z)^{**}$. $\eta$ is faithful
and satisfies
\begin{Aufl}
\item $\eta| A^{**}$ is the second conjugate
of the inclusion map 
$A\hookrightarrow A\rtimes _\alpha \Z$, and
\item $\eta| (A\rtimes _\alpha \Z)$ is the natural inclusion
of $A\rtimes _\alpha \Z$ into it second conjugate.
\end{Aufl}
\end{LEM}
\begin{proof} 
$A$ is a non-degenerate \Cast-subalgebra of
$C:=A\rtimes_\alpha \Z$ by Remark \ref{Rcross-prod}(ii). 
Thus the inclusion map
$\epsilon \colon A\to C$ induces a unital
isometric isomorphism $\epsilon ^{**}$ from  $A^{**}$
onto a W*-subalgebra of $C^{**}$. 

$\alpha$ extends naturally to an automorphism
of $\Mult{A}\subset A^{**}$ 
(that we also denoted by $\alpha$).
$A\rtimes _\alpha \Z$ is an 
ideal of 
$\Mult{A}\rtimes _\alpha \Z$ by Remark
\ref{Rcross-prod}(iii,iv).  The
natural inclusion of $A\rtimes _\alpha \Z$ into the
von-Neumann subalgebra 
$C^{**}\subset \LOp{H}$ (for suitable $H$)
is non-degenerate. 
Thus,
there is a unital *-homomorphism $\gamma$
from $\Mult{A}\rtimes _\alpha \Z$ into
$C^{**}$ that extends the
natural inclusion of $A\rtimes _\alpha \Z$ into the
von-Neumann subalgebra $C^{**}$ of $\LOp{H}$.

Let $U_0\in \Mult{A}\rtimes _\alpha \Z\subset
A^{**}\rtimes _\alpha \Z$ the
unitary that generates the copy of $\Z$
with $\alpha(a)=U_0aU_0^*$, and 
let $U:=\gamma(U_0)\in C^{**}$, then 
$U\epsilon(a)U^*=\epsilon(\alpha(a))$
for $a\in A$.
Thus $U\epsilon^{**}(a)U^*=\epsilon^{**}(\alpha(a))$
for all $a\in A^{**}$, and there is a unique
*-epi\-mor\-phism $\eta$ from $A^{**}\rtimes _\alpha \Z$
onto $\Cast (\epsilon^{**}(UA^{**})) =
\Cast(\epsilon^{**}(A^{**}),U)$
with $\eta(a)=\epsilon^{**}(a)$ and $\eta(U_0)=U$
(cf.\ Remark \ref{Rcross-prod}(v)).
Clearly the restriction of $\eta$ to
$A\rtimes _\alpha \Z$ is the natural inclusion into
$C^{**}$.

If we compose $\eta$ with the normalization
$C^{**}\to A^{**}\overline{\otimes} \LOp{\ell_2(\Z)}$
of the regular representation $d_{\textrm{red}}$ 
of $A\rtimes _\alpha \Z$ into 
$A^{**}\overline{\otimes} \LOp{\ell_2(\Z)}$,
then we get the regular representation of 
$A^{**}\rtimes _\alpha \Z$. The latter is
faithful. Thus $\eta$ is faithful.
\end{proof}

\begin{PRP}\Label{P:e-n-in-second}
Suppose that for $\alpha\in \Aut{A}$, there are
projections
$\{ e_n \fdg   n\in \Z  \}$ 
in the center of $A^{**}$ with 
\begin{Aufl}
\item
        $\alpha(e_n)=e_{n+1}\leq e_n$
        for $n\in \Z$, and
\item
        $\lim _{n\to\infty} \left\| a (e_{-n}-e_{n})\right\|
        =\| a\|$
        for $a$ in a dense subset $S$ of $A$.
\end{Aufl}

Then every *-representation 
$\rho\colon A\rtimes _\alpha \Z\to \LOp{\mathcal{H}}$
with faithful restriction $\rho| A$ is
a faithful representation of $A\rtimes _\alpha \Z$.
\end{PRP}
The latter equivalently means 
that every non-zero closed ideal
$I$ of $A\rtimes _\alpha \Z$ has non-zero intersection
with $A$.

\begin{proof}
Let $p:=\bigvee_{n\in\Z}e_n-\bigwedge_{n\in\Z}e_n$.
$\alpha \left( \bigwedge e_n \right) =\bigwedge e_n$ and
$\alpha \left( \bigvee e_n \right) =\bigvee e_n$ in $A^{**}$
by (i).  Clearly, $p$ is in the center of 
$ \left(  A \right) ^{**}$ and
$p=\sum_{n\in\Z} e_n-e_{n+1}$.  
Thus $\alpha(p)=p$ and 
$\epsilon \colon a\in A\to ap$ is a
*-homo\-mor\-phism which commutes with $\alpha$.  
By (ii), $\|a\|= \lim _{n\to\infty} \left\| a (e_{-n}-e_{n})\right\|
        \leq \| ap\|$ 
for $a\in S$.

The kernel of $\epsilon$ is an 
$\alpha$-invariant closed ideal 
$I$ of $ A$.  Let $d\in I$, i.e.~ $dp=0$. Since $S$
is dense in $A$, there is a sequence  
$(b_1, b_2, \ldots) $ in $S$
with $\lim _k \| b_k-d\| \to 0$.  Then
$\|b_k\|=\| b_k p \|=\| (b_k-d)p \|\leq \| b_k- d\|$
implies $\lim _k \| b_k \| =0$ and
$\| d | \leq \lim _k (\| b_k\| +\| b_k-d \|)=0$.
Thus $d=0$, and $\epsilon$ is faithful.

Let $\theta\in \C$ with $| \theta |=1$. Then
$v(\theta ):=\sum_{n\in\Z} (e_n-e_{n+1})\theta^n$
is a unitary in the center of $A^{**}p$ with
$\theta \alpha(v(\theta ))= v(\theta )$.
%, because
%$\alpha \left( e_n-e_{n+1} \right) =e_{n+1}-e_{n+2}$. 
Hence, every 
*-homo\-mor\-phism $h$
from $(A^{**}p)\rtimes \Z$ into a \Cast--algebra $G$
with faithful restriction $h|(A^{**}p)$ is itself faithful
by Lemma \ref{L:inner-dual-action}.

Let $\varrho \colon A \rtimes_\alpha\Z \to\LOp{H}$ a
*-representation of $A \rtimes_\alpha\Z$ such that
the restriction $\varrho|A$ is faithful,
and let 
$\varphi \colon (A \rtimes_\alpha\Z)^{**} \to\LOp{H}^{**}$
denote the second conjugate *-homomorphism $\varrho^{**}$.

Then $\varphi | A^{**}$ is also faithful, in particular
$\varphi | (A^{**}p$ is faithful.

By Lemma \ref{L:second-dual} there is a *-monomorphism $\eta$
from $A^{**}\rtimes _\alpha \Z$ into the second
conjugate of $A\rtimes _\alpha \Z$ such that
$\eta| A^{**}$ is the second conjugate of the
inclusion map $A\hookrightarrow A\rtimes_\alpha \Z$
and $\eta| (A\rtimes _\alpha \Z)$ is the natural
embedding from $A\rtimes _\alpha \Z$ in its
second dual (the universal W*-crossed product 
$A^{**}\rtimes _{(\alpha,W^*)} \Z$).

Since $\alpha(p)=p$ and $p$ is in the center of
$A^{**}$ we get that $p$ is in the center of 
$A^{**}\rtimes _\alpha \Z$ and that
$(A^{**}p)\rtimes _\alpha \Z$ is naturally
isomorphic to $(A^{**}\rtimes _\alpha \Z)p$.
It follows that the *-homomorphism
$d\mapsto dp$ from $A\rtimes _\alpha \Z$
into $(A^{**}\rtimes _\alpha \Z)p$ is faithful,
because
$a\mapsto ap$  is faithful 
on $A$ and is $\alpha$-equivariant and
Remark \ref{Rcross-prod}(iii) applies.

Let $\lambda :=\varphi\circ \eta$.
Then 
$\lambda\colon A^{**}\rtimes_\alpha \Z\to \LOp{H}^{**}$
satisfies
$\lambda (bp)=\varrho (b)\varphi(\eta(p))$
for $b\in A\rtimes _\alpha \Z$, and
$\lambda| (A^{**}p)$ is faithful. 
Thus $\lambda| (A^{**}\rtimes _\alpha \Z)p$ is faithful.
It follows that $\varrho$ must be faithful.
\end{proof}

%%%%%%%%%%%%%%HERE% begin of TEMP
\TEMP{ 
\section{Additional remarks}

\begin{REM}
We outline the idea of our proof:  
The relation $R_\pi$
for a pseudo-open map $\pi$ defines a map 
$\Gamma_\pi\colon p\in P\mapsto K_p\subset \Cont[0]{P}^*$ 
from $P$ into
closed faces $K_p$ of the simplex of probability-measures on 
$P$ by
letting $K_p$ be the set of 
probability-measures with support in
$\pi^{-1}\left(\overline{\{\pi(p)\}} \right)$.  
This is a lower semi-continuous
set-valued map, so, by the Michael selection theorem 
\cite{Michael},
for each
$p\in P$ and $(p,q)\in R_\pi$ there exists a 
$\sigma \left(\Cont[0]{P}^*,\Cont[0]{P}\right)$-continuous map 
$\varrho$ from
$P$ into $\Cont[0]{P}^*$ with $\varrho (p)=\delta_{q}$ and
$\varrho (r)\in K_r$ for all $r\in P$
(i.e.\ a continuous selection through $p$ and $q$).
$\varrho$ defines a positive (and hence completely positive) 
contraction 
$T\colon \Cont[0]{P}\to \Cont[0]{P}$ by 
$T(f)(p):=\varrho (p)(f)$.  
It follows that the
set $\mathcal{C}(R_\pi)$ of completely 
positive contractions $T$ of $\Cont[0]{P}$ with 
$T(I)\subset I$ for every closed ideal $I$ of $\Cont[0]{P}$ 
with $R_\pi$-invariant support (i.e.\ for every
$I$ that has support 
$\pi^ {-1}U$ for some open 
subset $U$ of $X$) determines the family of 
$R_\pi$-invariant
open subsets of $P$.  
Since $\Cont[0]{P}$ is separable, there is a 
countable sequence 
$T_1, T_2,\dots$ in $\mathcal{C}(R_\pi)$ 
which is dense in 
$\mathcal{C}(R_\pi)$ with respect to the point-norm topology.  
If we take Kasparov-Stinespring dilations of the $T_n$ and
take direct sums and infinite repeats 
and stabilize by the compact
operators we get a 
*-homo\-mor\-phism $\, h\,$ from $\Cont[0]{P,\K}$ into
the multiplier algebra $\Mult{\Cont[0]{P,\K}}$ of 
$\Cont[0]{P,\K}$ with 
the property that $\, h\,$ 
is unitarily equivalent to its infinite repeat, and that,
for closed ideals $I$ of $A:=\Cont[0]{P,\K}$ and $b\in A$,
\begin{Aufl}
\item $h(I)A\subset I$ and $h(b)A\subset I$
  imply $b\in I$ (i.e.\ $h(I)A\subset I$ implies
  $h(A)\cap \Mult{A,I}=h(I)$ in the notation of
        Definition \ref{D:Mult}),
\item $h(I)A \subset I$ 
  if and only if 
  the support of $I$ in
  $\Prim{A}\cong P$ is an $R_\pi$-invariant 
  open subset of $P$.
\end{Aufl}
If $\pi$ is pseudo-epimorphic then one can manage that 
$h\colon A\to \Mult{A}$ is
injective and non-degenerate.  It 
defines a \Name{Hilbert} $A$-bi-module which we use
to construct our crossed product $A_X:= E\rtimes_\sigma \Z$.
$A_X$ is isomorphic to the 
\Name{Cuntz--Pimsner} algebra of the bi-module.  
We calculate its
Hausdorff lattice of closed ideals by rather elementary means
in Section \ref{sec:ideal-of-cross}.  
It turns out that this lattice is isomorphic to the lattice of
$R_\pi$-invariant open subsets of $P$. This completes the
proof of our main Theorem \ref{T:main}.

\medskip

More generally we show that,
if $A$ and  $B$ are separable and stable and if $B$
contains an Abelian regular
\Cast-sub\-al\-ge\-bra $C$,
then every map $\Psi$ from $\latO{\Prim{B}}$ to 
$\latO{\Prim{A}}$ with
the properties (I) and (II) 
of Definition \ref{D:l-g-preserv} 
can by realized by a non-degenerate
*-mono\-mor\-phism 
$h\colon A\to \Mult{B}$ in a way that the ideal 
$I:=\ke \left(\Prim{A}\setminus\Psi(U)\right)$ 
%of $A$ given by $\Psi(U)$ 
for an open subset $U$ of 
$\Prim{B}$ is given by the formula
\begin{equation}
  I=h^ {-1} \left(h(A)\cap \Mult{B,J}\right) 
\end{equation}
where 
$J:=\ke \left( \Prim{B}\setminus U \right)$ 
is the closed ideal of $B$ corresponding to $U$. 
Moreover we construct $\, h\,$ such that 
$h(A)$ is contained in the
multiplier algebra $\Mult{D}$ of a non-degenerate 
\Cast-sub\-al\-ge\-bra  
$D\subset B$ with
$D\cong C\otimes\K$ and that 
$\, h\,$ is unitarily equivalent to its
infinite repeat $\delta_\infty\circ h$
(cf.\ Remark \ref{R:delta-infty}).  
With this properties $\, h\, $ is uniquely determined up to
unitary homotopy (cf.\ Definition \ref{D:unitary-htpy} and 
\cite[cor.~5.1.6]{K.book}).
\end{REM}

\begin{REM}
More general:
If $X$ is $\T{0}$ we call a subset $Z$ of 
$X$ a pseudo-G$_\delta$
set if $Z$ is intersection of a countable sequence 
$Z_n=F_n\cup U_n$
where $F_n$ is closed and $U_n$ is open in $X$.
If $X$ is the image of a Polish space by a
continuous and open epi\-mor\-phism, 
then $X$ is point-complete and
second countable and every pseudo-G$_\delta$ subset of 
$X$ has the \Name{Baire} property 
(i.e.~an intersection of a countable 
family of dense open subsets is dense).

\FRem{replace by \cite{Ped}/\cite{DiniEK1}?}

The point-completeness comes as follows:
Let $\{U_n\}_{n\in \N}$ be a countable base 
of the topology of a prime space
$X$ with Baire property, then $U_n$ is dense
in $X$ for $n\in \N$.
Hence there is a point
$x$ in the intersection $\bigcap_{n\in \N} U_n$. 
Necessarily, $X=\overline{\{x\}}$.
\end{REM}

\begin{REM}
By duality of open and closed sets of a topology 
one can define
a map
$$\Psi'\colon F\in \latF{X}\mapsto 
F\mapsto Y\setminus\Psi(X \setminus F)\in \latF{Y}.$$ 
This map satisfies the following properties 
(I$'$), (II$_0'$) and (III$'$):
\begin{Aufl}
\item[(I$'$)] % $\Psi$ is non-degenerate, i.e.~
  $\Psi'(X)=Y$ and 
  $\Psi'(\emptyset)=\emptyset$.
%\item[(II$\,'$)] $\Psi'(\overline{\bigcup_\alpha F_\alpha}) 
%  =\overline{\bigcup_\alpha\Psi'(F_\alpha)}$ 
%  for every family of closed 
%  subsets $F_\alpha\subset X$. 
\item[(II$_0'$)] $\Psi' ( F\cup G ) 
  =\Psi(F)\cup\Psi(G)$ for all closed
  subsets $F,G\subset X$
\item[(III$'$)] $\Psi'\left( \bigcap_\alpha F_\alpha \right) 
  =\bigcap_\alpha\Psi'\left(F_\alpha\right)$ 
  for every family of closed subsets $F_\alpha\subset X$.
\end{Aufl}
\end{REM}

\begin{REM}
Pre-\Name{Hilbert} 
$A$-modules are always non-degenerate $A$-modules, 
thus
\Name{Hilbert} $A$-modules $E$ satisfy $E=E\cdot A$ by the
Cohen factorization theorem (cf.\ \cite{BoDu}).
It follows that \Name{Hilbert} $A$-modules are also
$\Mult{A}$-modules and, thus, carry a compatible
vector space structure over $\C$.
There are examples of (degenerate) $A$-modules $E$ 
over a \Cast-algebra $A$ that do not carry a 
compatible vector space structure.

Thus on the algebraic level we need a definition of
a right $A$-module $E$ that $E$ is
a vector space over $\C$
with an algebra homomorphism $\rho\colon A\to L(E)$
that we write as $e\cdot a:=\rho(a)(e)$, i.e.\ the 
$A$-module
structure is
compatible with the vector addition scalar multiplication, 
i.e.\
$\forall (\lambda\in\C, a\in A, e,f\in E)\colon\quad 
\lambda(ea)=(\lambda e)a
=
e(\lambda a)\quad \textrm{and} \quad (e+f)\cdot a
=e\cdot a +f\cdot a$.
On non-degenerate complete normed $A$-modules 
$E$ the compatible
$\C$-vector space structure exists and is unique.
\end{REM}

\begin{REM}
Our \Name{Hilbert} $A$-bi-modules $\HM[E]{A}$ 
should not be mixed up
e.g.~with linking bi-modules (for ideals)
of $A$, because we consider mainly
$A$-bi-module structures 
that admit 
no non-trivial (opposite) sesqui\-linear form on $E$
which  is is 
$A$-linear with respect to the left action of  $A$ on $E$.
\end{REM}

\begin{EXA} 
$B:=\ell_\infty (\Z)$ , $\sigma\in \Aut{B}$
(forward) shift, $p\in B$ the projection with
$p(n)=1$ for $n\ge 1$, $p(n)=0$ for $n\leq 0$,
and $D:=\C\cdot p\subset B$. $D$ satisfies
($\alpha$) and ($\gamma$) but not ($\beta$).
$[D]_\sigma$ is the set of bounded sequences
$(f_n)$ with $\lim _{n\to -\infty} f_n=0$
and $\lim _{n\to +\infty} f_n$ exist in $\C$.

$c_0(\Z)$ is a $\sigma$-invariant ideal of
$[D]_\sigma =D_{-\infty,\infty}$,
because $p-\sigma(p)=\delta_{1,1}$.
$c_0\cap D_{1,2}=\C \cdot \delta_{1,1}$ and
$c_0\cap D=\{ 0\}$, 
because $\lim _{n\to +\infty} f_n\not=0$ for 
$f\in D_{-k,\infty}\setminus \{ 0\}$ and $k\in \N$.
$[D]_\sigma$ is not unital
and 
$0\to c_0\to [D]_\sigma \to \C \to 0$ is an
exact sequence defined by 
$f\in [D]_\sigma\mapsto \lim_n f_n\in \C$.
\end{EXA}

\begin{REM}\Label{R:uniform-lsc-maps}
Note that the ``uniform lower semi-continuity'' 
established in the proof of  
Lemma \ref{L:lsc-relation}(ii) also shows:\\
{\it Suppose that $X$ and $Y$ are metric spaces 
where $X$ is compact, and 
$R$ is a subset of $Y\times X$ with $p_1(R)=Y$.
Let $\lambda (y):=p_2(p_1^{-1}(y))\subset X$.
Then the restriction of $p_1\colon (y,x)\mapsto y$
to $R$ is an open map from $R$ onto $Y$,
if and only if,
for every $y\in Y$ and $\varepsilon >0$ there is 
$\delta>0$ with 
$\lambda(y)\subset U_\epsilon (\lambda (z))$
for all $z\in Y$ with  
$\textrm{dist} (z,y)<\delta$.}
(Here: $U_\varepsilon (W):=
\{ x\in X\fdg   
\textrm{dist} (x,W)<\varepsilon \}$.)
\end{REM}

\begin{REM}
$W$ is a $R_\pi$-invariant set 
if and only if $Y\setminus W$
is an $R_\pi^t$-invariant set.
\end{REM}

\begin{REM}

\noindent
(i)
$p_k\LOp{\Fock{E}}p_k$ is naturally isomorphic
to $\LOp{\bigoplus_{n=0}^k E^{(\otimes_A)n}}$.

\noindent
(ii)
The following relations hold (with $e_n\in E$):
\begin{eqnarray*}
(T_e)^*T_f ( a,f_1,f_2,\dots ) &=&
   ( \SP{e}{f}a,\SP{e}{f}f_1,\SP{e}{f}f_2,\dots ) , \\
T_f(T_e)^* ( a,e_1,e_2\otimes_A f_1,\dots ) 
%&=&
%( 0,f\SP{e}{e_1},f\otimes_A\SP{e}{e_2}f_1,\dots ) 
  &=& 
( 0,\theta_{fe}(e_1),\theta_{fe}(e_2)\otimes_A f_1,\dots ).
\end{eqnarray*}

\noindent
(ii)
Relations for images $S_e$ of $T_e$:
\begin{eqnarray*}
S^*_eS_f &=& \SP{e}{f} \\
S_fS^*_e &=& \theta_{fe^*} \\
S_ea = S_{ea} &,& aS_e=S_{ae}
\end{eqnarray*}
Relations require to
specify the natural images of 
$\K (E)$ and $A$.
\end{REM}

\medskip

\FRem{erase next?}
\begin{REM}
        The Kasparov trivialization theorem 
(cf.\ (iv) of Example\ \ref{Ex:H-module}) 
extends long known results of
  Dixmier and Douady in \cite{DixDoua}) 
        on the sub-triviality of
  continuous fields of Hilbert spaces. 
        In principle we need only
        the old results of \cite{DixDoua}.
\end{REM}

\FRem{next useful?}
\begin{REM}
The natural unital and strictly continuous
extension 
$$\Mult{\pi_J}\colon \Mult{A}\to \Mult{A/J}$$  
of the 
epi\-mor\-phism $A\to A/J\subset \Mult{A/J}$
has kernel $\Mult{A,J}$,
i.e., there is always a canonical embedding of 
$\Mult{A}/\Mult{A,J}$
in $\Mult{A/J}$ and $A\cap \Mult{A,J}=J$.

A non-commutative version of the
\Name{Tietze}
extension theorem was proved by Pedersen, it says that
$\Mult{\pi_J}\left(\Mult{A}\right)= \Mult{A/J}$ 
if $A$ contains
a strictly positive element, i.e.\ 
$\Mult{\pi_J}$ is also surjective if $A$ is a 
$\sigma$-unital \Cast-algebra
(cf.\ \cite[prop.\ 3.12.10]{Ped} for the separable
case. The  non-separable case immediately follows,
because in the non-separable
case for every $t\in \Mult{A/J}$ there
is a separable non-degenerate \Cast-subalgebra
$B$ of $A$ such that $t\in \Mult{\pi_J(B)}$.)

Every separable $A$ is $\sigma$-unital.
\end{REM}
%\medskip

\FRem{next useful/interesting?}
\begin{REM}\Label{R:ideals-of-cross-prod}
% not used for PROP
The proof of Proposition \ref{P:e-n-in-second} can be modified 
such that it shows also the following:

Let $E$ be a \Cast-algebra, $\sigma\in\Aut{E}$. 
Suppose that for
every $\sigma^{**}$-fixed \emph{closed} projection $p$ in 
the center of the
second conjugate $E^{**}$ of $E$ 
(i.e.~$1-p$ is a support-projection of an 
$\left(\sigma,\sigma^{-1}\right)$-invariant closed ideal of 
$E$)
there are mutually orthogonal central projections 
$p_n\le p$, $n\in\Z$
such that 
$\sigma^{**} \left( p_n \right) =p_{n+1}$ and
$p$ is the closure of $\sum_{n\in\Z} p_n$, i.e., for $a\in E$,
$ap=0$ if $ap_n=0$ for all $n\in\N$).

Then $\latI{E\times_\sigma\Z}$ 
is naturally lattice-isomorphic
to $\latI{E}^\sigma$, i.e.,\\ 
$\latO{\Prim{E\times_\sigma\Z}}$ is
nothing else than $\latO{\Prim{E}/\sim_\sigma}\cong
\latO{\Prim{E}}^\sigma$, 
where $\sim _\sigma $ denotes the equivalence relation
of the action of $\Z$ on $\Prim{E}$ induced by $\sigma$
and $\latO{\Prim{E}/\sim_\sigma}$ is the lattice of
open subsets of $\Prim{E}$ invariant under the $\Z$-action
induced by $\sigma$.
\end{REM}

\FRem{next useful/interesting?}
\begin{REM}\Label{L:rep-rp-faith}
Suppose that 
$M$ is a unital \Cast-algebra, 
$\alpha$ is the shift 
auto\-mor\-phism of $\ell_\infty ( \Z,M ) $.
If $\varrho$ is a *-representation of
$\ell_\infty ( \Z,M ) \rtimes_\alpha\Z$ such that 
the restriction $\varrho|\ell_\infty ( \Z,M ) $ is faithful,
then $\varrho$ is faithful.
%
%\begin{proof} 
Indeed: For $\theta\in\C$ with
$|\theta|=1$, let 
$v(\theta):= 
(\ldots, \theta^{-1},1, \theta, \theta^2, \ldots)$.
It is a unitary in the center of $B:=\ell_\infty (M)$ 
such that
$\theta \alpha \left( v(\theta) \right) = v(\theta)$, thus,
Lemma \ref{L:inner-dual-action} applies.
%\end{proof}
\end{REM}

\FRem{what is really needed from the following lemma?}
\begin{LEM}%\Label{L:a4.4}
Suppose that \Cast-algebras $D\subset B$ and 
$\sigma\in\Aut{B}$
satisfy $D\sigma^k(D)\subset\sigma^k(D)$ for $k\in\N$.
\begin{Aufl}
\item  
  $D_{1,n} \left( \sigma^n \right) ^k\left( D_{1,n} \right) 
  \subset  \left( \sigma^n \right) ^k \left( D_{1,n}\right)$
  holds for every $n,k\in\N$,
  and $[D]_\sigma$ is also the smallest 
        $\left(\sigma^n\right)$-invariant 
  \Cast-sub\-al\-ge\-bra 
        $\left[D_{1,n}\right]_{\sigma^n}$ of $B$ that
  contains $D_{1,n}$.
\item If $n\in\N$ and $J$ is a closed ideal of $D_{1,n}$ 
        with
  $J\sigma^k \left( D_{1,n} \right) \subset\sigma^k (J)$
  for $k\in\N$ then $J_{-\infty,\infty}$ equals 
        $[J]_\sigma$ and 
  is a $\sigma$-invariant closed ideal of $[D]_\sigma$.
\item If $I$ is a $\sigma$-invariant closed ideal of 
  $D_{-\infty,\infty}$
  then $J_n:=I\cap D_{1,n}$ satisfies 
        $J_n\sigma^k \left( D_{1,n} \right) 
  \subset\sigma^k \left( J_n \right) $ for $k\in\N$ and
  $\sigma \left( J_n \right) \subset J_{n+1}$.
\item Every $\sigma$-invariant closed ideal $I$ of 
  $D_{-\infty,\infty}$ is
  the closure of the union of 
        $\left( J_n \right) _{-\infty,\infty}$ where
  $J_n:=I\cap D_{1,n}$.
\end{Aufl}
\end{LEM}

\begin{proof}
\Ad{(i)} $D_{1,n}\sigma^{nk} \left( D_{1,n} \right) 
\subset
  \sigma^{nk} \left( D_{1,n} \right)$ follows from 
  Lemma \ref{L:D-sigma-k-E}.
  Thus, we can apply Lemma \ref{L:a4.3}(ii) 
        for $\sigma^n$ and 
  $D_{1,n}$ in place of $\sigma$ and $D$, i.e.\
  $\left[D_{1,n}\right]_{\sigma^n}= 
        \left( D_{1,n} \right) _{-\infty,\infty}
  =\overline{\bigcup  \left( D_{1,n} \right) _{-k,k}}$.
  $$D_{-nk,nk}\subset \left( D_{1,n} \right) _{-k,k}
        =D_{-nk-n+1,nk}
  \subset D_{-n(k+1),n(k+1)}\,,$$
  which implies 
  $\left[D_{1,n}\right]_{\sigma^n}=[D]_\sigma$.

\Ad{(ii)}
  $J\sigma^k \left( D_{1,n} \right) \subset\sigma^k(J)$ 
        implies
  $J\sigma^k(J)\subset J$, which shows 
  $[J]_\sigma=J_{-\infty,\infty}\subset[D]_\sigma$ 
  by Lemma \ref{L:a4.3}~(ii). 

  If $l-1\le n(m-1)$ then, with $k:=n(m-1)-(l-1)$,
  $$ \sigma^{n(m-1)}\left( D_{1,n} \right)\sigma^{l-1}(J)=
  \sigma^{l-1}\left(
        \left(J\sigma^{k}\left(D_{1,n}\right)\right)^*
        \right) 
  \subset\sigma^{l-1} \left( \sigma^{k}(J) \right) 
  \subset J_{-\infty,\infty}. $$
  Otherwise $n(m-1)<l-1$ and, with $k:=l-1-n(m-1)$,
  \begin{eqnarray*}
  \sigma^{n(m-1)} \left( D_{1,n}\sigma^{k}(J) \right) &=&
  \sigma^{n(m-1)} \left( D_{1,n}\sigma^{k}(D_{1,n}J) \right)\\ 
  \subset\sigma^{n(m-1)}
    \left(\sigma^{k}\left(D_{1,n}\right)\sigma^{k}(J)\right)
    &\subset&\sigma^{n(m-1)}\left(\sigma^{k}(J)\right) 
    \subset J_{-\infty,\infty},
  \end{eqnarray*}
  because $J=D_{1,n}J$ and 
  $D_{1,n}\sigma^k \left( D_{1,n} \right) \subset\sigma^k
   \left( D_{1,n} \right)$, by Lemma \ref{L:D-sigma-k-E}.

  Thus 
        $\left[D_{1,n}\right]_{\sigma^n}\cdot J_{-\infty,\infty}
        \subset J_{-\infty,\infty}$ and 
        $J_{-\infty,\infty}$ is a closed ideal of 
        $[D]_\sigma=\left[D_{1,n}\right]_{\sigma^n}$ by (i).

\Ad{(iii)} 
  By Lemma \ref{L:D-sigma-k-E},
  \begin{eqnarray*}
  J_n\sigma^k \left( D_{1,n} \right) &\subset&
   \left( I D_{1+k,n+k} \right) 
        \cap \left( D_{1,n}D_{1+k,n+k} \right) 
  \subset I\cap D_{1+k,n+k}\\
  &=&\sigma^k \left( I\cap D_{1,n} \right) 
  =\sigma^k \left( J_n \right) .
  \end{eqnarray*}
  $\sigma\left(J_n\right)\subset\sigma\left(I\right)\cap
  \sigma \left( D_{1,n} \right) \subset 
        I\cap D_{1,n+1}=J_{n+1}$.

\Ad{(iv)} 
  $I$ is the inductive limit of 
  $I\cap D_{-n,n}=\sigma^{-n-1} \left( J_{2n+1} \right) $ by
  Lemma \ref{L:a4.3}~(ii). 
\end{proof}

}  %end of TEMP

%%%%%%%%%%%%%%%%%%%%%%%%%%%%%%%%%%%%%%%%%


\begin{thebibliography}{99}
%
\providecommand{\bysame}{\leavevmode\hbox
to2.5em{\hrulefill}\thinspace}

\bibitem{Black} B.~Blackadar, {\em K-theory for
        Operator Algebras}, second edition,
        MSRI Publications \textbf{5}, 
        Cambridge Univ.\ Press [1998].
%
\bibitem{BoDu} F.~F.~Bonsall \& J.~Duncan,
  {\em Complete Normed Algebras}, 
  Springer, Berlin etc.\ [1973].
%
\bibitem{Brown} L.~G.~Brown,
        {\em Stable iso\-mor\-phism of hereditary 
        sub\-al\-ge\-bras of \Cast-algebras},
        Pacific J.\ Math.\ \textbf{71} [1977], 335--348.
%
%\bibitem{Brenken} B.~Brenken, 
%{\em \Cast-algebras of infinite graphs and Cuntz--Krieger
%algebras}, 
%Canad.\ Math.~Bull.\  \textbf{45} [2002], 321--336.
%
\bibitem{BrenkEndo} B.~Brenken,
{\em Endomorphism of type I von Neumann algebras with discrete center},
J. Operator Theory \textbf{51} [2004], 19--34
%
\bibitem{Cohen} P.~J.~Cohen, 
        {\em Factorization in group algebras}, 
Duke Math.~J.\ \textbf{26} [1959], 199--205.
%
\bibitem{Dix} J.~Dixmier, 
        {\em Les \Cast-alg\`ebres 
        et leurs repr\'esentations}, 
        Gauthier-Villars Paris [1969].
%
\bibitem{DixDoua} J.~Dixmier \& A.~Douady, 
        {\em Champs continus d'espaces hilbertiens et de
        \Cast-alg\`ebres}, 
        Bull.~ Soc.~Math.~France 
        \textbf{91} [1963], 227--284.
%
\bibitem{DS01} K.J.~Dykema \& D.~Shlyakhtenko,
  \emph{Exactness of Cuntz--Pimsner \Cast-algebras},
  Proceedings of the Edinburgh Math.\ Soc.\ \textbf{44} 
        [2001], 425--444.
%
\bibitem{Hausd} F.~Hausdorff,
  \emph{Mengenlehre}, Berlin, de Gruyter [1927].
%
\bibitem{HjelRor} J.~Hjelmborg \& M.~R{\o}rdam, 
{\em On stability of \Cast-algebra}, 
J.\ Funct.\ Analysis\ \textbf{155} [1998], 153--171.
%
\bibitem{HoffKeim} K.H.~Hoffmann \& K.~Keimel,
  \emph{A general character theory for 
        partially ordered sets and lattices},
  Mem.\ Amer.\ Math.\ Soc.\ \textbf{122} [1972].
%
\bibitem{JT91} K.K.~Jensen \& K. Thomsen,
  \emph{Elements of KK-theory}, 
  Boston [1991].
%
\bibitem{Kas} G.~G.~Kasparov, 
        {\em Hilbert \Cast-modules: 
        Theorems of Stinespring and Voiculescu}, 
        J.\ Operator Theory {\textbf 4} [1980], 133--150.
%
\bibitem{Kir00} E.~Kirchberg,
%\bysame , 
{\em Das nicht-kommutative Michael-Auswahl\-prinzip 
und die Klassifikation nicht-einfacher Algebren}, 
in {\em \Cast-Algebras}: 
Proceedings of the SFB-Workshop on \Cast-algebras,
M{\"u}nster, Germany, March 8-12, 
1999/J.~Cuntz,~S.~Echterhoff (ed.),
Berlin etc., Springer [2000],
pp.\ 92--141.
%
\bibitem{K.book} \bysame, %E.~Kirchberg,
  \emph{The classification of purely infinite 
        \Cast-algebras using
  Kasparov's theory}, in preparation.
%
\bibitem{DiniEK1} \bysame, %E.~Kirchberg,
  \emph{Dini functions on spectral spaces}, 
  SFB\-478-preprint, Heft 321, Uni\-ver\-sit{\"a}t  M{\"u}nster.
%
\bibitem{DiniEK2} \bysame, %E.~Kirchberg
        \emph{The range of generalized Gelfand transforms 
        on \Cast-algebras},
  to apper in J.~Operator Theory.
(SFB\-478-preprint, Heft 283, Uni\-ver\-sit{\"a}t  M{\"u}nster.)
%
\bibitem{DiniEK3} \bysame, %E.~Kirchberg,
  \emph{Dini spaces and Polish equivalence relations}, 
  in preparation.
%
\bibitem{KirRor2} E.~Kirchberg \& M.~R{\o}rdam, 
{\em Infinite non-simple C*--algebras:
absorbing the Cuntz algebra $\mathcal{O}_\infty$},
Advances in Math.~\textbf{167} [2002], 195--264.
%
\bibitem{KR03} \bysame\ \& \bysame,
%E.~Kirchberg \& M.~R{\o}rdam,
  \emph{Purely infinite \Cast-algebras: 
        ideal-preserving zero homotopies},
Geometric and Functional Analysis \textbf{15} [2005], 377-415. 
%
\bibitem{La95} E.C.~Lance,
  \emph{Hilbert \Cast-modules, a toolkit for 
        operator algebraists},
  London Mathematical Society 
        Lecture Notes \textbf{210} [1995].
%
\bibitem{Michael} E.~Michael, 
        {\em Continuous selection I}, 
        Ann.~of Math \textbf{63} [1956], 361--382. 
\FRem{on lsc cont. families??}
%
\bibitem{Ped} G.~K.~Pedersen, 
        {\em \Cast-algebras and their auto\-mor\-phism 
        groups}, Academic Press, London [1979]. 
%
\bibitem{Pi97} M.V.~Pimsner,
  \emph{A class of \Cast-algebras generalizing 
        both Cuntz-Krieger algebras
  and crossed products by $\Z$},
  Fields Inst.\ Commun.\ \textbf{12} [1997], 389--457.
%
\bibitem{Ror} M.~R{\o}rdam, 
        {\em A simple \Cast-algebra with a 
        finite and an infinite projection}, 
        Acta Mathematica \textbf{191} [2003], 109--142. 
%
\end{thebibliography}
\end{document}